\documentclass[12pt]{article}
\usepackage{color,amsmath,amsfonts,amssymb,epsfig,latexsym}
\usepackage{graphicx}
\usepackage[latin1]{inputenc}
\usepackage{amssymb}
\usepackage{bm}
\usepackage{color}
\usepackage{leftidx}

\usepackage{mathrsfs}
\renewcommand{\d}{{\rm d} }

\def\R{{\mathbb R}}
\def\N{{\mathbb N}}
\def\T{{\mathbb T}}
\def\Q{{\mathbb Q}}
\def\P{{\mathbb P}}
\def\J{{\mathbb J}}
\def\I{{\mathbb I}}
\def\D{{\mathbb D}}
\def\divg{\mathop{\rm div}\nolimits}

\def\bpsi{\boldsymbol\psi}

\def\bphi{\boldsymbol\varphi}

\def\bn{{\bf n}}
\def\bv{{\bf v}}
\def\bw{{\bf w}}
\def\bu{{\bf u}}
\def\bU{{\bf U}}
\def\bq{{\bf q}}
\def\bx{{\bf x}}
\def\by{{\bf y}}
\def\bX{{\bf X}}
\def\bY{{\bf Y}}
\def\ba{{\bf a}}
\def\bA{{\bf A}}
\def\bB{{\bf B}}

\def\by{{\bf y}}
\def\bn{{\bf n}}
\def\bN{{\bf N}}
\def\bpsi{\boldsymbol\psi}
\def\btau{\boldsymbol\tau}

\def\bomega{\boldsymbol\omega}
\def\bOmega{\boldsymbol\Omega}
\def\bgamma{\boldsymbol\gamma}

\def\bphi{\boldsymbol\varphi}
\newtheorem{teo}{Theorem}[section]
\newtheorem{definition}{Definition}[section]
\newenvironment{proof}[1][Proof]{\medskip\noindent\textit{#1. }\upshape}{$\hfill \;\blacksquare $\medskip}
\newtheorem{lemma}{Lemma}[section]

\newtheorem{remark}{Remark}[section]
\newtheorem{proposition}{Proposition}[section]

\numberwithin{equation}{section}
\begin{document}

\date{}
\title{A uniqueness result for 3D incompressible fluid-rigid body interaction problem}
\author{{\large Boris Muha\textsuperscript{1}\thanks{The research of B.M. leading to these results has been supported by Croatian Science Foundation under the project IP-2018-01-3706}, \v S\'arka Ne\v casov\'a\textsuperscript{2}\thanks{The research of \v S.N. leading to these results has received funding from
the Czech Sciences Foundation (GA\v CR),
19-04243S
 and RVO 67985840.}, Ana Rado\v sevi\'c\textsuperscript{2,3}\thanks{The research of A.R. leading to these results has been supported by Croatian Science Foundation under the project IP-2018-01-3706, and by
the Czech Sciences Foundation (GA\v CR),
19-04243S}}\\
{\small \textsuperscript{1} Department of Mathematics}\\
{\small Faculty of Science}\\
{\small University of Zagreb, Croatia}\\
{\small borism@math.hr}\\
{\small $^2$ Institute of Mathematics, Czech Academy of Sciences,}\\
{\small \v Zitn\'a 25, 115 67 Praha 1, Czech Republic }\\
{\small matus@math.cas.cz}\\
{\small \textsuperscript{3} Department of Mathematics}\\
{\small Faculty of Economics and Business}\\
{\small University of Zagreb, Croatia}\\
{\small aradosevic@efzg.hr}
}
\maketitle
\begin{abstract}
We study a 3D nonlinear moving boundary fluid-structure interaction problem describing the interaction of the fluid flow with a rigid body. The fluid flow is governed by 3D incompressible Navier-Stokes equations, while the motion of the rigid body is described by a system of ordinary differential equations called Euler equations for the rigid body. The equations are fully coupled via dynamical and kinematic coupling conditions. We consider two different kinds of kinematic coupling conditions: no-slip and slip.
In both cases we prove a generalization of the well-known weak-strong uniqueness result for the Navier-Stokes equations to the fluid-rigid body system. More precisely, we prove that weak solutions that additionally satisfy the Prodi-Serrin $\mathrm{L}^r-\mathrm{L}^s$ condition are unique in the class of Leray-Hopf weak solutions.
\end{abstract}
%
%
%
%
\section{Introduction}
The fluid-structure interaction (FSI) systems are multi-physics systems that include a fluid and solid component. They are everyday phenomena with a wide range of applications (see e.g. \cite{FSIforBIO,ParticlesInFlow,BorSun} and references within). Equations that arise from modeling such phenomena are typically nonlinear systems of partial differential equations with a moving boundary. The simplest model for the structure is a rigid body. The position of the rigid body at any given time moment is determined by two vectors describing the translation of the center of mass and the rotation around the center of the mass. Therefore the dynamics of the rigid body is described by a system of six ordinary differential equations (Euler equations) describing the conservation of linear and angular momentum. In this paper we consider the system where the rigid body moves in 3D container filled with an incompressible Newtonian fluid. The fluid flow is governed by the 3D incompressible Navier-Stokes  equations. The fluid domain is determined by the position of the rigid body, and the Navier-Stokes equations are coupled with a system of Euler ODE's via a dynamic and a kinematic coupling condition. The dynamic coupling condition represents the balance of forces acting on the rigid body. On the other hand, there are several possibilities for the kinematic coupling condition. The no-slip condition, which postulates equality of the velocities of the fluid and of the rigid body, on the rigid body boundary, is the most commonly used in the literature since it is the simplest to analyze and it is a physically reasonable condition in most situations. However, in some situations, e.g. in close to contact dynamics (see e.g. \cite{GVHil3,GVHil2}) or in the case of rough surfaces (see e.g. \cite{Bucur2010,Masmoudi2010,JagerMikelic2001,muha2016existence}), the Navier's slip coupling condition may be more appropriate since it allows for the discontinuity of the velocity in the tangential component on the rigid body boundary. In this paper we treat both cases.

The fluid-rigid body system has been extensively studied in the last twenty years and some aspects of the well-posedness theory are now well established. The existence of the unique local-in-time (or small data) solution is known in both two and three dimensions, and for both the slip (\cite{SarkaSlipRigidStrong,WangFluidSolid}) and the no-slip (\cite{CumTak,DE2,GGH13,MaityTucsnak18,Takahashi03}) kinematic coupling condition. On the other hand, it is known that a weak solution of Leray-Hopf type exists and is global in time or exists until the moment of contact between the boundary of the container and the rigid body for the slip (\cite{ChemetovSarka,GVHil3}) and the no-slip case (\cite{ConcaRigid00,DE,Feireisl03,GunzRigid00,Starovoitov02}). The question of the uniqueness of weak solution is still largely open. Even for the classical case of the 3D Navier-Stokes equations, the uniqueness of the Leray-Hopf weak solution is an outstanding open problem (see e.g. \cite{GaldiNS00}). However, there are classical results of weak-strong uniqueness type (see e.g. \cite{GaldiNS00,Serrin63,Tem}) which state that the strong solution (defined in an appropriate way) is unique in the larger class of weak solutions. For the Navier-Stokes equations the weak solutions that satisfy Serrin's conditions are regular (\cite{Serrin62}). In this paper our goal is to extend these classical weak-strong uniqueness type results to the case of a fluid-rigid body system under the condition that the rigid body does not touch the boundary of the container. Namely, in the case of contact it has been shown that weak solutions are not unique (\cite{FeireislRigid02,Starovoitov03}) because there are multiple ways of extending the solution after the contact.

There are not many uniqueness results in the context of weak solutions to FSI problems. The principal difficulty is that different solutions are defined on different domains so classical techniques do not apply. The uniqueness of weak solutions for a fluid-rigid body system in the 2D case was proven in \cite{GlassSueur} for the no-slip case and in \cite{bravin2018weak} for the slip case. To the best of our knowledge, the only results of weak-strong uniqueness type in the context of FSI are \cite{chemetov2017weak,Disser2016,GMN1}. In \cite{Disser2016} the authors studied a rigid body with a cavity filled with a fluid, while in \cite{chemetov2017weak} the requirement for strong solution is higher, namely the time derivative and the second space derivatives of the fluid velocity are in $L^2$. Our result is a generalization of these results, and also of the 2D uniqueness result. In \cite{GMN1} the authors studied a rigid body with a cavity filled with a compressible fluid. Based on a relative entropy inequality the  weak-strong uniqueness property is shown.

The paper consists of three sections and an appendix. The first section is the Introduction, in which we formulate the problem, give the literature review and state the main results. In the second section we give the proof of the main result, while the technical results are stated and proved in the Appendix. In the third section we extend the result to the case where the rigid body and the fluid are coupled via the Navier's slip condition.

\subsection{Formulation of the problem}

Let $\Omega\subset\R^3$ be a bounded $C^{2,1}$ domain which represents a container containing a fluid and a rigid body, and let $S_0\subset\Omega$ be a connected open set representing the rigid body at the initial time $t=0$ with the center of mass denoted by $\bq_0\in \Omega$. We assume that $S_0$ is also a domain of class $C^{2,1}$. The motion of the rigid body is fully described by two functions $\bq:[0,T]\to \R^3$ and $\Q:[0,T]\to SO(3)$, where $SO(3)$ is the 3D rotation group, representing the position of the center of mass and the rotation around the center of mass at the time moment $t$, respectively. More precisely,  the trajectories of all points of the
body are described by an orientation preserving isometry
\begin{equation}	\label{is}
{\bB}(t,\by)=\bq(t)+\Q(t)(\by-\bq(0)),\qquad \by\in S_{0},\; t\in [0,T],
\end{equation}
and at time $t$ the body occupies the set
\begin{equation}	\label{set}
S(t)=\{\bx\in \R^{3}:\ \,\bx=\bB(t,\by),\quad \by\in S_{0}\}
=\bB(t,S_{0}),\qquad t\in [0,T].
\end{equation}
The fluid domain at time $t$ is defined by $\Omega_F(t)=\Omega\setminus\overline{S(t)}$. Since the domain changes in time, we introduce the following notation:
\begin{equation}\label{FluidDomainNonCylindircal}
(0,T)\times\Omega_F(t)=\bigcup_{t\in (0,T)}\{t\}\times\Omega_F(t).
\end{equation}
The fluid flow is described by the incompressible Navier-Stokes equations:
\begin{equation}	\label{NS}
\left.
\begin{array}{l}
\varrho_F\big (\partial_{t}\bu+(\bu\cdot \nabla )\bu\big )
= \divg \left( {\T}(\bu,p)\right) , \\
\divg\bu = 0
\end{array}
\right\} \;\mathrm{in}\;(0,T)\times\Omega_{F}(t),
\end{equation}
where $\bu$ is the fluid velocity, $\varrho_F$ is the fluid density, ${\T}=-p\I+2\mu \,\D\bu$ is the fluid Cauchy stress tensor,  $\D\bu=\frac{1}{2}\left( \nabla\bu + \left( \nabla \bu \right) ^{T}\right)$ is the deformation-rate tensor, $p$ is the fluid pressure and $\mu >0$ is the fluid viscosity.

The Eulerian velocity of the rigid body is given by:
\begin{equation}\label{RigidVelocity}
\bu_S(t,\bx):=
\partial_{t}\bB(t,\bB^{-1}(t,\bx))=\ba(t)+\P(t)(\bx-\bq(t))\qquad \text{for all}\quad
\bx\in S(t),
\end{equation}
where $\ba(t)\in\R^3$ and $\P(t)$ are the translation and angular velocity satisfying
\begin{equation}	\label{compatible}
	\frac{d\bq}{dt}=a
	\quad\text{and}\quad
	\frac{d\Q}{dt}\Q^{T}=\P
	\quad\text{in } \left[ 0,T\right]. 
\end{equation}
The angular velocity $\P$ is a skew-symmetric matrix and therefore there
exists a vector $\bomega=\bomega(t)\in \R^{3}\mathbb{\ \ }$
such that
\begin{equation}	\label{om}
\P(t)\bx=\bomega(t)\times \bx,\qquad \forall \bx\in \R^{3}.
\end{equation}

The equations of motion for the rigid body follow from the Newton's second law and are given by
\begin{equation}	\label{RigidBody}
\left.
\begin{array}{l}
m\frac{d^{2}}{dt^{2}}\bq = \mathbf{f}_L, \\
\frac{d}{dt}(\J\bomega) = \mathbf{f}_T
\end{array}
\right\} \;\mathrm{in}\;(0,T),
\end{equation}
where $m$ is the mass of the rigid body, $f_L$ and $f_T$ are the total force and torque acting on the rigid body, respectively, and $\J$ is the inertial tensor defined as follows:
\begin{equation*}
\J=\int_{S(t)}\varrho_S(|\bx-\bq(t)|^{2}\I-(\bx-\bq(t))\otimes (\bx-\bq(t)))\,\d\bx,
\end{equation*}
where $\varrho_S$ denotes the density of the rigid body.

\subsubsection{The coupling conditions}
The fluid and the rigid body are coupled via dynamic and kinematic coupling condition. The dynamic boundary condition is just the balance of forces and torques:
\begin{equation}\label{Dynamic}
\mathbf{f}_L=-\int_{\partial S(t)}{\T}(\bu,p) \bn\,\d\bgamma(\bx),\quad
\mathbf{f}_T=-\int_{\partial S(t)}(\bx-\bq(t))\times {\T}(\bu,p)\bn\,\d\bgamma(\bx),
\end{equation}
where $\bn=\bn(t,\bx)$ is the unit \textit{interior} normal on $\partial S(t)$ at point $\bx\in\partial S(t)$.
For the kinematic condition we will consider two different possibilities. The first one is the no-slip condition which says that the fluid and the rigid body velocities are equal at the rigid body boundary:
\begin{equation}\label{NoSlip}
\bu(t,\bx)=\bu_S(t,\bx),\quad \bx\in\partial S(t),\; t\in (0,T).
\end{equation}
The other possibility is the Navier's slip boundary condition which allows for the discontinuity of the tangential component of the velocity along the interface:
\begin{equation}\label{slip}
\left .
\begin{array}{c}
\big (\bu(t,\bx)-\bu_S(t,\bx)\big )\cdot\bn (t,\bx)=0
\\
\beta \big (\bu_S(t,\bx)-\bu(t,\bx)\big )\cdot\btau (t,\bx)=\T(\bu(t,\bx))\bn(t,\bx)\cdot\btau(t,\bx)
\end{array}
\right \}\quad \bx\in\partial S(t),\; t\in (0,T),
\end{equation}
where $\beta>0$ is the friction coefficient at $\partial S(t)$, and $\btau$ is a unit tangent on $\partial S(t)$. The system is complemented with the no-slip boundary condition $\bu=0$ on $\partial\Omega$ and the initial conditions. For simplicity of notation, we assume $\varrho_F=\mu=m=1$ and $\varrho_S$ is constant.

To summarize, we consider the following fluid-rigid body interaction problem:
\\
find $(\bu,p,\bq,\bomega)$ such that
\begin{equation}\label{FSINoslip}
\begin{array}{l}
\left.
\begin{array}{l}
\partial_{t}\bu+(\bu\cdot \nabla )\bu
= \divg \left( {\T}(\bu,p)\right) , \\
\divg\bu = 0
\end{array}
\right\} \;\mathrm{in}\;\bigcup_{t\in(0,T)} \{t\}\times\Omega_{F}(t),
\\
\left.
\begin{array}{l}
\frac{d^{2}}{dt^{2}}\bq = -\int_{\partial S(t)}{\T}(\bu,p) \bn\,\d\bgamma(\bx), \\
\frac{d}{dt}(\J\bomega) = -\int_{\partial S(t)}(\bx-\bq(t))\times {\T}(\bu,p)\bn\,\d\bgamma(\bx)
\end{array}
\right\} \;\mathrm{in}\;(0,T),
\\
\bu =\bq'+\bomega\times (\bx-\bq),\quad \mathrm{on}\;\bigcup_{t\in(0,T)} \{t\}\times\partial S(t),
\\
\bu = 0 \quad \mathrm{on}\, \partial \Omega,
\\
\bu(0,.)=\bu_{0}\qquad \mathrm{in}\;\Omega,\quad \bq(0)=\bq_{0},\quad \bq^{\prime }(0)=\ba_{0},\quad
\bomega(0)=\bomega_{0}.
\end{array}
\end{equation}
The version of problem \eqref{FSINoslip} where the no-slip condition \eqref{NoSlip} is replaced by the slip condition \eqref{slip} will be called problem \eqref{FSINoslip}$_{\rm slip}$.

In order to state the main result of this paper,  we need to define the notion of weak solutions to the system \eqref{FSINoslip}. First we define a function space:
\begin{equation}\label{FluidFS}
V(t)=\{\bv\in H^1_0(\Omega):\divg\bv=0\;{\rm in}\; \Omega,\; \D\bv=0\;{\rm in}\; S(t)\}.
\end{equation}
\begin{remark}
	The condition $\D\bv=0$ in $S(t)$ is equivalent to the condition that $\bv(t)$ is rigid on $S(t)$, i.e. there exist $\ba(t),\bomega(t)\in\R^3$ such that $\bv(t,\bx)=\ba(t)+\bomega(t)\times(\bx-\bq(t))$, $\bx\in S(t)$.
\end{remark}
\begin{definition}\label{definition}
	The couple $\left(\bu, \bB\right)$
	is a weak solution to the system \eqref{FSINoslip} if the following conditions
	are satisfied:
	
	1) The function $\bB(t,\cdot ):\R^{3}\rightarrow \R^{3}$ \ is an orientation preserving isometry given by the formula \eqref{is}, which defines a time-dependent set $S(t)=\bB(t,S_0)$.
	The isometry $\bB$ is compatible with $\bu=\bu_S$ on $S(t)$ in the following sense: the rigid part of the velocity $\bu$, denoted by $\bu_S$, satisfies condition \eqref{RigidVelocity}, and 
	$\bq,\; \Q$ are absolutely continuous on $\left[ 0,T\right]$ and satisfy \eqref{compatible} and \eqref{om}.
	
	2) The function $\bu\in L^{2}(0,T;V(t))\cap L^{\infty
	}(0,T;L^2(\Omega ))$ satisfies the integral equality
	\begin{multline} \label{weak}
	\int_{0}^{T}\int_{\Omega}\{
	\bu \cdot \partial_{t}\bpsi
	+ (\bu\otimes \bu) :\D\bpsi
	- 2\,\D\bu:\D\bpsi\,\}\, \d\bx\d t
	- \int_{\Omega }\bu(T)\bpsi(T)\,\d\bx
	= - \int_{\Omega }\bu_{0}\bpsi(0)\,\d\bx,
	\end{multline}
	which holds for any test function $\boldsymbol{\psi }\in H^1(0,T;V(t))$.
	
	3) The  energy inequality
	\begin{equation*}	\label{EnergyInequality}
		\frac{1}{2}\|\bu(t)\|_{L^2(\Omega)}^2
		+ 2\int_{0}^{t}
		\int_{\Omega_{F}(\tau)}\,|\D \bu|^{2}
		\,\d\bx\,\d\tau
		\leq
		\frac{1}{2}\|\bu_0\|_{L^2(\Omega)}^2.
	\end{equation*}
	holds for almost every $t\in(0,T)$.
\end{definition}
\begin{remark}
		By using the standard cut-off argument (cf. \cite{GaldiNS00}, Lemma 2.1) it can be shown that the second condition is equivalent to the following statements:
		\begin{enumerate}
		\item The function $\bu\in L^{2}(0,T;V(t))\cap L^{\infty
		}(0,T;L^2(\Omega ))$ satisfies the integral equality
		\begin{equation*}
		\int_{0}^{T}\int_{ \Omega}\{
		\bu \cdot \partial_{t}\bpsi
		+ (\bu\otimes \bu) :\D\bpsi
		- 2\,\D\bu:\D\bpsi\,\}\, \d\bx\d t
		= - \int_{\Omega }\bu_{0}\bpsi(0)\,\d\bx,
		\end{equation*}
		which holds for any test function $\boldsymbol{\psi }\in H^1(0,T;V(t))$ satisfying $\bpsi(T,\cdot)=0$.
		
		\item
		The function $\bu\in L^{2}(0,T;V(t))\cap L^{\infty
		}(0,T;L^2(\Omega ))$ satisfies the integral equality
		\begin{multline} \label{weak2}
		\int_{0}^{t}\int_{\Omega}\{
		\bu \cdot \partial_{t}\bpsi
		+ (\bu\otimes \bu) :\D\bpsi
		- 2\,\D\bu:\D\bpsi\,\}\, \d\bx\d t\\
		- \int_{\Omega }\bu(t)\bpsi(t)\,\d\bx
		= - \int_{\Omega }\bu_{0}\bpsi(0)\,\d\bx,
		\end{multline}
		which holds for any test function $\boldsymbol{\psi }\in H^1(0,T;V(t))$ and $t\in(0,T)$.
		\end{enumerate}
\end{remark}
	
\begin{remark}
		The existence of weak solutions in the sense of Definition \ref{definition} has been studied in \cite{ConcaRigid00,DE,Feireisl03,GunzRigid00,Starovoitov02}. Even though the concept of weak solution in each of these works is defined in a slightly different context, for the case of the incompressible Navier-Stokes equations with constant density all these definitions are equivalent. In \cite{ConcaRigid00,GunzRigid00} the authors use a coordinate system which moves with the rigid body. This formulation is equivalent to our formulation (Eulerian) by a suitable change of variables. Namely, since the displacement is rigid, the change of variables is regular enough ($W^{1,\infty}$ in time and smooth in space) and therefore it yields an equivalent weak formulation. Concerning the regularity of the test functions, they can be taken to be smooth and then the weak formulation holds for all test functions used in Definition \ref{definition} by a standard density argument. In \cite{DE,Feireisl03} the position of the rigid body is tracked via conservation of mass for the global density (which can be then generalized to the compressible or variable density case). We defined the current position of the rigid body explicitly via the isometry $\bB$. However, in the incompressible, constant density case these approaches are equivalent (see e.g. Lemma A.7. in \cite{ChemetovSarka}). The $2D$ case is treated in \cite{Starovoitov02}.
\end{remark}

The weak formulation of \eqref{FSINoslip}$_{\rm slip}$ is defined in Section \ref{Sec:slip}. Now we can state the main result of this paper.

\begin{teo}\label{WeakStrong}
	Let $(\bu_{1},\bB_{1})$ and $(\bu_{2},\bB_{2})$ be two weak solutions corresponding to the same data.
	Assume that $d(S_i(t),\partial\Omega)>\delta_i$, $i=1,2$, for some constants $\delta_i>0$.
	If $\bu_{2}$ satisfies the following condition:
	\begin{equation}\label{LrLs}
	\bu_{2}\in L^r(0,T;L^s(\Omega))\quad
	\text{ for some } s,r \text{ such that }\quad
	\frac{3}{s}+\frac{2}{r}=1,\, s\in(3,+\infty)
	\end{equation}
	then
	\begin{equation*}
	(\bu_{1},\bB_{1}) = (\bu_{2},\bB_{2}).
	\end{equation*}
\end{teo}

\begin{remark}
	The condition \eqref{LrLs} is called a Prodi-Serrin condition.
\end{remark}

An analogous Theorem is also proven for \eqref{FSINoslip}$_{\rm slip}$ in Section \ref{Sec:slip}. These results are generalizations of the uniqueness results from \cite{bravin2018weak,chemetov2017weak,GlassSueur}. Namely, in \cite{chemetov2017weak} an analogous result is proven, but with a higher regularity assumption on $\bu_2$ ($\bu_2\in H^1(0,T;L^2(\Omega)\cap L^2(0,T;H^2(\Omega))$). In \cite{GlassSueur} and \cite{bravin2018weak} the 2D problem was studied for the no-slip and slip case, respectively. The uniqueness result for the 2D case follows from Theorem \ref{WeakStrong} by interpolation in the same way as in the classical case of the Navier-Stokes equations.

The basic strategy of the proof of Theorem \ref{WeakStrong} is analogous to the Navier-Stokes (see e.g. \cite{GaldiNS00,Serrin63}). However, there are considerable technical difficulties connected to the fact that the domain of the fluid is not known apriori and therefore we have to compare solutions which are apriori given on different domains. Another difficulty connected to the moving domain is the construction of an appropriate regularization operator (in the time variable). Namely, since the domain is changing, we cannot use the standard convolution operator. Finally, due to the change of variables we have to work with the weak formulation that includes the pressure variable. It seems that the existence of pressure connected to a weak solution of a fluid-rigid body system is missing from the literature. Even though for the proof of the weak-strong uniqueness result it is enough to prove the local-in-time existence of regular pressure (Proposition \ref{pressure_regularity}), we included the existence result for the pressure associated to the weak formulation, Theorem \ref{pressure}, because we believe it might be of independent interest. Here we emphasize that the local-in-time regularity is only needed for the estimate of the additional pressure term which is present due to the moving boundary. This term is not present in the standard weak-strong uniqueness proof for the Navier-Stokes equations. The main technical tool is a non-local change of variables (see e.g. \cite{GGH13,Takahashi03}) based on the idea from \cite{inoue1977existence} which can be used to map the problem to a fixed domain and to construct an appropriate regularization operator. In the slip case (Theorem~\ref{WeakStrongslip}) we do not prove the existence of the pressure associated with a weak solution, but proceed directly to prove the local-in-time existence of a regular pressure (Proposition~\ref{pressure_regularity_slip}).

\section{Proof of the main theorem}
\label{Sec:MainProof}

In this section we give the main steps of the proof of Theorem \ref{WeakStrong}, while proofs of some technical results are relegated to the Appendix. Let $(\bu_i,\bB_i)$, $i=1,2$, be weak solutions satisfying the assumptions of Theorem \ref{WeakStrong}, and let $p_i$, $\ba_i$, $\bomega_i$ be the pressure, the rigid body translation velocity and the rigid body angular velocity connected to the solutions $(\bu_i,\bB_i)$. Then, by \eqref{is} the isometry $\bB_{i}$, which defines the domain of the rigid body $S_i(t)=\bB_{i}(t,S_0)$, is given by
$$
\bB_{i}(t,\by) = \bq_{i}(t) + \Q_{i}(t)(\by-\bq_{0}),
$$
where $\bq_{i}'(t)=\ba_{i}(t)$ and $\Q_{i}'(t)\Q_{i}^T(t)=\P_{i}(t)$
for the skew-symmetric $\P_{i}(t)$ associated with $\bomega_{i}(t)$ (see \eqref{om}):
$$
\P_{i}(t)\bx = \bomega_{i}(t)\times\bx,\quad\forall\bx\in\R^3.
$$
We denote the fluid domain by $\Omega_{F}^{i}(t)=\Omega\setminus\overline{S_{i}(t)}$
and the function spaces
\begin{equation*}
V_i(t)=\{\bv\in H^1_0(\Omega):\divg\bv=0,\; \D\bv=0\;{\rm in}\; S_i(t)\}.
\end{equation*}

\begin{figure}
	\label{fig:change1}
	\centering
	\includegraphics{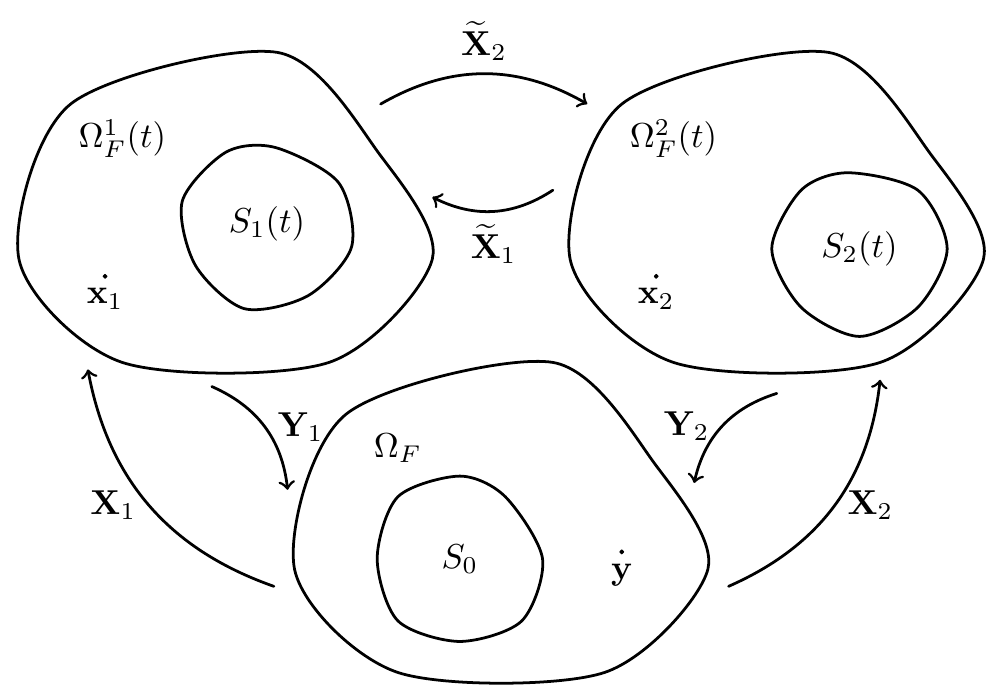}
	\caption{Change of coordinates}
\end{figure}

In order to compare two solutions, we need to transfer them to the same domain (see Figure 1). We follow the strategy from \cite{chemetov2017weak,GlassSueur} and use a local version of vector change of variables introduced for the study of the Navier-Stokes equations in a non-cylindrical domain by Innoue and Wakimoto \cite{inoue1977existence}. The construction of the local change of variables is now standard in the study of fluid-rigid body problems (see e.g. \cite{GGH13,Takahashi03}) so we omit it here and recall the basic facts in the Appendix, Section \ref{Sec:LT}.

Let $\bX_i$, $i=1,2$ be the change of coordinates associated to the solution $(\bu_i,p_i,\ba_i,\bomega_i)$ in the way described in the Appendix \ref{Sec:LT}, and let us denote the corresponding inverse transformation by
$$
\bY_{i}(t,\cdot )=\bX_{i}(t,\cdot )^{-1},\qquad i=1,2.
$$
In the neighborhood of $S_i(t),\,i=1,2 $, the transformations are rigid:
\begin{equation*}
\bX_{i}(t,\by)=\bq_{i}(t)+\Q_{i}(t)(\by-\bq_{0}),\qquad i=1,2
\end{equation*}
\begin{equation*}
\bY_{i}(t,\bx_{i})=\bq_{0}+{\Q}_{i}^{T}(t)(\bx_{i}-\bq_{i}(t)),\qquad i=1,2
\end{equation*}
We define the transformations $\widetilde{\bX}_1$ and $\widetilde{\bX}_2$ in the following way:
\begin{equation*}
\widetilde{\bX}_{1}(t,\bx_{2})=\bX_{1}(t,\bY_{2}(t,\bx_{2})),
\end{equation*}
$$
\widetilde{\bX}_{2}(t,\cdot )=\widetilde{\bX}_{1}(t,\cdot
)^{-1}
$$

In the neighborhood of $S_{i}(t)$ the transformations $\widetilde{\bX}_{i}$ are also rigid and are therefore given by the following expressions:
\begin{equation*}
\begin{split}
\widetilde{\bX}_{1}(t,\bx_{2})&=\bq_{1}(t)+\Q^{T}(t)(\bx_{2}-\bq_{2}(t))\quad
\text{in the neighborhood of }S_{2}(t),
\\
\widetilde{\bX}_{2}(t,\bx_{1})&=\bq_{2}(t)+\Q(t)(\bx_{1}-\bq_{1}(t))\quad
\text{in the neighborhood of }S_{1}(t),
\end{split}
\label{RigidTransf}
\end{equation*}
where ${\Q}={\Q}_{2}{\Q}_{1}^{T}$.	

Finally, we define the transformed second solution $(\bu_{2},p_{2},\ba_{2},\bomega_{2})$:
\begin{equation}
\left\{
\begin{array}{l}
\bU_2(t,\bx_1)=\nabla\widetilde{\bX}_1(t,\widetilde{\bX}_2(t,\bx_1))\bu_2(t,\widetilde{\bX}_2(t,\bx_1)),
\\
P_2(t,\bx_1)=p_2(t,\widetilde{\bX}_2(t,\bx_1)),
\\
\bA_{2}(t)=\Q^{T}(t)\ba_{2}(t),
\\
\bOmega_{2}(t)=\Q^{T}(t)\bomega_{2}(t),
\\
\mathcal{T}({\bU_2}(t,\by),P_2(t,\by))=\Q^{T}(t)
\T(\Q(t)\bU_2(t,\by),P_2(t,\by))\Q(t).
\end{array}
\right.  \label{eq:Transformed}
\end{equation}
It is easy to see that this change of variables is a volume preserving diffeomorphism, hence the transformed velocity satisfies the divergence-free condition (details can be found in Appendix \ref{Sec:LT} or \cite{inoue1977existence}, Proposition 2.4.).

\subsection{Weak formulation of the transformed solution}
\label{Sec:TransformedWeak}

The first step in the proof is to derive the weak formulation satisfied by the solution $(\bU_2,P_2,\bA_2,\bOmega_2)$.

Since the transformed solution will depend on $P_2$ because of \eqref{NSTransformed}, first we prove the existence of the pressure attached to a weak solution defined in Definition \ref{definition}. Such result is standard in the theory of incompressible Navier-Stokes equations but is, to the best of our knowledge, missing from the literature on fluid-rigid body systems.

\begin{teo}\label{pressure}
	Let $(\bu_2,\bB_2)$ be a weak solution to the system \eqref{FSINoslip}. Then, there exist functions {$p_{0}^2\in L^{\infty}(0,T;L^2(\Omega_{F}^2(t)))$ and $p_1^2\in L^{\frac{3}{4}}(0,T;L^2(\Omega_{F}^2(t)))$} such that for all $\bphi\in H^1_0(0,T;H^1_0(\Omega))$ satisfying $\D\bphi=0$ in $S_2(t)$, the following equality holds:
	\begin{equation} \label{weak3}
	\begin{split}
	\int_{0}^{T}\int_{\Omega}\bu_2\cdot\partial_t\bphi
	\,\d\bx \,\d\tau
	&+
	\int_{0}^{T}\int_{\Omega_F^2(\tau)} \{
	(\bu_2\otimes \bu_2) :\nabla\bphi
	-\,2\,\D\bu_2:\D\bphi \}
	\,\d\bx \,\d\tau
	\\
	&= \int_{0}^{T}\int_{\Omega_F^2(\tau)} \big(p_{0}^2\,\divg\partial_{t}\bphi - p_{1}^2\,\divg\bphi\big) \,\d\bx \,\d\tau.
	\end{split}
	\end{equation}
\end{teo}

\begin{remark}
	Below we will write $p_2=\partial_{t}p_0^2+p_1^2$, where $\partial_{t}p_0^2$ is the distributional derivative (with respect to t) of $p_0^2$.
	As in the Navier-Stokes case (see \cite{simon1999existence}), since $\Omega_{F}^2(t)$ is a Lipschitz domain we have that $p_2\in W^{-1,\infty}(0,T;L^2(\Omega_{F}^2(t)))$.
\end{remark}

We postpone the proof of Theorem \ref{pressure} to Section \ref{Sec:Pressure} because in the proof we will use a similar construction as the one used in the next proposition where we transform the weak formulation \eqref{weak3} to the domain $\Omega^1_F(t)$  via the transformation $\widetilde{\bX}_2$. We emphasize that the proof of Theorem \ref{pressure} does not use Proposition \ref{WeakU2} and the proof is postponed just for presentational purposes.
By \eqref{eq:Transformed} we have:
\begin{displaymath}
\bU_2(t,\bx_1) = \nabla\widetilde{\bX}_1(t,\widetilde{\bX}_2(t,\bx_1))\,\bu_2(t,\widetilde{\bX}_2(t,\bx_1)),
\end{displaymath}
i.e.
\begin{displaymath}
\bu_2(t,\bx_2) = \nabla\widetilde{\bX}_2(t,\widetilde{\bX}_1(t,\bx_2))\,\bU_2(t,\widetilde{\bX}_1(t,\bx_2)).
\end{displaymath}
Since $\widetilde{\bX}_1\in W^{1,\infty}(0,T; C^{\infty}(\Omega))$, the transformed velocity $\bU_2$ belongs to $L^2(0,T;V_1(t))\cap L^{\infty}(0,T;L^2(\Omega))$ and satisfies the Prodi-Serrin condition.

\begin{proposition}\label{WeakU2}
	Let
	\begin{equation*}
	\left\langle \widetilde{F}(t), \bpsi\right\rangle = \left\langle (\mathcal{L}-\Delta)\bU_{2}
	+\mathcal{M}\bU_{2}
	+\tilde{\mathcal{N}}\bU_{2}
	, \bpsi\right\rangle,
	\quad \forall\bpsi\in H^{1}(\Omega_{F}^{1}(t)),
	\end{equation*}
	and
	\begin{equation*}
	\widetilde{\bomega}\times \bx = {\Q}^{T}{\Q}^{\prime }\bx,
	\end{equation*}
	where $\mathcal{L}$, $\mathcal{M}$, $\tilde{\mathcal{N}}$
	are linear operators corresponding to the terms defined by \eqref{transformed_laplace}, \eqref{transformed_time} and \eqref{transformed_convective}, respectively.
	Then, the transformed solution
	$(\bU_2,\bA_2, \bOmega_2)$ 	
	satisfies the equality
	\begin{equation} \label{transformed_weak}
	\begin{split}
	&\int_{0}^{T}\int_{\Omega}\bU_2\cdot\partial_t\boldsymbol{\psi }\, \d\bx_1\,\d t
	+\int_{0}^{T}\int_{\Omega_F^1(t)}\Big((\bu_1\otimes \bU_2) :\nabla\boldsymbol{\psi }^T
	-\,2\,\D\bU_2:\D\boldsymbol{\psi }\,\Big)\, \d\bx_1\, \d t
	\\
	&\qquad -\int_{0}^{T}\int_{\Omega _{F}^{1}(t)}(\bU_{2}-\bu_{1})\cdot \nabla \bU_{2}\cdot \boldsymbol{\psi }\ \,\d\bx_{1}\,\d t
	\\
	& = \int_{0}^{T} \langle \widetilde{F}(t),\boldsymbol{\psi }(t)\rangle\, \d t
	+ \left\langle \widetilde{\mathcal{G}}P_2,\bpsi\right\rangle
	+ \int_{0}^{T} \int_{S_1(t)} \widetilde{\bomega}\times \bU_{2}\cdot \bpsi-\bu_{1}\times\bU_{2}\cdot\bpsi_{\omega}\,\d\bx_1\,\d t
	\end{split}
	\end{equation}
	for any test function $\bpsi\in H_0^1(0,T;V_1(t))$, i.e. $\bpsi$ is rigid on $S_1(t)$:
	\begin{equation*}
	\boldsymbol{\psi }(t,\mathbf{x}_1)=\boldsymbol{\psi }_{h}(t)+\boldsymbol{\psi }
	_{\omega }\times (\mathbf{x}_1-\mathbf{q}_{1}(t))\qquad \text{for }\mathbf{x}_1
	\in S_{1}(t).
	\end{equation*}
Here $\widetilde{\mathcal{G}}P_2$ is a bounded linear functional on $H_0^1(0,T;V_1(t))$ defined by \eqref{weak_pressure} corresponding to the pressure.
\end{proposition}

\begin{remark}
		In the definition of the operator $\widetilde{F}$, ie  $\mathcal{L}$, $\mathcal{M}$, $\tilde{\mathcal{N}}$ we use transformations $\widetilde{\bX}_1$ and $\widetilde{\bX}_2$ instead of $\bY$ and $\bX$, respectively.	
\end{remark}

\begin{proof}
	Let $\bpsi$ be a test function defined on the domain $(0,T)\times\Omega_F^1(t)$, i.e. $\bpsi\in H^1(0,T;V_1(t))$. Then, we can take the test function $\bphi$ in \eqref{weak3} to be defined in the following way:
	\begin{equation*}
	\bphi(t,\bx_2) = \nabla\widetilde{\bX}_1(t,\bx_2)^T\bpsi(t,\widetilde{\bX}_1(t,\bx_2)).
	\end{equation*}
	Since $\widetilde{\bX}_1\in W^{1,\infty}(0,T; C^{\infty}(\Omega))$, $\bphi$ belongs to $H^1(0,T;H^1_0(\Omega)$. Moreover, from the definition is immediate that $\D(\varphi)=0$ in $S_2(t)$ and therefore can be used as a test function in \eqref{weak3}. Here we note that in general $\divg \bphi\neq 0$ because we used $\nabla\widetilde{\bX}_1^T$  instead of $\nabla\widetilde{\bX}_1$ in definition of $\bphi$.
	
	We now compute the transformation of all terms in the weak formulation {\eqref{weak3}} by using properties of the local change of variables. Here we present the main steps, while details are given in Appendix, Section \ref{Sec:WFDEtails}.
	
	\noindent
	\textbf{The fluid time-derivative term}.
	\begin{align}
	&\int_{\Omega_F^2(\tau)}\bu_2\cdot\partial_t\bphi
	= \int_{\Omega_F^1(\tau)}\big (\bU_2\cdot \partial_t\bpsi
	-\mathcal{M}\bU_2\cdot\bpsi
	+\nabla(\bU_2\cdot\bpsi)\cdot\partial_t\widetilde{\bX}_1\big )
	\notag\\
	&\quad
	=\int_{\Omega_F^1(\tau)} \big (\bU_2\cdot \partial_t\bpsi
	-\mathcal{M}\bU_2\cdot\bpsi
	+(\bu_1-\bU_2)\cdot\nabla\bU_2\cdot\bpsi + (\bu_1-\bU_2)\otimes\bU_2:\nabla\bpsi^T\big ),
	\label{weak_timeDerivative}
	\end{align}
	where $\mathcal{M}$ is the operator defined in \eqref{OpN}. In the last equality we used
	$\partial_t\widetilde{\bX}_1 = \bu_1-\bU_2$ on $\partial S_1(t)$ and $\divg(\partial_t\widetilde{\bX}_1)=0$ (by construction of the transformation $\widetilde{\bX}_2$, see \cite{GGH13}).	
	
	\noindent
	\textbf{Convective term}.	
	\begin{equation}\label{weak_convective}
	\bu_2\otimes\bu_2:\nabla\bphi
	= \bU_2\otimes\bU_2:\nabla\bpsi^T
	- \widetilde{\mathcal{N}}\bU_2\cdot\bpsi,
	\end{equation}
	where $\widetilde{\mathcal{N}}$ is the operator defined in \eqref{OpC}. Combining \eqref{weak_timeDerivative} and \eqref{weak_convective} we get the following expression for the acceleration term:
	\begin{align} \label{Wu2t6}
	\int_{\Omega_F^2(\tau)}& \bu_2\cdot\partial_t\bphi
	+ \bu_2\otimes\bu_2:\nabla\bphi \, \d\bx_2
	\notag\\
	&= \int_{\Omega_F^1(\tau)} \bU_2\cdot \partial_t\bpsi - \mathcal{M}\bU_2\cdot\bpsi
	+ (\bu_1-\bU_2)\cdot\nabla\bU_2\cdot\bpsi
	\notag\\
	&\qquad\qquad
	+ (\bu_1-\bU_2)\otimes\bU_2:\nabla\bpsi^T
	+ \bU_2\otimes\bU_2:\nabla\bpsi^T - \widetilde{\mathcal{N}}\bU_2\cdot\bpsi
	\, \d\bx_1
	\notag\\
	&= \int_{\Omega_F^1(\tau)}\bU_2\cdot \partial_t\bpsi
	+ (\bu_1-\bU_2)\cdot\nabla\bU_2\cdot\bpsi
	+ \bu_1\otimes\bU_2:\nabla\bpsi^T\\
	&\qquad\qquad
	-(\mathcal{M}\bU_2+\widetilde{\mathcal{N}}\bU_2)\cdot\bpsi
	\, \d\bx_1.
	\notag
	\end{align}	
	
	\noindent
	\textbf{Pressure term}.
	For the pressure term we define
	\begin{equation}	\label{weak_pressure}
	\begin{split}
	\left\langle \widetilde{\mathcal{G}}P_2,\bpsi\right\rangle
	&:= -\int_{0}^{T}\int_{\Omega^2_F(\tau)} p_2\cdot\divg\bphi\,
	{
		= \int_{0}^{T}\int_{\Omega^2_F(\tau)} \left( p_0^2\cdot\divg\partial_t\bphi-p_1^2\cdot\divg\bphi \right)
	}
	\end{split}
	\end{equation}
	Since $\widetilde{\bX}_1, \widetilde{\bX}_2\in W^{1,\infty}(0,T;C^{\infty}(\overline{\Omega}))$, and
	$$
	\bphi \mapsto \int_{0}^{T}\int_{\Omega^2_F(\tau)} p_2\cdot\divg\bphi
	$$
	is a bounded linear functional on $H_0^1(0,T;H^1_0(\Omega))$, it is easy to see that $\widetilde{\mathcal{G}}P_2$ is a bounded linear functional on $H_0^1(0,T;V_1(t))$.
	
	\noindent
	\textbf{Diffusive term}.
	\begin{equation}\label{DualL}
	\int_{\Omega^2_F(\tau)} 2\,\D\bu_2:\D\bphi
	=\left\langle \mathcal{L}\bU_2, \bpsi\right\rangle,
	\end{equation}
	where
	\begin{multline} \label{operatorL}
	\left\langle \mathcal{L}\bU_2, \bpsi\right\rangle
	= \int_{\Omega_F(\tau)} \big(\sum_{ijk}
	(g^{jk}\partial_j\bU_{i}^2\partial_k\bpsi_i
	+ g^{jk}\partial_k\bU_i^2\partial_i\bpsi_j)
	- \sum_{ijkl} (g^{kl}+\partial_l\widetilde{\bX}_k^1)\Gamma_{li}^j\partial_k\bU_{i}^2\bpsi_j\\
	+ \sum_{ijkl} (g^{kl}\Gamma_{li}^j\bU_{i}^2\partial_k\bpsi_j
	+ g^{jl}\Gamma_{li}^k\bU_i^2\partial_k\bpsi_j)
	- \sum_{ijklm} (g^{kl}+\partial_k\widetilde{\bX}_l^1)\Gamma_{li}^m\Gamma_{km}^j\bU_{i}^2\bpsi_j
	\big).
	\end{multline}
	If we denote
	$$
	\left\langle \Delta\bU_2, \bpsi\right\rangle
	:= \int_{\Omega_F(\tau)}2\,\D\bU_2\cdot\D\bpsi,
	$$
	we get
	\begin{multline*}
	\left\langle (\mathcal{L}-\Delta)\bU_2, \bpsi\right\rangle
	= \int_{\Omega_F(\tau)} \big(\sum_{ijk}
	((g^{jk}-\delta_{jk})\partial_j\bU_{i}^2\partial_k\bpsi_i
	+ (g^{jk}-\delta_{jk})\partial_k\bU_i^2\partial_i\bpsi_j)
	- \sum_{ijkl} (g^{kl}+\partial_l\widetilde{\bX}_k^1)\Gamma_{li}^j\partial_k\bU_{i}^2\bpsi_j\\
	+ \sum_{ijkl} (g^{kl}\Gamma_{li}^j\bU_{i}^2\partial_k\bpsi_j
	+ g^{jl}\Gamma_{li}^k\bU_i^2\partial_k\bpsi_j)
	- \sum_{ijklm} (g^{kl}+\partial_k\widetilde{\bX}_l^1)\Gamma_{li}^m\Gamma_{km}^j\bU_{i}^2\bpsi_j
	\big).
	\end{multline*}
	
	\noindent
	\textbf{Rigid body terms}:
	In the solid domain $S_2(t)$ we have
	\begin{equation*}
	\bu_2(t,\bx_2) = \Q(t)\,\bU_2(t,\widetilde{\bX}_1(t,\bx_2)),
	\end{equation*}
	\begin{equation*}
	\bphi(t,\bx_2) = \Q(t)\bpsi(t,\widetilde{\bX}_1(t,\bx_2)),
	\end{equation*}
	which implies
	\begin{align}
	\bu_2\cdot\partial_t\bphi
	&= \Q\bU_2\cdot\frac{d}{dt}(\Q\bpsi)
	= \Q\bU_2\cdot(\Q'\bpsi + \Q\partial_t\bpsi + \Q\nabla\bpsi\partial_t\widetilde{\bX}_1)
	\notag\\
	&= \bU_2\cdot\Q^T\Q'\bpsi
	+ \bU_2\cdot\partial_t\bpsi
	+ \bU_2\cdot\nabla\bpsi\partial_t\widetilde{\bX}_1
	\notag\\
	&= \bU_2\cdot\widetilde{\bomega}\times\bpsi
	+ \bU_2\cdot\partial_t\bpsi
	+ \bU_2\cdot\nabla\bpsi\partial_t\widetilde{\bX}_1.
	\notag
	\end{align}
	Since $\partial_t\widetilde{\bX}_1=\bu_1-\bU_2$ and $\bpsi(t,\mathbf{x})$ is rigid, i.e.
	\begin{equation*}
	\bpsi(t,\mathbf{x_1})=\boldsymbol{\psi }_{h}(t)+\boldsymbol{\psi }
	_{\omega }(t)\times (\bx_1-\mathbf{q}_{1}(t))
	\quad\Rightarrow\quad
	\nabla\bpsi\,\bx_1 = \bpsi_{\omega}\times\bx_1,
	\end{equation*}
	it follows
	\begin{align}
	\bu_2\cdot\partial_t\bphi
	&= \bU_2\cdot\widetilde{\bomega}\times\bpsi
	+ \bU_2\cdot\partial_t\bpsi
	+ \bU_2\cdot(\bpsi_{\omega}\times(\bu_1-\bU_2))
	\notag\\
	&= -\widetilde{\bomega}\times\bU_2\cdot\bpsi
	+ \bU_2\cdot\partial_t\bpsi
	+ \bu_1\times\bU_2\cdot\bpsi_{\omega}
	- \underbrace{\bU_2\times\bU_2\cdot\bpsi_{\omega}}_{=0}.
	\end{align}

\end{proof}

As in \cite{chemetov2017weak}, the following two lemmas give us
estimates for the additional terms in \eqref{transformed_weak}.

\begin{lemma}	\label{estimateOmega}
	For the vector $\widetilde{\boldsymbol{\omega }}$, as defined in Proposition \ref{WeakU2},
	the following equality holds:
	\begin{equation*}
	\widetilde{\boldsymbol{\omega }}(t)
	= \bOmega_{2}(t) - \bomega_{1}(t),
	\qquad \forall t \in [0,T_0].
	\end{equation*}
\end{lemma}
\begin{lemma}	\label{estimateF}
	The following estimate holds:
	\begin{equation*}
	\Vert \widetilde{F}\Vert
	_{L^{2}(0,T;(H^{1}(\Omega _{F}^{1}(t))'))}
	\leq C\left( ||\mathbf{a}_{1}-\mathbf{A}_{2}||_{L^{2}(0,T)}+||
	\boldsymbol{\omega }_{1}-\boldsymbol{\Omega }_{2}||_{L^{2}(0,T)}\right),
	\end{equation*}
	where $C$ depends only on
	$\Vert \mathbf{U}_{2}\Vert_{L^{2}(0,T;H^{1}(\Omega _{F}^{1}(t)))}$
	and
	$\Vert \mathbf{U}_{2}\Vert_{L^{\infty }(0,T;L^{2}(\Omega _{F}(t)))}$.
\end{lemma}

\begin{proof}
	
	As in \cite{chemetov2017weak}, we get that
	\begin{equation}
	\left.
	\begin{array}{l}
	\Vert \widetilde{\mathbf{X}}_{2}(t,.)-\mathrm{id}{\Vert _{W^{3,\infty
			}(\Omega _{F}^{1}(t))}}\leq C(\Vert \mathbf{a}_{1}-\mathbf{A}_{2}\Vert
	_{L^{2}(0,T)}+\Vert \boldsymbol{\omega }_{1}-\boldsymbol{\Omega }_{2}\Vert _{L^{2}(0,T)}), \\
	\Vert \partial _{t}\widetilde{\mathbf{X}}_{2}(t,.)\Vert _{W^{1,\infty
		}(\Omega _{F}^{1}(t))}\leq C(|\mathbf{a}_{1}(t)-\mathbf{A}_{2}(t)|+|
	\boldsymbol{\omega }_{1}(t)-\boldsymbol{\Omega }_{2}(t)|),
	\end{array}
	\right\} \,\, t\in \lbrack 0,T].  \label{TrEst1}
	\end{equation}
	and 
	\begin{eqnarray*}
		\Vert g_{ij}(t)-\delta _{ij}\Vert _{W^{1,\infty }(\Omega _{F}(t))}
		&+&\Vert
		g^{ij}(t)-\delta _{ij}\Vert _{W^{1,\infty }(\Omega _{F}(t))}
		+ \Vert \Gamma
		_{ij}^{k}(t)\Vert _{L^{\infty }(\Omega _{F}(t))} \\
		&\leq &C(\Vert\mathbf{a}_{1}-\mathbf{A}_{2}\Vert_{L^{2}(0,T)}+\Vert\bomega_{1}-\bOmega_{2}\Vert_{L^{2}(0,T)}),\qquad t\in \lbrack
		0,T].
	\end{eqnarray*}	
	Now, for $\bpsi\in H^1(\Omega_{F}(\tau))$ we obtain the following estimates:
	\begin{align*}
	\vert\left\langle (\mathcal{L}-\Delta)\bU_2, \bpsi\right\rangle\vert
	&= \Big\vert
	\int_{\Omega_F(\tau)} \sum_{ijk}
	(g^{jk}-\delta_{jk})\partial_j\bU_{i}^2\partial_k\bpsi_i
	+ \int_{\Omega_F(\tau)} \sum_{ijk}
	(g^{jk}-\delta_{jk})\partial_k\bU_i^2\partial_i\bpsi_j
	\\
	&\qquad- \sum_{ijkl} (g^{kl}+\partial_l\widetilde{\bX}_k^1)\Gamma_{li}^j\partial_k\bU_{i}^2\bpsi_j
	+ \sum_{ijkl} (g^{kl}\Gamma_{li}^j\bU_{i}^2\partial_k\bpsi_j
	+ g^{jl}\Gamma_{li}^k\bU_i^2\partial_k\bpsi_j)
	\\
	&\qquad- \sum_{ijklm} (g^{kl}+\partial_k\widetilde{\bX}_l)\Gamma_{li}^m\Gamma_{km}^j\bU_{i}^2\bpsi_j
	\Big\vert
	\\
	&\leq C\big(
	\Vert
	g^{ij}(t)-\delta_{ij}\Vert_{L^{\infty }(\Omega _{F}(\tau))}
	+
	\Vert \Gamma
	_{ij}^{k}(t)\Vert_{L^{\infty }(\Omega_{F}(\tau))}
	\big)
	\Vert
	\bU_2
	\Vert_{H^{1}(\Omega_{F}(\tau))}
	\Vert
	\bpsi
	\Vert_{H^{1}(\Omega_{F}(\tau))}
	\\
	&\leq C(\Vert\ba_{1}-\bA_{2}\Vert_{L^{2}(0,T)}
	+\Vert\bomega_{1}-\bOmega_{2}\Vert_{L^{2}(0,T)})
	\Vert
	\bU_2
	\Vert_{H^{1}(\Omega_{F}(\tau))}
	\Vert
	\bpsi
	\Vert_{H^{1}(\Omega_{F}(\tau))},
	\end{align*}
	which are different from results in \cite{chemetov2017weak}.
	The other terms are the same as in \cite{chemetov2017weak}.	
\end{proof}

\subsection{Existence of an associated pressure - proof of Theorem \ref{pressure}}\label{Sec:Pressure}

Since the pressure is defined only on the fluid domain, we decompose the test space in two parts which correspond to the fluid part and the rigid body part. More precisely, we introduce the following decomposition of the space $V(t)$:

\begin{lemma} \label{V_decomposition}
	Let $V(t)$ be the function space defined by \eqref{FluidFS}. Then
	\begin{equation*}
	V(t)= \widetilde{H}_{0,\sigma}^1(S(t)) + W(S(t)),
	\end{equation*}
	where
	$$
	\widetilde{H}_{0,\sigma}^1(S(t)) = \{\, \bphi_0\in H^1_0(\Omega) \,:\, \divg\bphi_0=0 \text{ in } \Omega,\ \bphi_{0}= 0 \text{ in } S(t)\,\},
	$$
	$$
	W(S(t)) = \{\, \bphi\in H^1_0(\Omega) \,:\,\bphi = \mathrm{Ext}(\ba+\bomega\times(\bx-\bq)) \text{ in } S(t), \text{ for some } \ba,\bomega\in\R^3 \,\},
	$$
	the extension operator $\mathrm{Ext}$ is described in Appendix \ref{Sec:LT}.
	
\end{lemma}
\begin{proof}
	Let $\bphi\in V(t)$. There exist $\ba,\bomega\in\R^3$
	such that
	$\bphi = \ba+\bomega\times(\bx-\bq(t))$ in $S(t)$.
	Let $\bphi_{\ba,\bomega}\in H^1_0(\Omega)$ be an extension of $\bphi$ such that $\divg \bphi_{\ba,\bomega} = 0$. The construction of such a function can be found in \cite{GGH13}, Section 3.
	It follows that
	$\bphi_{\ba,\bomega}\in  W(S(t))$ and
	$\bphi_{0} := \bphi - \bphi_{\ba,\bomega} \in \widetilde{H}_{0,\sigma}^1(S(t))$ .
\end{proof}

Now we see that the equation \eqref{weak} for $\bu_2$ is equivalent to the following couple of equations:
\begin{multline}	\label{test_a_omega}
\int_{0}^{T}\int_{\Omega}
\bu_2 \cdot \partial_{t}\bphi_{\ba,\bomega}\, \d\bx_2\,\d t
+ \int_{0}^{T}\int_{\Omega_F^2(t)}\big(
(\bu_2\otimes \bu_2) :\nabla\bphi_{\ba,\bomega}
- \,2\,\D\bu_2:\D\bphi_{\ba,\bomega}\,\big)\, \d\bx_2\,\d t
= 0
\end{multline}
for any test function $\bphi_{\ba,\bomega}\in H_0^1(0,T;W(S(t)))$, and
\begin{multline} \label{weak_t}
\int_{0}^{T}\int_{\Omega_F^2(t)}
\bu_2 \cdot \partial_{t}\bphi_{0}\, \d\bx_2\,\d t
+ \int_{0}^{T}\int_{\Omega_F^2(t)}\big(
(\bu_2\otimes \bu_2) :\nabla\bphi_{0}
- \,2\,\D\bu_2:\D\bphi_{0}\,\big)\, \d\bx_2\,\d t
= 0,
\end{multline}
for any test function $\bphi_{0}\in H_0^1(0,T;H_{0,\sigma}^1(S(t)))$.

We derive the pressure equation from \eqref{weak_t}. The construction of the pressure will be divided into the following steps:
\begin{enumerate}
	\item Transform the equation to the fixed domain $\Omega_{F}=\Omega_{F}^2(0)$.
	\item Construct the pressure on the fixed domain.
	\item Transform the equation back to the domain $\Omega_{F}^2(t)$.
\end{enumerate}

\noindent
\textbf{Step 1}.
Let $\bpsi\in H^1(0,T;V(t))$. We define $\bU$ and $\bphi$ such that
\begin{displaymath}
\bu_2(t,\bx_2) = \nabla\bY_2^T(t,\bx_2)\,\bU(t,\bY_2(t,\bx_2)),
\end{displaymath}
\begin{displaymath}
\bphi(t,\bx_2)
= \nabla\bX_2(t,\bY_2(t,\bx_2))\,\bpsi(t,\bY_2(t,\bx_2)),
\end{displaymath}
where $\bX_2$ and $\bY_2$ are the coordinate transformations defined in Section \ref{Sec:TransformedWeak}. Note that the transformation of the velocity does not preserve divergence, but the transformation of the test function does.

By similar calculation as in Section \ref{Sec:TransformedWeak}, we get
\begin{equation*}
\int_{\Omega^2_F(t)}\bu_2\cdot\partial_{t}\bphi\ \d\bx_2 =
\int_{\Omega_F}\big(
\bU\cdot\partial_t\bpsi + \nabla(\bU\cdot\bpsi)\cdot\partial_{t}\bY_2
\big)\ \d\by
-\left\langle \widehat{\mathcal{M}}\bU,\bpsi \right\rangle,
\end{equation*}
\begin{equation*}
\int_{\Omega^2_F(t)} \bu_2\otimes\bu_2:\nabla\bphi \ \d\bx_2
= \int_{\Omega_F}
\nabla\bY_2\nabla\bY_2^T(\bU\otimes\bU):\nabla\bpsi^T
\ \d\by
+ \left\langle \widehat{\mathcal{N}}\bU,\bpsi \right\rangle,
\end{equation*}
\begin{equation*}
\int_{\Omega^2_F(t)} 2\,\D\bu_2:\D\bphi \ \d\bx_2
= \left\langle \widehat{\mathcal{L}}\bU,\bpsi \right\rangle,
\end{equation*}
where $\widehat{\mathcal{M}},\widehat{\mathcal{N}}$ and $\widehat{\mathcal{L}}$ are the functionals defined by

\begin{displaymath}
\left\langle \widehat{\mathcal{M}}\bU,\bpsi \right\rangle
:= \int_{\Omega_F}\big(\nabla\bU\partial_{t}\bY_2 + \nabla\bX_2^T\partial_{t}\nabla\bY_2^T\bU \big)\cdot\bpsi\ \d\by,
\end{displaymath}
\begin{displaymath}
\left\langle \widehat{\mathcal{N}}\bU,\bpsi \right\rangle
:= \int_{\Omega_F }
\Big(
\sum_{ijkl}\Gamma_{il}^{j}g^{kl}\bU_{j}\bU_{k}\bpsi_i
\Big)\ \d\by,
\end{displaymath}
\begin{multline*}
\left\langle \widehat{\mathcal{L}}\bU,\bpsi \right\rangle
:= 
\int_{\Omega_F(\tau)} \Big(\sum_{ijk}
(g^{jk}\partial_j\bU_{i}\partial_k\bpsi_i
+ g^{jk}\partial_k\bU_i\partial_i\bpsi_j)
- \sum_{ijkl} (g^{kl}+\partial_l\bY_k)\Gamma_{li}^j\partial_k\bU_{i}\bpsi_j\\
+ \sum_{ijkl} (g^{kl}\Gamma_{li}^j\bU_{i}\partial_k\bpsi_j
+ g^{jl}\Gamma_{li}^k\bU_i\partial_k\bpsi_j)
- \sum_{ijklm} (g^{kl}+\partial_k\bY_l)\Gamma_{li}^m\Gamma_{km}^j\bU_{i}\bpsi_j
\Big)\d\by.
\end{multline*}
where $g^{ij}$ and $\Gamma_{il}^{j}$ 
are coefficients defined by \eqref{covariant} and \eqref{Christoffel}, respectively, which depend on transformations $\bX_2,\bY_2$.
The difference with Section \ref{Sec:TransformedWeak} is that we use a transformation to the fixed domain, so we use different labels to emphasize the difference between the transformed operators $\widehat{\mathcal{M}},\widehat{\mathcal{N}},\widehat{\mathcal{L}}$ and $\mathcal{M},{\mathcal{N}},{\mathcal{L}}$.

Since $\partial_t\bY_2 = -\bU$ in $\partial S_0$, it follows
\begin{align*}
\int_{\Omega_F^2(\tau)} & \bu_2\cdot\partial_t\bphi
+ \bu_2\otimes\bu_2:\nabla\bphi
- 2\,\D\bu_2:\D\bphi \, \d\bx_2
\notag\\
&= \int_{\Omega_F} \big (
\bU\cdot\partial_t\bpsi + \nabla(\bU\cdot\bpsi)\cdot\partial_{t}\bY_2
+\nabla\bY_2\nabla\bY_2^T(\bU\otimes\bU):\nabla\bpsi^T
\big )\, \d\by
\notag\\
&\qquad
- \left\langle
\widehat{\mathcal{L}}\bU
+\widehat{\mathcal{M}}\bU
- \widehat{\mathcal{N}}\bU,\bpsi \right\rangle
\notag\\
&= \int_{\Omega_F} \big (
\bU\cdot\partial_t\bpsi
-\bU\cdot\nabla\bU\cdot\bpsi
- (\I-\nabla\bY_2\nabla\bY_2^T)\bU\otimes\bU:\nabla\bpsi^T
\big )\, \d\by
\notag\\
&\qquad
- \left\langle
\widehat{\mathcal{L}}\bU
+\widehat{\mathcal{M}}\bU
- \widehat{\mathcal{N}}\bU,\bpsi \right\rangle.
\notag
\end{align*}
Now we define
\begin{equation*}
\left\langle \widehat{F}(t), \bpsi\right\rangle =
\left\langle
(\widehat{\mathcal{L}}-\Delta)\bU
+\widehat{\mathcal{M}}\bU
- \widehat{\mathcal{N}}\bU,\bpsi \right\rangle,
\quad \bpsi\in H_0^1(\Omega_{F}),
\end{equation*}
where
\begin{align*}
\left\langle \Delta\bU, \bpsi\right\rangle
&= \int_{\Omega_F} 2\D\bU:\D\bpsi
= \int_{\Omega_F} \sum_{ijk} \delta_{jk}\partial_j\bU_{i}\partial_k\bpsi_i 
+ \int_{\Omega_F} \sum_{ijk} \delta_{jk}\partial_k\bU_{i}\partial_i\bpsi_j.
\end{align*}
Finally, we get the equation
\begin{multline} \label{weak_fixedDomain}
\int_{0}^{T}\int_{\Omega_F}
\bU\cdot \partial_t\bpsi \, \d\by\,\d t
-\int_{0}^{T}\int_{\Omega_F}2\,\D\bU:\D\bpsi\, \d\by\, \d t\\
-\int_{0}^{T}\int_{\Omega_F}\Big(
\bU\cdot\nabla\bU\cdot\bpsi
+(\I-\nabla\bY_2\nabla\bY_2^T)\bU\otimes\bU:\nabla\bpsi^T \Big)\, \d\by\, \d t
= \int_{0}^{T} \langle \widehat{F}(t),\bpsi(t)\rangle\, \d t,
\end{multline}
for all $\bpsi\in H_0^1(0,T;V(0))$.

\noindent
\textbf{Step 2}.
On a fixed domain we can use the existing results on the existence of pressure. We use the construction from \cite{Neu}  Section 7.3.2.C or more generally Theorem 3 in Section 7.3.4.

By using the test function $\vartheta(t)\bpsi(\by)$ in the equation \eqref{weak_fixedDomain} we get
\begin{equation*}
\int_{0}^{T}\big(\left\langle -\bU,\bpsi\right\rangle_{\Omega_F}\vartheta'
+ \left\langle H,\bpsi\right\rangle_{\Omega_F}\vartheta \big)\ \d t
= 0
\quad \forall \bpsi\in C_{0,\sigma}^{\infty}(\Omega),
\forall \vartheta\in C_{0}^{\infty}(0,T),
\end{equation*}
where
\begin{multline*}
\left\langle H,\bpsi\right\rangle_{\Omega_F}
:= \langle \widehat{F}(t),\bpsi\rangle
+\int_{\Omega_F}\Big(
\bU\cdot\nabla\bU\cdot\bpsi
+(\I-\nabla\bY_2\nabla\bY_2^T)\bU\otimes\bU:\nabla\bpsi^T
+ 2\,\D\bU:\D\bpsi\,\Big)\, \d\by.
\end{multline*}
It is easy to show that $H\in L^1(0,T;W_0^{-1,2}(\Omega_F))$.
Therefore, Theorem 3 from \cite{Neu} implies that there exist functions $ p_0 $ and $ p_1 $ such that
\begin{equation*}
\int_{0}^{T}\big(\left\langle -\bU,\partial_t\bpsi\right\rangle_{\Omega_{F}}
+ \left\langle H,\bpsi\right\rangle_{\Omega_{F}} \big)\ \d t
= \int_{0}^{T}\int_{\Omega_F } \big(-p_0\divg\partial_{t}\bpsi
+ p_1\divg\bpsi)
\ \d\by\ \d t
\end{equation*}
for all  $\bpsi\in C_{0}^{\infty}((0,T)\times\Omega_F)$, i.e.
\begin{align*}
\int_{0}^{T}\int_{\Omega_F}
\bU\cdot \partial_t\bpsi \, \d\by\,\d t
-\int_{0}^{T}\int_{\Omega_F}&
\Big(
\bU\cdot\nabla\bU\cdot\bpsi
+(\I-\nabla\bY_2\nabla\bY_2^T)\bU\otimes\bU:\nabla\bpsi^T\, \Big)\, \d\by\, \d t
\\
-\int_{0}^{T}\int_{\Omega_F}2\D\bU:\D\bpsi\, \d\by\, \d t
&= \int_{0}^{T} \langle \widehat{F}(t),\bpsi(t)\rangle\, \d t
+\int_{0}^{T}\int_{\Omega_F } \big(p_0\divg\partial_{t}\bpsi
- p_1\divg\bpsi)
\ \d\by\ \d t,
\end{align*}
for all $\bpsi\in C_{0}^{\infty}((0,T)\times\Omega_F)$.
More precisely, if $\Omega_{F}$ is a Lipschitz domain, then $p_0\in L^{\infty}(0,T;L^2(\Omega_{F}))$ and $p_1\in L^{\frac{4}{3}}(0,T;L^2(\Omega_{F}))$ (see \cite{Neu}  Section 7.3.2.C and 7.3.2.E). Therefore, by density, the equation above is true for all $\bpsi\in H_{0}^{1}(0,T; H_{0}^{1}(\Omega_F))$.

\noindent
\textbf{Step 3:}
Finally, we transform the equation back to the domain $\Omega_{F}^2(t)$ by using the same transformations as in Step 1:
\begin{displaymath}
\bU(t,\by) = \nabla\bX_2(t,\by)^T\,\bu_2(t,\bX_2(t,\by)),
\end{displaymath}
\begin{displaymath}
\bpsi(t,\by)
= \nabla\bY_2(t,\bX_2(t,\by))\,\bphi(t,\bX_2(t,\by)),
\end{displaymath}
\begin{displaymath}
p_{i}^{2}(t,\bx_2)
= p_{i}(t,\bY_2(t,\bx_2)),
\quad i=0,1.
\end{displaymath}
Since our transformation preserves the divergence (see \cite{inoue1977existence}), i.e.
\begin{displaymath}
\divg_{\by}\bpsi = \divg_{\bx_2}\bphi,
\end{displaymath}
it follows
\begin{displaymath}
\int_{0}^{T}\int_{\Omega_F } \big(p_0\divg\partial_{t}\bpsi
- p_1\divg\bpsi)
\ \d\by\ \d t
= \int_{0}^{T}\int_{\Omega_F^2(t) } \big(p_0^2\divg\partial_{t}\bphi
- p_1^2\divg\bphi)
\ \d\bx_2\ \d t.
\end{displaymath}
For the remaining terms, all calculations from the first step are the same, so we get
\begin{multline}\label{weak_pressure1}
\int_{0}^{T}\int_{\Omega_F^2(t)}
\bu_2 \cdot \partial_{t}\bphi\, \d\bx_2\,\d t
+ \int_{0}^{T}\int_{\Omega_F^2(t)}\big(
(\bu_2\otimes \bu_2) :\nabla\bphi
- 2\,\D\bu_2:\D\bphi\,\big)\, \d\bx_2\,\d t\\
= \int_{0}^{T}\int_{\Omega_F^2(t) } \big(p_0^2\divg\partial_{t}\bphi
- p_1^2\divg\bphi)\, \d\bx_2\,\d t,
\end{multline}
for all $\bphi\in H_0^1(0,T;H_{0}^{1}(\Omega_{F}^2(t)))$.

By summing \eqref{test_a_omega} and \eqref{weak_pressure1} we get	
\begin{multline*}
\int_{0}^{T}\int_{\Omega}
\bu_2 \cdot \partial_{t}\bphi\, \d\bx_2\,\d t
\,+\, \int_{0}^{T}\int_{\Omega_F^2(t)}\big(
(\bu_2\otimes \bu_2) :\nabla\bphi
- 2\,\D\bu_2:\D\bphi\,\big)\, \d\bx_2\,\d t\\
= \int_{0}^{T}\int_{\Omega_F^2(t) } \big(p_0^2\divg\partial_{t}\bphi
- p_1^2\divg\bphi)\, \d\bx_2\,\d t,
\end{multline*}
for all $\bphi\in H_0^1(0,T;H_0^1(\Omega))$ such that $\mathbb{D}\bphi = 0$ in $S_2(t)$.

\subsection{Regularization procedure}
Since weak solutions are not regular enough to be used as test functions, first we need to construct a regularization in the time variable. The usual convolution is not applicable because the solutions are defined on a moving domain. Therefore we again use the change of variables. A similar construction can be found in \cite{bravin2018weak}.
Let $\bX$ and $\bY$ be the transformations described in Appendix \ref{Sec:LT}. First we define the Lagrangian velocity:
\begin{equation*}
\bar{\bu}(t,\by)=\nabla \bY(t,\bX(t,\by))\, \bu(t,\bX(t,\by)).
\end{equation*}
Then we extend the function $\bar{\bu}$ to the time interval $(-\infty,+\infty)$:
\begin{equation}\label{RegDef1}
\bar{\bu}(t,\cdot) \mapsto \begin{cases}
\xi(t)\bar{\bu}(0,\cdot), & t\leq 0,\\
\xi(t)\bar{\bu}(t,\cdot), & 0<t<T,\\
\xi(t)\bar{\bu}(T,\cdot), & t\geq T,
\end{cases}
\end{equation}
where $\xi\in C_{0}^{\infty}(\R)$ is such that $0\leq \xi\leq T$ and $\xi\equiv 1$ in an open neighbourhood of $[0,T]$.

Now, we define a regularization of $\bar{\bu}$ (convolution in time) by
\begin{equation}\label{RegDef2}
\bar{\bu}^h(t,\by)
= \int_{-\infty}^{+\infty} j_h(t-s)\,\bar{\bu}(s,\by)\,\d s,
\end{equation}
where $j_h\in C_{0}^{\infty}(\R)$ is an even, positive function with support in $(-h,h)$, and $\int_{-\infty}^{+\infty} j_h(s)\d s=1$.

This is a divergence-free function defined on the Lagrangian domain. At the end, we transform it back to the Eulerian domain:
\begin{align}\label{RegDef3}
\bu^h(t,\bx)
&=\nabla \bX(t,\bY(t,\bx))\,\bar{\bu}^h(t,\bY(t,\bx)).
\end{align}
This is again a divergence-free function and we will use it as a test function. Moreover, since
$$
\bar{\bu}^h\to \bar{\bu}\;{\rm in}\; L^2(0,T;H^1(\Omega))
\qquad\text{and}\qquad \bar{\bu}^h\to \bar{\bu}\,{\rm in}\; L^p(0,T;L^q(\Omega)),
\quad p,q\in(1,\infty),
$$
as $h\to 0$, (see Galdi \cite{GaldiNS00}, Lemma 2.5), it is easy to see that
$$
\bu^h\to \bu\;{\rm in}\; L^2(0,T;H^1(\Omega))
\qquad\text{and}\qquad \bu^h\to \bu\;{\rm in}\; L^p(0,T;L^q(\Omega)), \quad p,q\in(1,\infty),
$$
as $h\to 0$.
Moreover, we will need the following version of the Reynolds transport theorem:
\begin{lemma}	\label{Reynolds_gen}
	Let
	$
	\bu,\bv\in{L}^{2}(0,T;{H}^1(\Omega(t)))
	$
	such that $\bu(t),\bv(t)\in{L}^{2}(\Omega(t))$ for all $t\in[0,T]$,
	and let 
	$\bX(t):\Omega\to\Omega(t)$ be a volume-preserving $C^{\infty}$ diffeomorphism such that the derivatives \eqref{partialDerivatives} are bounded.
	Let $\bu^h$ denote the regularization defined by \eqref{RegDef1}-\eqref{RegDef3}.
	Then, for almost every $t\in[0,T]$,
	\begin{multline}	\label{Reynolds_gen_eq}
	\int_{0}^{t}\int_{\Omega(\tau)}\big (\bu\cdot\partial_t\bv^h + \bv\cdot\partial_t\bu^h\big)\,\d\bx\,\d\tau
	\\
	\to -\int_{0}^{t}\int_{\Omega(\tau)} \nabla(\bv\cdot\bu)\cdot \partial_t\bX\,\d\bx\,\d\tau
	+ \int_{\Omega(t)}\bv(t)\cdot \bu(t)\,\d\bx\,
	- \int_{\Omega(0)}\bv(0)\cdot \bu(0)\,\d\bx,
	\end{multline}
	when $h\to 0$.
\end{lemma}
The proof is technical and we postpone it to Appendix \ref{Sec:Reynolds_gen}.

\subsection{Regularity of pressure}\label{Sec:pressurereg}

For the proof of Theorem \ref{WeakStrong} we need to prove more regularity on the pressure $p_2$ in order to define and estimate $\left\langle \widetilde{\mathcal{G}}P_2,\bu_{1}\right\rangle$. Because of the structure of the operator $\widetilde{\mathcal{G}}$ it will be sufficient to prove the local-in-time regularity of $p_2$, i.e. the regularity of $t p_2$.

First we present two auxiliary lemmas:
\begin{lemma}\label{IntConvSerrin}
	Let $\bu_2$ be a weak solution to the problem \eqref{FSINoslip} satisfying the Prodi-Serrin condition \eqref{LrLs}. Then,
	$$
	(\bu_{2}\cdot\nabla)\bu_{2}\in L^p(0,T;L^q(\Omega)),
	\qquad
	\frac{1}{q} = \frac{1}{2} + \frac{1}{s},
	\quad
	\frac{1}{p} = \frac{1}{2} + \frac{1}{r}.
	$$	
\end{lemma}
\begin{proof}	
	Since $\bu_{2}\in L^r(0,T;L^s(\Omega)), \nabla\bu_{2}\in L^2(0,T;L^2(\Omega))$, $\frac{3}{s}+\frac{2}{r}=1$ and $s\in(3,\infty]$, it follows from H\"{o}lder's inequality that
	\begin{align*}
	&\left\| \bu_{2}\cdot\nabla\bu_{2} \right\|_{L^p(0,T;L^q(\Omega))}^p
	=
	\int_{0}^{T}\left( \int_{\Omega } |\bu_{2}\cdot\nabla\bu_{2}|^{q}\,\d\bx_2 \right)^\frac{p}{q}\d t\\
	&\qquad
	\leq
	\int_{0}^{T}\left(
	\|\bu_{2}\|_{L^s(\Omega)}^{q} \|\nabla\bu_{2}\|_{L^2(\Omega)}^{q}
	\right)^\frac{p}{q}\d t\,
	= \int_{0}^{T}\|\bu_{2}\|_{L^s(\Omega)}^{p}
	\|\nabla\bu_{2}\|_{L^2(\Omega)}^{p}\, \d t\\
	&\qquad
	\leq \|\bu_{2}\|_{L^r(0,T;L^s(\Omega))}^{p}
	\|\nabla\bu_{2}\|_{L^2(0,T;L^2(\Omega))}^{p},
	\end{align*}
	for
	\begin{displaymath}
	\frac{1}{s}+\frac{1}{2} = \frac{1}{q},
	\quad
	\frac{1}{r}+\frac{1}{2} = \frac{1}{p}.
	\end{displaymath}
\end{proof}

\begin{lemma} \label{weak_estimates}
	Let $p,q$ be defined as in Lemma \ref{IntConvSerrin}.
	If $\bu\in L^2(0,T; H^1(\Omega))\cap L^{\infty}(0,T; L^2(\Omega))$, then
	$$
	\bu\in L^{p'}(0,T;L^{q'}(\Omega)),
	\qquad
	\frac{1}{p'}+\frac{1}{p}=1,
	\quad
	\frac{1}{q'}+\frac{1}{q}=1.
	$$
	with the estimate
	\begin{equation*}
	\left\| \bu \right\|_{L^{p'}(0,T;L^{q'}(\Omega))}
	\leq 
	\|\bu\|_{L^{\infty}(0,T;L^{2}(\Omega))}
	+
	\|\bu\|_{L^2(0,T;H^1(\Omega))}.
	\end{equation*}
\end{lemma}
\begin{proof}
	We have $\bu\in L^2(0,T; L^6(\Omega))\cap L^{\infty}(0,T; L^2(\Omega))$ by the Sobolev inequality. Now, by the interpolation inequality, we conclude
	\begin{align*}
	\left\| \bu \right\|_{L^{p'}(0,T;L^{q'}(\Omega))}^{p'}
	&=
	\int_{0}^{T}\|\bu\|_{L^{q'}(\Omega)}^{p'}\d t
	\leq
	\int_{0}^{T}\left(
	\|\bu\|_{L^2(\Omega)}^{\alpha} \|\bu\|_{L^6(\Omega)}^{1-\alpha}
	\right)^{p'}\d t\\
	&\leq
	\|\bu\|_{L^{\infty}(0,T;L^{2}(\Omega))}^{\alpha p'}
	\|\bu\|_{L^2(0,T;L^6(\Omega))}^{(1-\alpha)p'},
	\end{align*}
		for $p'\in(2,\infty), q'\in (2,6)$ and $\alpha\in(0,1)$  such that
		\begin{displaymath}
		\frac{1}{q'} = \frac{\alpha}{2} + \frac{1-\alpha}{6}
		\quad\text{and}\quad
		\frac{1}{p'} = \frac{1-\alpha}{2}.
		\end{displaymath}
	By Young's inequality, it follows that
	\begin{align*}
	\left\| \bu \right\|_{L^{p'}(0,T;L^{q'}(\Omega))}
	&\leq
	\alpha \|\bu\|_{L^{\infty}(0,T;L^{2}(\Omega))}
	+ (1-\alpha)\|\bu\|_{L^2(0,T;L^6(\Omega))}
	\leq
	 \|\bu\|_{L^{\infty}(0,T;L^{2}(\Omega))}
	+ \|\bu\|_{L^2(0,T;L^6(\Omega))}
	,
	\end{align*}
	for all $p'\in(2,\infty), q'\in (2,6)$ such that
	\begin{displaymath}
	\frac{3}{q'} + \frac{2}{p'} = \frac{3}{2}.
	\end{displaymath}
	Now, for $p,q$ as in Lemma \ref{IntConvSerrin} we have
	\begin{displaymath}
	\frac{3}{q} + \frac{2}{p} 
	= 3\left(\frac{1}{2} + \frac{1}{s} \right)
	+ 2\left(\frac{1}{2} + \frac{1}{r} \right)
	= \frac{5}{2} + \left(\frac{3}{s} + \frac{2}{r} \right) \stackrel{\eqref{LrLs}}{=} \frac{7}{2},
	\end{displaymath}
	Finally, for $p',q'$ such that
	$\frac{1}{p'}+\frac{1}{p}=1,
	\frac{1}{q'}+\frac{1}{q}=1$, we conclude that
	\begin{displaymath}
	\frac{3}{q'} + \frac{2}{p'} 
	= 3\left(1 - \frac{1}{q} \right)
	+ 2\left(1 - \frac{1}{p} \right)
	= 5 - \left(\frac{3}{q} + \frac{2}{p} \right) 
	= \frac{3}{2},
	\end{displaymath}
	and it easy to see that $p'\in(2,\infty), q'\in (2,6)$.	
\end{proof}

\begin{proposition} \label{pressure_regularity}
	Let $(\bu_2,p_2)$ be a weak solution given by Theorem \ref{pressure}, and let $\bu_2$ satisfy the Prodi-Serrin condition \eqref{LrLs}. Then the following regularity result holds:
	$$
	t\bu_{2}\in L^p(0,T;W^{2,q}(\Omega_F^2(t))),
	\quad
	t\partial_{t}\bu_{2}, t\nabla p_{2}\in L^p(0,T;L^{q}(\Omega_F^2(t))),
	$$
	for $\frac{1}{q} = \frac{1}{2} + \frac{1}{s}$, $\frac{1}{p} = \frac{1}{2} + \frac{1}{r}.$
\end{proposition}
\begin{proof}
	
	The proof of Proposition \ref{pressure_regularity} is analogous to the proof of \cite[Proposition 3]{GlassSueur} and \cite[Lemma 3]{bravin2018weak}. The only difference is that we consider the 3D problem and therefore use the integrability of the convective term given by Lemma \ref{IntConvSerrin}. Therefore we will give a sketch of the proof without going into details. The idea is to consider the following auxiliary system:	
	\begin{equation}	\label{Stokes}
	\begin{array}{l}
	\left.
	\begin{array}{l}
	\partial_{t}\bv - \Delta\bv + \nabla p
	= \mathbf{g} , \\
	\divg\bv = 0
	\end{array}
	\right\} \;\mathrm{in}\;\bigcup_{t\in(0,T)} \{t\}\times\Omega_{F}(t),
	\\
	\left.
	\begin{array}{l}
	{\frac{d}{dt}\ba} = -\int_{\partial S(t)}{\T}(\bv,p) \bn\,\d\bgamma(\bx) 
	{\,+\, \mathbf{g}_1}, \\
	\frac{d}{dt}(\J\bomega) = -\int_{\partial S(t)}(\bx-\bq(t))\times {\T}(\bv,p)\bn\,\d\bgamma(\bx) 
	{\,+\, \mathbf{g}_2}
	\end{array}
	\right\} \;\mathrm{in}\;(0,T),
	\\
	\bv ={\ba}+\bomega\times (\bx-\bq),\quad \mathrm{on}\;\bigcup_{t\in(0,T)} \{t\}\times\partial S(t),
	\\
	\bv = 0 \quad \mathrm{on}\, \partial \Omega,
	\end{array}
	\end{equation}
	{where $\mathbf{g}\in L^{p}(0,T;L^{q}(\Omega_F^2(t))),\, \mathbf{g}_1,\mathbf{g}_2\in L^{p}(0,T)$.}
	Then we get the corresponding regularity result which is given by the following lemma:	
	\begin{lemma} \label{strongStokes}
		There exists a unique solution to the system \eqref{Stokes} on $[0, T]$ with vanishing initial data which satisfies
		$$
		\bv\in L^p(0,T;W^{2,q}(\Omega_F^2(t))),
		\quad
		\partial_{t}\bv, \nabla {p}\in L^p(0,T;L^{q}(\Omega_F^2(t))).
		$$
	\end{lemma}	
	The proof of Lemma \ref{strongStokes} is analogous to the proof of \cite[Lemma 4]{GlassSueur}
	and is based on a result contained in \cite{GGH13}.
	First, by using the change of coordinates described in Appendix \ref{Sec:LT}, we transform the problem \eqref{Stokes} to the fixed domain
		\begin{equation*}	\label{Stokes_fixed}
		\begin{array}{l}
		\left.
		\begin{array}{l}
		\partial_{t}\tilde{\bv} - \Delta\tilde{\bv} + \nabla \tilde{p}
		= \tilde{\mathbf{f}} , \\
		\divg\tilde{\bv} = 0
		\end{array}
		\right\} \;\mathrm{in}\; (0,T)\times\Omega_{F},
		\\
		\left.
		\begin{array}{l}
		\frac{d}{dt}\tilde{\ba} = -\int_{\partial S_0}{\mathcal{T}}(\tilde{\bv},\tilde{p}) \bn\,\d\gamma(\by) \,+\,\tilde{\mathbf{f}}_1, \\
		\frac{d}{dt}(\J\tilde{\bomega}) = -\int_{\partial S_0}\by\times {\mathcal{T}}(\tilde{\bv},\tilde{p})\bn\,\d\gamma(\by)
		\,+\,\tilde{\mathbf{f}}_2
		\end{array}
		\right\} \;\mathrm{in}\;(0,T),
		\\
		\tilde{\bv} =\tilde{\ba} + \tilde{\bomega}\times \by,\quad \mathrm{on}\; (0,T)\times\partial S_0,
		\\
		\tilde{\bv} = 0 \quad \mathrm{on}\, \partial \Omega,
		\end{array}
		\end{equation*}
	where 
		$$
		\tilde{\mathbf{f}}
		= \tilde{\mathbf{g}} +(\mathcal{L}-\Delta)\tilde{\bv}-\mathcal{M}\tilde{\bv}-(\mathcal{G}-\nabla)\tilde{\bv},
		\qquad
		\tilde{\mathbf{f}}_1
		=\Q_1^T\mathbf{g}_1-\tilde{\bomega}\times\tilde{\ba},
		\qquad
		\tilde{\mathbf{f}}_2
		=\Q_1^T\mathbf{g}_2,
		$$
		$$
		\tilde{\mathbf{g}}(t,\by)
		=\nabla\bY_2(t,\bX_2(t,\by)) \mathbf{g}(t,\bX_2(t,\by))
		$$
	Then, we obtain the required regularity from the maximal regularity result \cite[Theorem 4.1]{GGH13} by using the fixed point procedure from \cite[Sections 5-7]{GGH13}.
	Fixed point procedure provides a solution for a sufficiently small time $T'$. However, since $T'$ does not depend on the initial data the solution can be extended to the interval $[0,T']$ by iterating this procedure in the standard way, see \cite[Theorem 4.1]{GGH13} which is stated for an exterior domain but still hold true for bounded domain (\cite[Section 7]{GGH13}).
	
	Now, Proposition \ref{pressure_regularity} is a consequence of Lemma \ref{strongStokes} and the following:
	\begin{enumerate}
		\item Every strong solution of \eqref{Stokes} is a weak solution of \eqref{Stokes} (see \cite[Lemma 5]{GlassSueur}).
		
		\item The weak solution of \eqref{Stokes} is unique (see \cite[Lemma 8]{GlassSueur}).
		
		\item The pair $t\bu_{2}$ is a weak solution of \eqref{Stokes} with 
		$$
		\mathbf{g} = \bu_{2}-t(\bu_{2}\cdot\nabla)\bu_{2}{\in L^{p}(0,T;L^{q}(\Omega_F^2(t)))},
		$$
		$$
		\mathbf{g}_1=\ba \in L^{p}(0,T),
		\quad \mathbf{g}_2=\bomega \in L^{p}(0,T)
		$$			
		(see \cite[Lemma 6]{GlassSueur}).
		The regularity of $\mathbf{g}, \mathbf{g}_1$ and $\mathbf{g}_2$ follows 
		from Lemma \ref{IntConvSerrin} and the fact that $\ba,\bomega\in L^{\infty}(0,T)$.	 
	\end{enumerate}
\end{proof}

The next step is to use the regularity result from Proposition \ref{pressure_regularity} to estimate the difference between the pressure terms.

\begin{lemma}	\label{estimate_pressure1}
	Let $P_2$ be the transformed pressure defined by \eqref{eq:Transformed}$_2$
	and $\widetilde{\mathcal{G}}$ the operator defined by \eqref{weak_pressure}.
	Then,
	\begin{equation}	\label{pressure_identity}
	\left\langle (\widetilde{\mathcal{G}}-\nabla)P_2,\bpsi\right\rangle
	= \int_{0}^{T} \int_{\Omega _{F}^{1}(t)} (\mathcal{G}-\nabla) P_2\cdot \bpsi \,\d\bx_{1}\,\d t
	\end{equation}
	for all $\bpsi\in L^{p'}(0,T;L^{q'}(\Omega_{F}^{1}(t)))$,
	where $\mathcal{G}$ is an operator defined by \eqref{OpP}, and $p', q'$ are defined as in Lemma \ref{weak_estimates}. In addition, we have estimate
	\begin{multline}	\label{estimate_pressure2}
	\Vert (\mathcal{G}-\nabla)P_2 \Vert
	_{L^{p}(0,T;L^{q}(\Omega _{F}^{1}(t)))}\\
	\leq C\left( ||\mathbf{a}_{1}-\mathbf{A}_{2}||_{L^{\infty}(0,T)}+||
	\boldsymbol{\omega }_{1}-\boldsymbol{\Omega }_{2}||_{L^{\infty}(0,T)}\right)
	\Vert t\nabla P_{2}\Vert_{L^{p}(0,T;L^{q}(\Omega _{F}^{1}(t)))},
	\end{multline}
	where $C>0$ is nondecreasing with respect to $T$.
\end{lemma}
\begin{proof}
	By Proposition \ref{pressure_regularity}, $t\nabla p_2\in L^p(0,T;L^q(\Omega_F^2(t)))$, which implies that
	$t\nabla P_2\in L^p(0,T;L^q(\Omega_F^1(t)))$, so we can write
	\begin{align*}
	&\left\langle (\widetilde{\mathcal{G}}-\nabla)P_2,\bpsi\right\rangle
	=
	-\int_{0}^{T} \int_{\Omega _{F}^{2}(t)} p_2\cdot \divg\bphi \,\d\bx_{2}\,\d t
	{
		-\int_{0}^{T} \int_{\Omega _{F}^{1}(t)} \nabla P_2\cdot \bpsi \,\d\bx_{1}\,\d t
	}\\
	&\qquad=
	\int_{0}^{T} \int_{\Omega _{F}^{2}(t)} \nabla p_2\cdot \bphi \,\d\bx_{2}\,\d t
	{
		-\int_{0}^{T} \int_{\Omega _{F}^{1}(t)} \nabla P_2\cdot \bpsi \,\d\bx_{1}\,\d t
	}\\
	&\qquad
	= \int_{0}^{T} \int_{\Omega _{F}^{1}(t)} \nabla\widetilde{\bX}_1\nabla\widetilde{\bX}_1^T\nabla P_2\cdot \bpsi \,\d\bx_{1}\,\d t
	{
		-\int_{0}^{T} \int_{\Omega _{F}^{1}(t)} \nabla P_2\cdot \bpsi \,\d\bx_{1}\,\d t
	}\\
	&\qquad
	= \int_{0}^{T} \int_{\Omega _{F}^{1}(t)} (\mathcal{G}-\nabla) P_2\cdot \bpsi \,\d\bx_{1}\,\d t
	\end{align*}
	for all $\bphi\in C_{0}^{\infty}((0,T)\times\Omega_{F}^{2}(t))$ and for $\bpsi=\nabla\widetilde{\bX}_{2}^T\bphi$.
	
	By the construction of $\widetilde{\bX}_1$, it can be shown that
	$$
	\left\|\partial_{t}\widetilde{\bX}_1(t,\cdot)\right\|_{W^{1,\infty}(\Omega_{F}^{1}(t))}
	\leq C
	(\vert\mathbf{a}_{1}(t)-\mathbf{A}_{2}(t)\vert
	+\vert\bomega_{1}(t)-\bOmega_{2}(t)\vert),
	\qquad \forall t\in[0,T]
	$$
	(see \cite{GlassSueur,chemetov2017weak} and Appendix \ref{Sec:LT}) and then we get the following estimates:
	\begin{align*}
	&\left\|\frac{1}{t} (\nabla\widetilde{\bX}_1\nabla\widetilde{\bX}_1^T-\I)\right\|_{L^{\infty}(\Omega _{F}^{1}(t))}
	\notag\\
	&\qquad
	\leq
	\left\| \nabla\widetilde{\bX}_1\right\|_{L^{\infty}((0,T)\times\Omega _{F}^{1}(t))} \left\|\frac{1}{t}(\nabla\widetilde{\bX}_1^T-\I)\right\|_{L^{\infty}(\Omega _{F}^{1}(t))}
	+
	\left\|\frac{1}{t}(\nabla\widetilde{\bX}_1-\I)\right\|_{L^{\infty}(\Omega _{F}^{1}(t))}
	\notag\\
	&\qquad
	\leq C \left\|\partial_{t}\nabla\widetilde{\bX}_1\right\|_{L^{\infty}(\Omega _{F}^{1}(t))}
	\leq C (\vert\mathbf{a}_{1}(t)-\mathbf{A}_{2}(t)\vert
	+\vert\bomega_{1}(t)-\bOmega_{2}(t)\vert),
	\end{align*}
	since $\widetilde{\bX}_1\in W^{1,\infty}(0,T;C^{\infty}(\Omega))$.
	Now, we have
	\begin{align}
	\int_{0}^{T} \int_{\Omega _{F}^{1}(t)}
	&\left| (\nabla\widetilde{\bX}_1\nabla\widetilde{\bX}_1^T-\I)\nabla P_2\cdot \bpsi \right|
	\notag\\
	&\leq
	\int_{0}^{T} \int_{\Omega _{F}^{1}(t)}
	\left|
	\frac{1}{t} (\nabla\widetilde{\bX}_1\nabla\widetilde{\bX}_1^T-\I)\right|
	\left|t\nabla P_2\right|
	\left|\bpsi \right|
	\notag\\
	&\leq C
	\int_{0}^{T}
	(\vert\mathbf{a}_{1}(t)-\mathbf{A}_{2}(t)\vert
	+\vert\bomega_{1}(t)-\bOmega_{2}(t)\vert)
	\left\|t\nabla P_2\right\|_{L^{q}(\Omega _{F}^{1}(t))}
	\left\|\bpsi \right\|_{L^{q'}(\Omega _{F}^{1}(t))}
	\label{estimate_pressure_term}\\
	&\leq C (\Vert\mathbf{a}_{1}-\mathbf{A}_{2}\Vert_{L^{\infty}(0,T)}
	+\Vert\bomega_{1}-\bOmega_{2}\Vert_{L^{\infty}(0,T)})
	\left\|t\nabla P_2\right\|_{L^{p}L^{q}} \left\|\bpsi
	\right\|_{L^{p'}L^{q'}},
	\notag
	\end{align}
	for all $\bpsi\in C_{0}^{\infty}((0,T)\times\Omega_{F}^{1}(t))$.
	By a standard argument, we conclude that
	$$
	\int_{0}^{T} \int_{\Omega _{F}^{1}(t)}
	(\widetilde{\bX}_1\nabla\widetilde{\bX}_1^T-\I)\nabla P_2\cdot \bpsi\,\d\bx_{1}\,\d t
	$$
	is well defined for all
	$\bpsi\in L^{p'}(0,T;L^{q'}(\Omega_{F}^{1}(t)))$
	and the estimate \eqref{estimate_pressure2} holds.
\end{proof}

\begin{remark}
	Following from Lemma \ref{weak_estimates} and the embedding $H^1((0,T)\times\Omega_{F}^{1}(t))\hookrightarrow L^{\infty}(0,T;L^2(\Omega))$, identity \eqref{pressure_identity} is true for all $\bpsi\in H^1(0,T;V_1(t))$.
\end{remark}

\begin{remark}	\label{weak_transformed}
	By taking the test function $(1-{sgn}_{+}^{\varepsilon}(\cdot-t))\bpsi$ for any fixed $t\in (0,T)$ and letting $\varepsilon\to 0$, it can be shown that the transformed solution $(\bU_2,P_2,\bA_2, \bOmega_2)$ satisfies the following equality:
	\begin{multline} \label{transformed_weak2}
	\int_{0}^{t}\int_{\Omega}\bU_2\cdot\partial_t\boldsymbol{\psi }\, \d\bx_1\,\d \tau
	+\int_{0}^{t}\int_{\Omega_F^1(\tau)}\Big((\bu_1\otimes \bU_2) :\nabla\boldsymbol{\psi }^T
	-2\,\D\bU_2:\D\boldsymbol{\psi }\,\Big)\, \d\bx_1\, \d \tau
	\\
	-\int_{0}^{t}\int_{\Omega _{F}^{1}(\tau)}(\bU_{2}-\bu_{1})\cdot \nabla \bU_{2}\cdot \boldsymbol{\psi }\ \,\d\bx_{1}\,\d \tau
	- \int_{\Omega_F^{1}(t)}\bU_2(t)\cdot\bpsi(t)\, \d\bx_1
	\\
	= \int_{0}^{t} \langle F(\tau),\boldsymbol{\psi }(\tau)\rangle\, \d \tau
	- \int_{\Omega_F}\bu_0\cdot\bpsi(0)\, \d\bx_1
	\\
	+ \int_{0}^{t} \int_{S_1(\tau)} \widetilde{\bomega}\times \bU_{2}\cdot \bpsi-\bu_{1}\times\bU_{2}\cdot\bpsi_{\omega}\,\d\bx_1\,\d \tau,
	\end{multline}
	for any test function $\bpsi\in H^1(0,T;V_1(t))$, where
	\begin{equation*}
	\left\langle {F}(t), \bpsi\right\rangle = \left\langle (\mathcal{L}-\Delta)\bU_{2}
	+ \mathcal{M}\bU_{2}
	+ \tilde{\mathcal{N}}\bU_{2}
	+ (\mathcal{G}-\nabla)P_{2}
	, \bpsi\right\rangle,
	\quad \forall\bpsi\in H^{1}(\Omega_{F}^{1}(t)).
	\end{equation*}
	Furthermore, by Lemma \ref{estimateF} and Lemma \ref{estimate_pressure1} the following estimates hold
	\begin{align}
	\int_0^t \left\langle {F}(t), \bpsi\right\rangle\,\d\tau
	&= \int_0^t \left\langle \widetilde{F}(t), \bpsi\right\rangle\,\d\tau
	+ \int_0^t \left\langle (\mathcal{G}-\nabla)P_{2}, \bpsi\right\rangle\,\d\tau
	\notag
	\\
	&\leq C_1
	\|\bpsi\|_{L^2(0,T;H^1(\Omega))}
	+ C_2\|\bpsi\|_{L^{p'}(0,T;L^{q'}(\Omega))}
	\label{estimateF2} 
	\\
	&
	\leq C
	\left(
	\|\bpsi\|_{L^2(0,T;H^1(\Omega))}
	+ \|\bpsi\|_{L^{\infty}(0,T;L^{2}(\Omega))}
	\right).
	\notag
	\end{align}
	for all $\bpsi\in L^2(0,T; H^1(\Omega))\cap L^{\infty}(0,T; L^2(\Omega))$.
	The last estimate follows from Lemma \ref{weak_estimates}.
\end{remark}

\begin{lemma}	\label{estimate_pressure}
	Let
	$
	(\bu,\ba,\bomega)
	= (\bu_1-\bU_{2},\ba_1-\bA_{2},\bomega_1-\bOmega_{2})
	$
	be the difference between two weak solutions and let $p, q$ be defined as in Lemma \ref{IntConvSerrin}. Then the following estimate holds:
	\begin{equation*}
	\begin{split}
	\left| \int_{0}^{t} \int_{\Omega _{F}^{1}(t)} (\mathcal{G}-\nabla) P_2\cdot \bu \,\d\bx_{1}\,\d \tau\right|
	\leq &\,C
	\int_{0}^{t} \left( \Vert \tau\nabla P_2\Vert_{L^{q}(\Omega _{F}^{1}(\tau))}^{p}+1\right)\Vert\bu(\tau)\Vert_{L^{2}(\Omega _{F}^{1}(\tau))}^2 \,\d \tau	\\
	&+ \varepsilon
	\int_{0}^{t}\Vert\nabla\bu(\tau)\Vert_{L^{2}(\Omega _{F}^{1}(\tau))}^2 \,\d \tau
	\end{split}
	\end{equation*}
	for all $\varepsilon>0$.
\end{lemma}
\begin{proof}
	Lemma \ref{weak_estimates} and \eqref{estimate_pressure_term}
	imply the following estimates:
	\begin{align*}
	&\left| \int_{0}^{t} \int_{\Omega _{F}^{1}(\tau)} (\mathcal{G}-\nabla) P_2\cdot \bu \,\d\bx_{1}\,\d \tau\right|
	\\
	&\qquad\leq C
	\int_{0}^{t}
	(\vert\mathbf{a}_{1}(\tau)-\mathbf{A}_{2}(\tau)\vert
	+\vert\bomega_{1}(\tau)-\bOmega_{2}(\tau)\vert)
	\left\|\tau\nabla P_2\right\|_{L^{q}(\Omega _{F}^{1}(\tau))}
	\left\|\bu \right\|_{L^{q'}(\Omega _{F}^{1}(\tau))}\,\d \tau\\
	&\qquad\leq C
	\int_{0}^{t}
	\left\|\tau\nabla P_2\right\|_{L^{q}(\Omega _{F}^{1}(\tau))}^{2-p}\left\|\bu \right\|_{L^{q'}(\Omega _{F}^{1}(\tau))}^2
	+\left\|\tau\nabla P_2\right\|_{L^{q}(\Omega _{F}^{1}(\tau))}^{p}
	\left\|\bu\right\|_{L^{2}(\Omega _{F}^{1}(\tau))}^{2}
	\,\d \tau.
	\end{align*}
	The second inequality follows from the identity
	$2b\mu\nu\leq b^{2\alpha}\mu^2+b^{2(1-\alpha)}\nu^2$. To estimate the first term we use the interpolation (with $\alpha=\frac{2}{p'}=\frac{3}{s}$) and Young's inequality
	\begin{align*}
	\int_{0}^{t}&
	\left\|\tau\nabla P_2\right\|_{L^{q}(\Omega _{F}^{1}(\tau))}^{2-p}\left\|\bu \right\|_{L^{q'}(\Omega _{F}^{1}(\tau))}^2
	\,\d \tau
	\\
	&\leq C
	\int_{0}^{t}
	\left\|\tau\nabla P_2\right\|_{L^{q}(\Omega _{F}^{1}(\tau))}^{2-p}
	\left\|\bu \right\|_{L^{2}(\Omega _{F}^{1}(\tau))}^{2(1-\alpha)}
	\left\|\bu \right\|_{H^{1}(\Omega _{F}^{1}(\tau))}^{2\alpha}
	\,\d \tau
	\\
	&\leq C
	\int_{0}^{t}
	\left\|\tau\nabla P_2\right\|_{L^{q}(\Omega _{F}^{1}(\tau))}^{\frac{2-p}{1-\alpha}}
	\left\|\bu \right\|_{L^{2}(\Omega _{F}^{1}(\tau))}^{2}
	\,\d \tau
	+\varepsilon
	\int_{0}^{t}
	\left\|\bu \right\|_{H^{1}(\Omega _{F}^{1}(\tau))}^{2}
	\,\d \tau
	\\
	&= C
	\int_{0}^{t}
	\left\|\tau\nabla P_2\right\|_{L^{q}(\Omega _{F}^{1}(\tau))}^{p}
	\left\|\bu \right\|_{L^{2}(\Omega _{F}^{1}(\tau))}^{2}
	\,\d \tau
	+\varepsilon
	\int_{0}^{t}
	\left\|\bu \right\|_{H^{1}(\Omega _{F}^{1}(\tau))}^{2}
	\,\d \tau
	\\
	&\leq C
	\int_{0}^{t}
	\left(\left\|\tau\nabla P_2\right\|_{L^{q}(\Omega _{F}^{1}(\tau))}^{p} + 1 \right)
	\left\|\bu \right\|_{L^{2}(\Omega _{F}^{1}(\tau))}^{2}
	\,\d \tau
	+\varepsilon
	\int_{0}^{t}
	\left\|\nabla\bu \right\|_{L^{2}(\Omega _{F}^{1}(\tau))}^{2}
	\,\d \tau.
	\end{align*}
\end{proof}

\begin{remark}
	Let
	$
	(\bu,\ba,\bomega)
	= (\bu_1-\bU_{2},\ba_1-\bA_{2},\bomega_1-\bOmega_{2})
	$
	be the difference of two weak solutions and let $p, q$ be defined as in Lemma \ref{IntConvSerrin}. Then 	
	Lemma \ref{estimateF} implies	
	\begin{align*}
	\int_0^t \left\langle \widetilde{F}(t), \bu\right\rangle\,\d\tau
	&
	\leq C_1\left( \|{\ba}_{1}-{\bA}_{2}\|_{L^{2}(0,t)}
	+ \|
	{\bomega }_{1}-{\bOmega }_{2}\|_{L^{2}(0,t)}\right)	\|\bu\|_{L^2(0,t;H^1(\Omega))}
	\\
	&\leq C_2 \|\bu\|_{L^{2}(0,t;L^2(\Omega))}	\|\bu\|_{L^2(0,t;H^1(\Omega_{F}^1(t)))}
	\\
	&\leq
	C \|\bu\|_{L^{2}(0,t;L^2(\Omega))}^2	
	+\varepsilon \|\nabla\bu\|_{L^{2}(0,t;L^2(\Omega))}^2
	\end{align*}
	and by Lemma \ref{estimate_pressure} we get the estimate we will need in the proof of uniqueness.
	\begin{align}
	\int_0^t \left\langle {F}(t), \bu\right\rangle\,\d\tau
	&= \int_0^t \left\langle \widetilde{F}(t), \bu\right\rangle\,\d\tau
	+ \int_0^t \left\langle (\mathcal{G}-\nabla)P_{2}, \bu\right\rangle\,\d\tau
	\notag
	\\
	&\leq C
	\int_{0}^{t} \left( \Vert \tau\nabla P_2\Vert_{L^{q}(\Omega _{F}^{1}(\tau))}^{p}+1\right)\Vert\bu(\tau)\Vert_{L^{2}(\Omega _{F}^{1}(\tau))}^2 \,\d \tau
	+ \varepsilon
	\int_{0}^{t}\Vert\nabla\bu(\tau)\Vert_{L^{2}(\Omega _{F}^{1}(\tau))}^2 \,\d \tau 
	\label{estimateF_Gronwall}
	\end{align}
	
\end{remark}

\subsection{Transformed Energy-Type Equality}

In this sub-section we prove an energy-type equality which is satisfied by the transformed solution:

\begin{proposition} 	\label{EnergyEquality}
	The transformed solution $(\bU_2,P_2)$ satisfies the energy-type equality
	\begin{equation}
	\begin{split}
	&\frac{1}{2}\|\bU_2(t)\|_{L^2(\Omega)}^2
	+ \int_{0}^{t}\int_{\Omega_{F}^1(\tau)}\, (\bU_2-\bu_1)\cdot\nabla\bU_2\cdot\bU_2\,\d\bx_1\d\tau
	\\
	&\qquad
	+ \,\int_0^t\int_{\Omega_{F}^1(\tau)} 2|\D\bU_2|^2\,\d\bx_1\d\tau
	=
	-\int_{0}^{t}\langle F(\tau),\mathbf{U}_2(\tau) \rangle \,\d\tau
	+ \frac{1}{2}\|\bu_0\|_{L^2(\Omega)}^2,
	\end{split}
	\end{equation}
	for almost every $t\in[0,T].$
\end{proposition}
Before proving Proposition \ref{EnergyEquality} we need the following lemma.
\begin{lemma}	\label{additionalLemma}
	Suppose that $\bu$ satisfies Prodi-Serrin condition \eqref{LrLs}
	and let $\bv,\bw\in L^2(0,T;V(t)))$ be such that $\bv\in L^{\infty}(0,T;L^2(\Omega))$.	
	Then
	\begin{multline*}
	\Big \vert\, \int_0^T\int_{\Omega_F(t)}\bv\cdot\nabla \bw\cdot \bu \,\,\Big \vert
	\leq C \Big( \int_0^T\Vert\nabla \bw\Vert_{L^2(\Omega_F(t))}^2\, \Big )^{\frac{1}{2}}
	\Big ( \int_0^T
	\Vert \bv \Vert_{H^1(\Omega_F(t))}^2
	\,\Big )^{\frac{n}{2s}}
	\Big ( \int_0^T \Vert \bv\Vert_{L^2(\Omega_F(t))}^2 \Vert \bu\Vert_{L^s(\Omega_F(t))}^r\,\Big )^{\frac{1}{r}},
	\end{multline*}
	where $C$ depends only on $n$, $s$,  $\Omega_F$ and $\|\nabla\bX\|_{L^{\infty}(0,T;L^2(\Omega))}$ for the transformation $\bX$ defined as in Appendix \ref{Sec:LT}.
\end{lemma}	

\begin{proof}
	The proof of this lemma is standard and analogous to the proof for the cylindrical domain (see e.g.  \cite{GaldiNS00}, Lemma 4.1). Since we are working in a moving domain, we reproduce the argument here for the convenience of the reader.
	By H\"older's inequality we get
	\begin{align}
	\Big|\int_{0}^{T} \int_{\Omega_F(t)} \bv\cdot\nabla\bw\cdot\bu\Big|
	&\leq \int_{0}^{T}\Vert\nabla\bw\Vert_{L^2(\Omega_F(t))}\Vert\bv\Vert_{L^p(\Omega_F(t))}\Vert\bu\Vert_{L^s(\Omega_F(t))}
	\notag\\
	&\leq \Big(\int_{0}^{T}\Vert\nabla\bw\Vert_{L^2(\Omega_F(t))}^2\Big)^{\frac{1}{2}}
	\Big(\int_{0}^{T}\Vert\bv\Vert_{L^p(\Omega_F(t))}^2\Vert\bu\Vert_{L^s(\Omega_F(t))}^2\Big)^{\frac{1}{2}},
	\label{estimate1}
	\end{align}
	where
	\begin{equation}	\label{HolderCondition}
	\frac{1}{2}+\frac{1}{p}+\frac{1}{s}=1
	\qquad \Big(s\geq n \,\Rightarrow\, \frac{1}{p}=\frac{1}{2}-\frac{1}{s}\geq \frac{1}{2}-\frac{1}{n} \,\Rightarrow\,
	\bv\in L^{p}(\Omega_F(t))
	\Big)
	\end{equation}
	Let us define transformed function 
	$$
	\overline{\bv} (t,\by) = \bv(t,\bX(t,\by)),\quad\by\in\Omega.
	$$	
	Since $\bX$ preserves the volume, the change of variables and the interpolation inequality give
	\begin{equation*}
		\Vert\bv\Vert_{L^{p}(\Omega_F(t))}
		=\Vert\overline{\bv}\Vert_{L^{p}(\Omega_F)}
		\leq C
		\Vert\overline{\bv}\Vert_{H^{1}(\Omega_F)}^{\alpha}
		\Vert\overline{\bv}\Vert_{L^{2}(\Omega_F)}^{1-\alpha}
		\leq 
		C_1
		\Vert\bv\Vert_{H^{1}(\Omega_F(t))}^{\alpha}
		\Vert\bv\Vert_{L^{2}(\Omega_F(t))}^{1-\alpha},
	\end{equation*}	
	where
	$$\alpha=\frac{n}{s}$$
	and $C_1$ depends only on $n$, $s$, $\Omega_{F}$ and $\|\nabla\bX\|_{L^{\infty}(0,T;L^2(\Omega))}$.
	
	Now we estimate
	\begin{align}
	\Big(\int_{0}^{T}&\Vert\bv\Vert_{L^p(\Omega_F(t))}^2
	\Vert\bu\Vert_{L^s(\Omega_F(t))}^2\Big)^{\frac{1}{2}}
	\leq C\Big(\int_{0}^{T}
	\Vert\bv\Vert_{H^{1}(\Omega_F(t))}^{2\alpha}
	\Vert\bv\Vert_{L^{2}(\Omega_F(t))}^{2(1-\alpha)}
	\Vert\bu\Vert_{L^s(\Omega_F(t))}^2\Big)^{\frac{1}{2}}
	\notag\\
	&\leq C\Big(\int_{0}^{T}
	\Vert\bv\Vert_{H^{1}(\Omega_F(t))}^{2}
	\Big)^{\frac{\alpha}{2}}
	\Big(\int_{0}^{T}
	\Vert\bv\Vert_{L^{2}(\Omega_F(t))}^{2}
	\Vert\bu\Vert_{L^s(\Omega_F(t))}^{\frac{2}{1-\alpha}}\Big)^{\frac{1-\alpha}{2}}
	\notag\\
	&\leq C\Big(\int_{0}^{T}
	\Vert\bv\Vert_{H^{1}(\Omega_F(t))}^{2}
	\Big)^{\frac{n}{2s}}
	\Big(\int_{0}^{T}
	\Vert\bv\Vert_{L^{2}(\Omega_F(t))}^{2}
	\Vert\bu\Vert_{L^s(\Omega_F(t))}^{\frac{2s}{s-n}}\Big)^{\frac{s-n}{2s}}
	\label{estimate2}.
	\end{align}
	Now \eqref{estimate1} and \eqref{estimate2} with
	\begin{equation*}
	r=\frac{2s}{s-n} \quad\Rightarrow\quad
	\frac{n}{s}+\frac{2}{r}=1
	\end{equation*}
	give
	\begin{multline*}
	\Big \vert\, \int_0^T\int_{\Omega_F(t)}\bv\cdot\nabla \bw\cdot \bu \,\,\Big \vert
	\leq C \Big ( \int_0^T\Vert\nabla \bw\Vert_{L^2(\Omega_F(t))}^2\, \Big )^{\frac{1}{2}}
	\Big ( \int_0^T
	\Vert \bv \Vert_{H^1(\Omega_F(t))}^2
	\,\Big )^{\frac{n}{2s}}
	\Big ( \int_0^T \Vert \bv\Vert_{L^2(\Omega_F(t))}^2 \Vert \bu\Vert_{L^s(\Omega_F(t))}^r\,\Big )^{\frac{1}{r}}.
	\end{multline*}
\end{proof}

\begin{proof}[Proof of Proposition \ref{EnergyEquality}]
Applying the test function $\bpsi=\bU_{2}^h$ in \eqref{transformed_weak2}, we get
	\begin{multline}	\label{EnergyEquality1}
	\int_0^t\int_{\Omega}\bU_2\cdot\partial_t\bU_2^h
	\,\d\bx_1\,\d\tau
	+ \int_0^t\int_{\Omega_{F}^1(\tau)} \big (\bu_1\otimes\bU_2:\nabla(\bU_2^h)^T
	-2\,\D\bU_2:\D\bU_2^h\,)
	\,\d\bx_1\,\d\tau
	\\
	-\int_0^t\int_{\Omega_{F}^1(\tau)} (\bU_2-\bu_1)\cdot\nabla\bU_2\cdot\bU_2^h
	\,\d\bx_1\,\d\tau
	-\int_{\Omega}\bU_2(t)\cdot\bU_2^h(t)\,\d\bx_1
	\\
	= -\int_{0}^{t}\langle F(\tau), \bU_2^h(\tau) \rangle \,\d\tau
	- \int_{\Omega}\bu_{0}\cdot \bU_2^h(0,\cdot )\,\d\bx_1
	\\
	+ \int_{0}^{t} \int_{S_1(\tau)} \widetilde{\bomega}\times \bU_{2}\cdot \bU_{2}^h-\bu_{1}\times\bU_{2}\cdot\bOmega_{2}^h\,\d\bx_1\,\d \tau,
	\end{multline}
	where $\bU_{2}^h$ denotes the regularization of $\bU_2$ defined by \eqref{RegDef3} and
	$$
	\bOmega_{2}^h(t)
	=\int_{-\infty}^{+\infty} j_h(t-s)\Q_{1}(t)\Q_{1}^T(s)\bOmega_{2}\,\d s.
	$$
	Then Lemma \ref{Reynolds_gen}, with $\bu=\bv=\bU_2$, implies
	\begin{align}
	\int_{0}^{t}&\int_{\Omega_F^1(\tau)}\bU_2\cdot\partial_t\bU_2^h
	\notag\\
	\to & -\frac{1}{2}\int_{0}^{t}\int_{\Omega_F^1(\tau)} \nabla(\bU_2\cdot\bU_2)\cdot\bu_1
	+ \frac{1}{2}\int_{\Omega_F^1(t)}\bU_2(t)\cdot \bU_2(t)\,\,
	- \frac{1}{2}\int_{\Omega_F}\bu_0\cdot \bu_0
	\notag\\
	&= -\int_{0}^{t}\int_{\Omega_F^1(\tau)} \bu_1\otimes\bU_2:\nabla(\bU_2)^T
	+ \frac{1}{2}\int_{\Omega_F^1(t)}\bU_2(t)\cdot \bU_2(t)\,\,
	- \frac{1}{2}\int_{\Omega_F}\bu_0\cdot \bu_0,
	\label{EETime_fluid}
	\end{align}
	when $h\to 0$, for almost every $t\in[0,T]$.	
	On the solid domain $S_{1}(\tau)$ we have
	\begin{align}
	\int_{0}^{t}&\int_{S_1(\tau)}\bU_2\cdot\partial_t\bU_2^h
	\notag\\
	\to &
	-\frac{1}{2}\int_{0}^{t}\int_{S_1(\tau)} \nabla(\bU_2\cdot\bU_2)\cdot\bu_1
	+ \frac{1}{2}\int_{S_1(t)}\bU_2(t)\cdot \bU_2(t)\,\,
	- \frac{1}{2}\int_{S_0}\bu_0\cdot \bu_0
	\notag\\
	&= -\int_{0}^{t}\int_{S_1(\tau)} (\nabla\bU_2^T\bU_2)\cdot\bu_1
	+ \frac{1}{2}\int_{S_1(t)}\bU_2(t)\cdot \bU_2(t)\,\,
	- \frac{1}{2}\int_{S_0}\bu_0\cdot \bu_0
	\notag\\
	&= \int_{0}^{t}\int_{S_1(\tau)} (\bOmega_2\times\bU_2)\cdot\bu_1
	+ \frac{1}{2}\int_{S_1(t)}\bU_2(t)\cdot \bU_2(t)\,\,
	- \frac{1}{2}\int_{S_0}\bu_0\cdot \bu_0
	\notag\\
	&= -\int_{0}^{t}\int_{S_1(\tau)} (\bu_1\times\bU_2)\cdot\bOmega_2
	+ \frac{1}{2}\int_{S_1(t)}\bU_2(t)\cdot \bU_2(t)\,\,
	- \frac{1}{2}\int_{S_0}\bu_0\cdot \bu_0
	\label{EETime_body}
	\end{align}
	when $h\to 0$, for a.e. $t\in[0,T]$, since
	$$
	\bU_2(t,\bx_{1})=\bA_2(t)+\bOmega_{2}(t)\times(\bx_{1}-\bq_{1}(t))
	\quad\Rightarrow\quad
	\nabla\bU_2^{T}\,\bx = -\bOmega_{2}\times\bx.
	$$
	From \eqref{EETime_fluid},\eqref{EETime_body} \eqref{estimateF2}
	and using Lemmas \ref{estimateOmega} and \ref{additionalLemma} we can pass to the limit in all terms in \eqref{EnergyEquality1}. We obtain
	\begin{multline}	\label{EELimit}
	-\int_{0}^{t}\int_{S_1(\tau)} \bu_1\times\bU_2\cdot\bOmega_2\,\d\bx_1\,\d\tau
	-\frac{1}{2}\int_{\Omega}\bU_2(t)\cdot \bU_2(t)\,\d\bx_1\,
	-\,\int_0^t\int_{\Omega_{F}^1(\tau)} 2\D\bU_2:\D\bU_2\,\d\bx_1\,\d\tau
	\\
	-\int_0^t\int_{\Omega_{F}^1(\tau)} (\bU_2-\bu_1)\cdot\nabla\bU_2\cdot\bU_2
	\,\d\bx_1\,\d\tau
	=\int_{0}^{t}\langle F(\tau), \bU_2(\tau) \rangle \,\d\tau
	- \frac{1}{2}\int_{\Omega}\bu_{0}\cdot \bU_2(0,\cdot )\,\d\bx_1
	\\
	+ \int_{0}^{t} \int_{S_1(\tau)} \underbrace{\widetilde{\bomega}\times \bU_{2}\cdot \bU_{2}}_{=0}-\bu_{1}\times\bU_{2}\cdot\bOmega_{2}\,\d\bx_1\,\d\tau
	,
	\end{multline}
	for almost every $t\in[0,T]$.
\end{proof}

\subsection{Closing the estimates}
Now we have all the ingredients to finish the proof of Theorem \ref{WeakStrong} by following the steps from the analogous proof in the Navier-Stokes case, see e.g. \cite{GaldiNS00}. Let us denote
$$
(\bu,p,\ba,\bomega)
= (\bu_1-\bU_{2},p_1-P_{2},\ba_1-\bA_{2},\bomega_1-\bOmega_{2}).
$$
Since $\bu_1$ is a weak solution, it satisfies the energy inequality
\begin{equation}	\label{EI1}
\frac{1}{2}\|\bu_1(t)\|_{L^2(\Omega)}^2
+ 2\int_{0}^{t}
\int_{\Omega_{F}^{1}(\tau)}\,|\D \bu_1|^{2}
\,\d\bx_1\,\d\tau
\leq
\frac{1}{2}\|\bu_0\|_{L^2(\Omega)}^2,
\quad\text{for a.e.} t\in[0,T].
\end{equation}
From Proposition \ref{EnergyEquality} it follows that the solution $\bU_2$ satisfies the following energy-type equality:
\begin{equation}	\label{EE2}
\begin{split}
&\frac{1}{2}\|\bU_2(t)\|_{L^2(\Omega)}^2
+ \int_{0}^{t}
\int_{\Omega_{F}^{1}(\tau)}\, (\bU_2-\bu_1)\cdot\nabla\bU_2\cdot\bU_2
\,\d\bx_1\,\d\tau
+  2\int_{0}^{t}\int_{\Omega_{F}^{1}(\tau)}\,|\D\bU_2|^{2}\,\d\bx_1\,\d\tau
\\
&\qquad
= \,
-\int_{0}^{t}\langle F(\tau),\bU_2(\tau) \rangle \,\d\tau
+ \frac{1}{2}\|\bu_0\|_{L^2(\Omega)}^2,
\end{split}
\end{equation}
for a.e. $t\in[0,T]$.
We take the test function $\bU_2^h$ in the weak formulation for $\bu_1$ \eqref{weak2}:
\begin{equation}	\label{EqU1}
\begin{split}
&-\int_{0}^{t}\int_{\Omega}\bu_1\cdot\partial_t\bU_2^h \,\,\d\bx_1\,\d\tau
- \int_{0}^{t}\int_{\Omega_{F}^{1}(\tau)}\,\big(\bu_1\otimes\bu_1:\nabla(\bU_2^h)^{T}-\,2\D\bu_1:\D\bU_2^h\,\big)\,\d\bx_1\,\d\tau
\\
&\qquad\qquad\qquad\qquad\qquad\qquad\qquad\qquad
+\int_{\Omega}\bu_1(t)\cdot\bU_2^h(t) \,\d\bx_1
= \int_{\Omega}\bu_{0}\cdot\bU_2^h(0,\cdot )\,\d\bx_1,
\end{split}
\end{equation}
where $\bu^h$ denotes the regularization of $\bu$ defined by \eqref{RegDef3}.
Then we take the test function $\bu_1^h$ in the weak formulation for $\bU_2$ \eqref{transformed_weak2}:
\begin{multline}	\label{EqU2Tmp}
-\int_0^t\int_{\Omega}\bU_2\cdot\partial_t\bu_1^h \,\d\bx_1\,\d\tau
- \int_0^t\int_{\Omega _{F}^{1}(\tau)} \big (\bu_1\otimes\bU_2:\nabla(\bu_1^h)^T
-\,2\D\bU_2:\D\bu_1^h\,\big)\,\d\bx_1\,\d\tau
\\
+ \int_0^t\int_{\Omega _{F}^{1}(\tau)} (\bU_2-\bu_1)\cdot\nabla\bU_2\cdot\bu_1^h \,\,\d\bx_1\,\d\tau
+ \int_{\Omega}\bU_2(t)\cdot\bu_1^h(t)\,\d\bx_1
\\
= -\int_{0}^{t}\langle F(\tau),\mathbf{u}_1^h(\tau) \rangle \,\d\tau
+ \int_{\Omega}\bu_{0}\cdot\bu_1^h(0,\cdot )\,\d\bx_1
\\
- \int_{0}^{t} \int_{S_1(\tau)} \widetilde{\bomega}\times \bU_{2}\cdot \bu_1^h-\bu_{1}\times\bU_{2}\cdot\bomega_1^h\,\d\bx_1\,\d\tau.
\end{multline}
Now $\eqref{EqU1}+\eqref{EqU2Tmp}$ gives
\begin{multline}	\label{EqU2Tmp22}
- \int_0^t\int_{\Omega}\big (\bU_2\cdot\partial_t\bu_1^h + \bu_1\cdot\partial_t\bU_2^h\big)\,\d\bx_1\,\d\tau
\\
- \int_0^t\int_{\Omega_{F}^{1}(\tau)}\big (\bu_1\otimes\bU_2:\nabla(\bu_1^h)^T+ \bu_1\otimes\bu_1:\nabla(\bU_2^h)^T\big)\,\d\bx_1\,\d\tau
\\
+ \,\int_0^t\int_{\Omega_{F}^{1}(\tau)}\big(2\D\bU_2:\D\bu_1^h + 2\D\bu_1:\D\bU_2^h\,\big)\,\d\bx_1\,\d\tau
\\
+ \int_0^t\int_{\Omega_{F}^{1}(\tau)} (\bU_2-\bu_1)\cdot\nabla\bU_2\cdot\bu_1^h \,\d\bx_1\,\d\tau
+ \int_{\Omega}\big(\bU_2(t)\cdot\bu_1^h(t) + \bu_1(t)\cdot\bU_2^h(t)\big) \,\d\bx_1
\\
=
-\int_{0}^{t}\langle F(\tau),\mathbf{u}_1^h(\tau) \rangle \,\d\tau
+\int_{\Omega}\bu_{0}\cdot(\bu_1^h(0,\cdot )+\bU_2^h(0,\cdot )) \,\d\bx_1
\\
- \int_{0}^{t} \int_{S_1(\tau)} \widetilde{\bomega}\times \bU_{2}\cdot \bu_1^h-\bu_{1}\times\bU_{2}\cdot\bomega_1^h\,\d\bx_1\,\d \tau.
\end{multline}
Let $h\to 0$. Lemma \ref{Reynolds_gen} implies
\begin{align}
\int_0^t\int_{\Omega_{F}^{1}(\tau)}&\big (\bU_2\cdot\partial_t\bu_1^h + \bu_1\cdot\partial_t\bU_2^h\big)
\notag\\
\to &
- \int_0^t\int_{\Omega_{F}^{1}(\tau)} \nabla(\bU_{2}\cdot\bu_{1})\cdot\bu_{1}
+ \int_{\Omega_F^1(t)}\bU_2(t)\cdot \bu_1(t)\,\,
- \int_{\Omega_F^1}\bu_0\cdot \bu_0
\notag\\
&= - \int_0^t\int_{\Omega_{F}^{1}(\tau)}\big (\bu_1\otimes\bU_2:\nabla(\bu_1)^T+ \bu_1\otimes\bu_1:\nabla(\bU_2)^T\big)
\\
&\qquad
+ \int_{\Omega_F^1(t)}\bU_2(t)\cdot \bu_1(t)\,\,
- \int_{\Omega_F}\bu_0\cdot \bu_0,
\qquad\text{for a.e. } t\in[0,T],
\notag
\end{align}
on the fluid domain, and
\begin{align}
\int_0^t\int_{S_1(\tau)}&\big (\bU_2\cdot\partial_t\bu_1^h + \bu_1\cdot\partial_t\bU_2^h\big)
\notag\\
&\to
- \int_0^t\int_{S_1(\tau)} \nabla(\bU_{2}\cdot\bu_{1})\cdot\bu_{1}
+ \int_{S_1(t)}\bU_2(t)\cdot \bu_1(t)\,\,
- \int_{S_0}\bu_0\cdot \bu_0
\notag\\
&= -\int_{0}^{t}\int_{S_1(\tau)} (\nabla\bU_2^T\bu_1 + \nabla\bu_1^T\bU_2)\cdot\bu_1
+ \int_{S_1(t)}\bU_2(t)\cdot \bu_1(t)\,\,
- \int_{S_0}\bu_0\cdot \bu_0
\notag\\
&= \int_{0}^{t}\int_{S_1(\tau)} (\bOmega_2\times\bu_1+\bomega_1\times\bU_2)\cdot\bu_1
+ \int_{S_1(t)}\bU_2(t)\cdot \bu_1(t)\,\,
- \int_{S_0}\bu_0\cdot \bu_0
\notag\\
&=
-\int_{0}^{t}\int_{S_1(\tau)}
( \underbrace{(\bu_1\times\bu_1)}_{=0}\cdot\bOmega_2 + (\bu_1\times\bU_2)\cdot\bomega_1 )
\\
&\qquad
+ \int_{S_1(t)}\bU_2(t)\cdot \bu_1(t)\,\,
- \int_{S_0}\bu_0\cdot \bu_0,
\qquad\text{for a.e. } t\in[0,T],
\notag
\end{align}
on the solid domain.
From above, by using \eqref{estimateF2} and Lemmas \ref{estimateOmega} and \ref{additionalLemma}, we get	
\begin{multline}	\label{EqU2}
\int_{0}^{t}\int_{\Omega_F^{1}(\tau)} 4\,\D\bu_1:\D\bU_2  \,\d\bx_1 \,\d\tau
+ \int_{0}^{t}\int_{\Omega_F^{1}(\tau)} (\bU_2-\bu_1)\cdot\nabla\bU_2\cdot\bu_1 \,\d\bx_1 \,\d\tau
\\
+ \int_{\Omega} \,\bu_1(t)\cdot\bU_2(t)\,\d\bx_1
= -\int_{0}^{t}\langle F(\tau),\bu_1(\tau) \rangle \,\d\tau
+ \int_{\Omega} \,|\bu_0|^{2} \,\d\bx_1\,
\\
- \int_{0}^{t} \int_{S_1(\tau)} \widetilde{\bomega}\times \bU_{2}\cdot \bu_1 \,\d\bx_1\,\d \tau,
\end{multline}
for a.e. $t\in[0,T]$.
Now we take $\eqref{EI1}+\eqref{EE2}-\eqref{EqU2}$:
\begin{equation} \label{estimate}
\begin{split}
& \frac{1}{2}\|\bu(t)\|_{L^2(\Omega)}^2
+ 2\,\int_{0}^{t}\int_{\Omega_F^{1}(\tau)} |\D\bu|^{2} \,\d\bx_1 \,\d\tau
+ \int_{0}^{t}\int_{\Omega_F^{1}(\tau)} \bu\cdot\nabla\bU_2\cdot\bu \,\d\bx_1 \,\d\tau
\\
&\leq -\int_{0}^{t}\langle F(\tau),\bu(\tau)\rangle \,\d\tau \,
+ \int_{0}^{t} \int_{S_1(\tau)} \widetilde{\bomega}\times \bU_{2}\cdot \bu \,\d\bx_1\,\d \tau,
\end{split}
\end{equation}
for a.e. $t\in[0,T]$.
By integration by parts the last term on the left side of this identity is written as
\begin{equation}	\label{int}
\int_{0}^{t}\int_{\Omega _{F}^{1}(\tau)}\bu\cdot \nabla \bU_{2}\cdot \bu
=-\int_{0}^{t}\int_{\Omega _{F}^{1}(t)}
\bu\cdot \nabla \bu\cdot \bU_{2}
\end{equation}
and it can be estimated by using Lemma \ref{additionalLemma} (for $\bv=\bw=\bu_1-\bU_{2}$ and $\bu=\bU_{2}$):
\begin{align*}
\Big \vert\, \int_0^t\int_{\Omega_{F}^{1}(t)}\bu\cdot\nabla \bu\cdot \bU_2 \,\,\Big \vert
&\leq C \Big ( \int_0^t
\Vert\bu\Vert_{H^1}^2
\, \Big )^{1-\frac{1}{r}}
\Big ( \int_0^t \Vert \bu\Vert_{L^2}^2 \Vert \bU_{2}\Vert_{L^s}^r\,\Big )^{\frac{1}{r}}\\
&\leq \varepsilon  \int_0^t
\Vert \bu\Vert_{H^1}^2
+ C \int_0^t \Vert \bu\Vert_{L^2}^2 \Vert \bU_{2}\Vert_{L^s}^r
\\
&\leq \varepsilon  \int_0^t
\Vert\D \bu\Vert_{L^2}^2
+ C \int_0^t \Vert \bu\Vert_{L^2}^2 (1+\Vert \bU_{2}\Vert_{L^s}^r).
\end{align*}
The right-hand side of \eqref{estimate} can be estimated by using Lemma \ref{estimateOmega} and \eqref{estimateF_Gronwall}.

By putting all estimates together we conclude:
\begin{multline*}
\frac{1}{2}\|\bu(t)\|_{L^2(\Omega)}^2
+ \,2\int_0^t\int_{\Omega_F(\tau)} |\D(\bu)|^{2}
\leq  C\int_{0}^{t}\Vert
\mathbf{u}(\tau)\Vert _{L^{2}(\Omega )}^{2}\big (1+\Vert \mathbf{U}_{2}(\tau)\Vert _{L^{s}(\Omega
	_{F}^{1}(\tau))}^{r}
	+\Vert \tau\nabla P_2\Vert_{L^{q}(\Omega _{F}^{1}(\tau))}^{p}
)\\
+ \varepsilon\,\int_0^t\int_{\Omega_F(\tau)} |\D\bu|^{2},
\qquad\text{for a.e. } t\in[0,T].
\end{multline*}
Taking $\varepsilon=2$ we get
\begin{equation*}
\Vert \mathbf{u}(t)\Vert _{L^{2}(\Omega )}^{2}\leq C\int_{0}^{t}\Vert
\mathbf{u}(\tau)\Vert _{L^{2}(\Omega )}^{2}\big (1+\Vert \mathbf{U}_{2}(r)\Vert _{L^{s}(\Omega
	_{F}^{1}(\tau))}^{r} 
	+\Vert \tau\nabla P_2\Vert_{L^{q}(\Omega _{F}^{1}(\tau))}^{p}
)\ \d \tau,
\qquad\text{for a.e. } t\in[0,T].
\end{equation*}
Now by the integral Gronwall's inequality we conclude that $\bu = 0${,
	that is 
	$$
	\bu_1=\bU_{2}=\nabla\widetilde{\bX}_1\bu_2,\quad\text{in } (0,T)\times\Omega.
	$$
	It remains to show that $\bB_{1}=\bB_{2}$ and $\bu_{1}=\bu_{2}$.
	Since
	$$
	\bu_1 =\ba_{1}+\bomega_1\times(\bx_1-\bq_1)
	\quad\text{and}\quad
	\bU_2 =\bA_{2}+\bOmega_2\times(\bx_1-\bq_1)
	$$
	on $S_1(t)$, it follows that 
	\begin{align}
	\ba_{1} &= \bA_{2} 
	= \Q^{T}\ba_{2}
	= \Q_1\Q_2^{T}\ba_{2}
	\quad\Leftrightarrow\quad 
	\Q_1^{T}\ba_{1} = \Q_2^{T}\ba_{2},
	\label{transl_vel}
	\\
	\bomega_{1} &= \bOmega_{2}
	= \Q^{T}\bomega_{2}
	= \Q_1\Q_2^{T}\bomega_{2}
	\quad\Leftrightarrow\quad 
	\Q_1^{T}\bomega_{1} = \Q_2^{T}\bomega_{2}.
	\label{rot_vel}
	\end{align}
	Now, 
	\eqref{rot_vel} and
	\begin{align*}
	\Q_i^{T}\bomega_{i} \times\bx
	&= \Q_i^{T}(\bomega_{i} \times\Q_i\bx)
	= \Q_i^{T}\P_i\Q_i\bx
	= \Q_i^{T}\Q_i'\Q_i^{T}\Q_i\bx
	= \Q_i^{T}\Q_i'\bx,
	\quad\forall\bx\in\R^3.
	\end{align*}
	give
	$$
	\Q_1^{T}\Q_1' = \Q_2^{T}\Q_2'
	\quad\Leftrightarrow\quad
	\Q_1' = \Q_1^{T}\Q_2^{T}\Q_2'
	$$
	which implies that
	$$
	\Q_1'-\Q_2' = \Q_1^{T}\Q_2^{T}\Q_2' - \Q_2'
	= (\Q_1-\Q_2)\Q_2^{T}\Q_2.
	$$
	We conclude that $\Q_{\Delta} = \Q_1-\Q_2$ is the solution to the problem
	$$
	\left\lbrace 
	\begin{array}{l}
	\frac{d}{dt} \Q_{\Delta}=\Q_{\Delta}\mathbb{W}\\
	\Q_{\Delta}(0) =0,
	\end{array}
	\right.
	$$
	where $\mathbb{W} = \Q_2^{T}\Q_2'$. 
	Since the above problem has a unique solution $\Q_{\Delta}=0$, it follows that $\Q_1=\Q_2.$
	Now, \eqref{transl_vel} implies that 
	$\ba_{1}=\ba_{2}$ so we conclude that
	$\bB_1=\bB_2$ and $\widetilde{\bX}_1=\widetilde{\bX}_2=\I$, that is $\bu_{1}=\bu_{2}$.	
}

\section{Slip boundary condition}
\label{Sec:slip}
Since the theory for weak and strong solutions for the fluid-rigid body problem with slip coupling condition is already developed, the weak-strong uniqueness can be proved along the same line as in the no-slip case. Therefore, here we just formulate the result and outline the differences. First, recall the definition of the problem \eqref{FSINoslip}$_{\rm slip}$:

\noindent
Find $(\mathbf{u},p,\mathbf{q},{\mathbb{Q}})$ such that
\begin{equation*}
\left.
\begin{array}{l}
\partial _{t}\mathbf{u}+(\mathbf{u}\cdot \nabla )\mathbf{u}=\mbox {div}
\left( {\mathbb{T}}(\mathbf{u},p)\right) , \\
\mbox {div }\mathbf{u}=0
\end{array}
\right\} \;\mathrm{in}\;\Omega _{F}(t)\times (0,T),
\end{equation*}
\begin{equation*}
\left.
\begin{array}{l}
\frac{d^{2}}{dt^{2}}\mathbf{q}=-\int_{\partial S(t)}{\mathbb{T}}(\mathbf{u}
,p)\mathbf{n}\,d\mathbf{\gamma }(\mathbf{x}), \\
\frac{d}{dt}(\mathbb{J}\boldsymbol{\omega })=-\int_{\partial S(t)}(\mathbf{x}-\mathbf{
	q}(t))\times {\mathbb{T}}(\mathbf{u},p)\mathbf{n}\,d\mathbf{\gamma }(\mathbf{
	x})
\end{array}
\right\} \;\mathrm{in}\;(0,T),
\end{equation*}
\begin{equation*}
(\mathbf{u}-\mathbf{u}_{s})\cdot \mathbf{n}=0,\qquad \beta (\mathbf{u}_{s}-
\mathbf{u})\cdot \boldsymbol{\tau }={\mathbb{T}}(\mathbf{u},p)\mathbf{n}
\cdot \boldsymbol{\tau }\qquad \mathrm{on}\;\partial S(t),
\end{equation*}
\begin{equation*}
\mathbf{u}(0,.)=\mathbf{u}_{0}\qquad \mathrm{in}\;\Omega ;\qquad \mathbf{q}
(0)=\mathbf{q}_{0},\qquad \mathbf{q}^{\prime }(0)=\mathbf{a}_{0},\qquad
\boldsymbol{\omega }(0)=\boldsymbol{\omega }_{0}.
\end{equation*}
The existence of a weak solution to system \eqref{FSINoslip}$_{\rm slip}$ was proven in \cite{ChemetovSarka}. Here we just briefly recall the definition of a weak solution.
We define the function spaces for the weak formulation:
\begin{equation*}
V^{0,2}(\Omega )=\{\mathbf{v}\in L^{2}(\Omega ):\ \,\mbox{div }\mathbf{v}
=0\;\ \text{ in}\;\mathcal{D}^{\prime }(\Omega ),\qquad \mathbf{v}\cdot
\mathbf{n}=0\;\ \text{ in}\;H^{-1/2}(\partial \Omega )\},
\end{equation*}
\begin{equation*}
BD_{0}(\Omega )=\left\{ \mathbf{v}\in L^{1}(\Omega ):\ \,\mathbb{D}\mathbf{v}
\in \mathcal{M}(\Omega ),\qquad \,\mathbf{v}=0\;\ \ \mbox{   on  }\;\partial
\Omega \right\},
\end{equation*}
where $\mathcal{M}(\Omega )$ is the space of bounded Radon measures,
$$
KB(S) =\left\{ \mathbf{v}\in BD_{0}(\Omega ):\,\ \mathbb{D}\mathbf{v}\in
L^{2}(\Omega \backslash \overline{S}),\; \mathbb{D}\mathbf{v}=0\ \ \text{
	a.e. on }S,\;\mbox{div}\,\mathbf{v}=0
\;{\rm in}\; \mathcal{D}^{\prime }(\Omega )\right\},
$$
where $S\subset\Omega$ is an open connected set with boundary $\partial S\in C^{2}$.

\begin{definition}
	\label{definitionSlip} The pair $(\mathbf{B},\mathbf{u}) $ is a
	weak solution to the system \eqref{FSINoslip}$_{\rm slip}$ if the following conditions
	are satisfied:
	
	1) The function $\mathbf{B}(t,\cdot ):\mathbb{R}^{3}\rightarrow \mathbb{R}
	^{3}$ \ is a orientation preserving isometry \eqref{is}, which defines a
	time-dependent set $S(t)$ by \eqref{set}. {
		The isometry $\bB$ is compatible with $\bu=\bu_S$ on $S(t)$ in the following sense: the rigid part of velocity $\bu$, denoted by $\bu_S$, satisfy condition \eqref{RigidVelocity}, and 
		$\bq,\; \Q$ are absolutely continuous on $\left[ 0,T\right]$ and satisfy \eqref{om}.
	}
	
	2) The function $\mathbf{u}\in L^{2}(0,T;KB(S(t)))\cap L^{\infty
	}(0,T;V^{0,2}(\Omega ))$ satisfies the integral equality
	{		
		\begin{multline} \label{Weak}
		\int_{0}^{T} \int_{S(t)}\mathbf{u}\boldsymbol{
			\psi }_{t}\, \d\mathbf{x}\,\d t
		+\int_{0}^{T} \int_{\Omega_F(t)}\{\mathbf{u}\boldsymbol{
			\psi }_{t}+(\mathbf{u}\otimes \mathbf{u}) :\mathbb{D}\boldsymbol{\psi }
		-2\,\mathbb{D}\mathbf{u}:\mathbb{D}\boldsymbol{\psi }\,\}\d\mathbf{x}\,\d t
		\notag \\
		-\int_{\Omega }\mathbf{u}(T,\cdot)\boldsymbol{\psi }(T,\cdot )\,\d\mathbf{x}
		=-\int_{\Omega }\mathbf{u}_{0}\boldsymbol{\psi }(0,\cdot )\,\d\mathbf{x}
		+\int_{0}^{T}\int_{\partial S(t)}\beta (\mathbf{u}_{s}-\mathbf{u}_{f})(
		\boldsymbol{\psi }_{s}-\boldsymbol{\psi }_{f})\,d\mathbf{\gamma }\,\d t,
		\end{multline}
		which holds for any test function $\boldsymbol{\psi }$ such that
		\begin{eqnarray}
		\boldsymbol{\psi } &\in &L^{4}(0,T;KB(S(t))),  \notag \\
		\boldsymbol{\psi }_{t} &\in &L^{2}(0,T;L^{2}(\Omega \backslash \partial
		S(t))).  
		\label{testSlip}
		\end{eqnarray}
	}
	By $\mathbf{u}_{s}(t,\mathbf{\cdot }),$ $\boldsymbol{\psi }_{s}(t,\mathbf{
		\cdot })$ and $\mathbf{u}_{f}(t,\mathbf{\cdot }),$ $\boldsymbol{\psi }_{f}(t,
	\mathbf{\cdot })$\ \ we denote the trace values of $\mathbf{u},$ $
	\boldsymbol{\psi }$ \ on $\partial S(t)$\ \ from \ the "\textit{rigid}" side
	$S(t)$ and the "fluid" side $F(t)$, respectively.
	
	3) The  energy inequality
		\begin{equation*}
		\frac{1}{2}\|\bu(t)\|_{L^2(\Omega)}^2
		+ 2\int_{0}^{t}
		\int_{\Omega_{F}(\tau)}\,|\D \bu|^{2}
		\,\d\bx\,\d\tau
		+ \int_{0}^{t}
		\int_{\partial S(\tau)}\,\beta|\bu-\bu_s|^{2}
		\,\d\bx\,\d\tau
		\leq
		\frac{1}{2}\|\bu_0\|_{L^2(\Omega)}^2
		\end{equation*}
		holds for almost every $t\in(0,T)$.
\end{definition}

\begin{teo}\label{WeakStrongslip}
	Let $(\bu_{1},\bB_{1})$ and $(\bu_{2},\bB_{2})$ be two weak solutions  corresponding to Definition \ref{definitionSlip} with the  same data.
	Assume that $d(S_i(t),\partial\Omega)>\delta_i$, $i=1,2$, for some constants $\delta_i>0$.
	If $\bu_{2}$ satisfies the following condition:
	\begin{equation}\label{LrLs2}
	\bu_{2}\in L^r(0,T;L^s(\Omega))\quad
	\text{ for some } s,r \text{ such that }\quad
	\frac{3}{s}+\frac{2}{r}=1,\, s\in(3,+\infty)
	\end{equation}
	then
	\begin{equation*}
	(\bu_{1},\bB_{1}) = (\bu_{2},\bB_{2}).
	\end{equation*}
\end{teo}
As in the no-slip case, the first step is to show that the transformed solution $(\mathbf{U}_{2},P_{2},\mathbf{A}_{2},\boldsymbol{\Omega }_{2})$ satisfies the following equality:
\begin{multline*}
\int_{0}^{T}\int_{\Omega}\bU_2\cdot\partial_t\boldsymbol{\psi }\, \d\bx_1\,\d t
+\int_{0}^{T}\int_{\Omega_F^1(t)}\Big((\bu_1\otimes \bU_2) :\nabla\boldsymbol{\psi }^T
-(\bU_{2}-\bu_{1})\cdot \nabla \bU_{2}\cdot \boldsymbol{\psi }\,\Big)\, \d\bx_1\, \d t\\
-2\int_{0}^{T}dt\int_{\Omega \backslash \partial S_{1}(t)}
\mathbb{D}\mathbf{U}_{2}:\mathbb{D}\boldsymbol{\psi }\,d\mathbf{x}_{1}
= \int_{0}^{T} \langle F(t),\boldsymbol{\psi }(\tau)\rangle\, \d t
-\int_{\Omega }\mathbf{u}_{0}\boldsymbol{\psi }(0,\cdot )\,d\mathbf{x}_{1}\\
+ \int_{0}^{T}dt\left\{
\int_{\partial S_{1}(t)}\beta (\mathbf{U}_{s}^{2}-\mathbf{U}_{2})\cdot (
\boldsymbol{\psi }_{s}-\boldsymbol{\psi }_{f})\,d\mathbf{\gamma }(\mathbf{x}
_{1})\right\}
\\
+\int_{0}^{T}(\widetilde{\boldsymbol{\omega }}\times (J_{1}\boldsymbol{
	\Omega }_{2})\cdot \boldsymbol{\psi }_{\omega }+\widetilde{\boldsymbol{
		\omega }}\times \mathbf{A}_{2}\cdot \boldsymbol{\psi }_{h})\ dt,
\label{StrongSlip}
\end{multline*}
which holds for any test function $\boldsymbol{\psi }$ \ satisfying 
\eqref{testSlip}. Let us note that this function $\boldsymbol{\psi }$ is rigid on $S_{1}(t)$, that is,
\begin{equation*}
\boldsymbol{\psi }(t,\mathbf{x})=\boldsymbol{\psi }_{h}(t)+\boldsymbol{\psi }
_{\omega }\times (\mathbf{x}-\mathbf{q}_{1}(t))\qquad \text{for }\mathbf{x}
\in S_{1}(t).
\end{equation*}

The proof of the above claim is analogous to the proof in the no-slip case (see Proposition \ref{WeakU2} and Remark \ref{weak_transformed}). We just have to see how the term corresponding to the slip condition transforms. We have

\begin{displaymath}
\bU(t,\bx_1) = \nabla\tilde{\bX}_1(t,\tilde{\bX}_2(t,\bx_1))\,
\bu(t,\tilde{\bX}_2(t,\bx_1)),
\end{displaymath}
i.e.
\begin{displaymath}
\bu(t,\bx_2) = \nabla\tilde{\bX}_2(t,\tilde{\bX}_1(t,\bx_2))\,\bU(t,\tilde{\bX}_1(t,\bx_2)),
\end{displaymath}
\begin{equation*}
\bphi(t,\bx_2) = \nabla\tilde{\bX}_1(t,\bx_2)^T\bpsi(t,\tilde{\bX}_1(t,\bx_2)),
\end{equation*}
and
\begin{displaymath}
\nabla\tilde{\bX}_1(t,\tilde{\bX}_2(t,\bx_1)) = \Q^T(t),
\quad
\nabla\tilde{\bX}_2(t,\bx_1) = \Q(t)
\quad \quad \mathrm{on}\;\partial S_{1}(t)
\end{displaymath}
implies
\begin{displaymath}
\mathbf{u}_{s}^{2}-\mathbf{u}_{2}
=\Q (\mathbf{U}_{s}^{2}-\mathbf{U}_{2})\quad \quad \mathrm{on}\;\partial S_{1}(t),
\end{displaymath}
hence
\begin{displaymath}
(\mathbf{u}_{s}^{2}-\mathbf{u}_{2})\cdot(\bphi_s-\bphi_f)
=\Q (\mathbf{U}_{s}^{2}-\mathbf{U}_{2})\cdot\Q(\bpsi_s-\bpsi_f)
= (\mathbf{U}_{s}^{2}-\mathbf{U}_{2})\cdot(\bpsi_s-\bpsi_f)
\end{displaymath}
on $\partial S_{1}(t)$.

{Existence of an associated pressure can be proved along the same lines as in Section\ref{Sec:Pressure} with the difference that on the fixed domain we apply the results from \cite{KNN} instead of \cite{Neu}. The local-in-time regularity of the pressure is proved in the same way as in Section \ref{Sec:pressurereg} by using the results from \cite{bravin2018weak} instead of \cite{GlassSueur} and \cite{GGH13}.
	More precisely, the following proposition holds.
	
\begin{proposition} \label{pressure_regularity_slip}
		Let $(\bu_2,p_2)$ be a weak solution of the problem \eqref{FSINoslip}$_{\rm slip}$, and let $\bu_2$ satisfy the Prodi-Serrin condition. Then the following regularity result holds:
		$$
		t\bu_{2}\in L^p(0,T;W^{2,q}(\Omega_F^2(t))),
		\quad
		t\partial_{t}\bu_{2}, t\nabla p_{2}\in L^p(0,T;L^{q}(\Omega_F^2(t))),
		$$
		for $\frac{1}{q} = \frac{1}{2} + \frac{1}{s}$, $\frac{1}{p} = \frac{1}{2} + \frac{1}{r}.$
\end{proposition}	
	
The proof of Proposition \ref{pressure_regularity} for the slip case is analogous to the proof of
\cite[Lemma 3]{bravin2018weak}, which is given for the two-dimensional case but also holds true for the three-dimensional, by the similar adaptation as in the no-slip case (Section \ref{Sec:pressurereg}). The main ingredient is the following lemma, which is analogous to \cite[Theorem 4]{bravin2018weak} and is based on the maximal regularity result for Stokes-rigid body system, together with a fixed point argument.	
	\begin{lemma} \label{strongStokes_slip}
		Let $\mathbf{g}\in L^{p}(0,T;L^{q}(\Omega_F^2(t)))$ and $ \mathbf{g}_1,\mathbf{g}_2\in L^{p}(0,T)$.
		There exists a unique solution to the system 
		\begin{equation*}	\label{Stokes_slip}
		\begin{array}{l}
		\left.
		\begin{array}{l}
		\partial_{t}\bv - \Delta\bv + \nabla p
		= \mathbf{g} , \\
		\divg\bv = 0
		\end{array}
		\right\} \;\mathrm{in}\;\bigcup_{t\in(0,T)} \{t\}\times\Omega_{F}(t),
		\\
		\left.
		\begin{array}{l}
		{\frac{d}{dt}\ba} = -\int_{\partial S(t)}{\T}(\bv,p) \bn\,\d\bgamma(\bx) 
		{\,+\, \mathbf{g}_1}, \\
		\frac{d}{dt}(\J\bomega) = -\int_{\partial S(t)}(\bx-\bq(t))\times {\T}(\bv,p)\bn\,\d\bgamma(\bx) 
		{\,+\, \mathbf{g}_2}
		\end{array}
		\right\} \;\mathrm{in}\;(0,T),
		\\
		(\mathbf{v}-\mathbf{v}_{S})\cdot \mathbf{n}=0,\qquad \beta (\mathbf{v}_{S}-
		\mathbf{v})\cdot \boldsymbol{\tau }={\mathbb{D}}(\mathbf{v})\mathbf{n}
		\cdot \boldsymbol{\tau },\quad \mathrm{on}\;\bigcup_{t\in(0,T)} \{t\}\times\partial S(t),
		\\
		\bv = 0 \quad \mathrm{on}\, \partial \Omega,
		\end{array}
		\end{equation*}	
		with vanishing initial data which satisfies
		$$
		\bv\in L^p(0,T;W^{2,q}(\Omega_F^2(t))),
		\quad
		\partial_{t}\bv, \nabla {p}\in L^p(0,T;L^{q}(\Omega_F^2(t))).
		$$
	\end{lemma}	
}
The rest of the proof of the weak-strong uniqueness 
follows the proof of the no-slip case.
All the estimates for the additional term
$$
\int_{0}^{T}dt\left\{
\int_{\partial S_{1}(t)}\beta (\mathbf{U}_{s}^{2}-\mathbf{U}_{2})\cdot (
\boldsymbol{\psi }_{s}-\boldsymbol{\psi }_{f})\,d\mathbf{\gamma }(\mathbf{x}
_{1})\right\}
$$
are the same as in \cite{chemetov2017weak}.

\appendix
\section{Appendix}

\subsection{Local transformation}
\label{Sec:LT}

In the proof of Theorem \ref{WeakStrong}, since fluid domains of the strong and the weak solution are a priori different, we transform the problem into a common domain.
We use the transformation presented in \cite{chemetov2017weak} to transform a strong solution to the domain of a weak solution, which is a moving domain, in a way that preserves the divergence-free condition.
It is defined by a transformation to a fixed domain as in \cite{Takahashi03} or \cite{GGH13}, which we also need for the construction of regularization. Even though this transformation is by now standard in the literature, here we briefly describe this transformation and recall its main properties for the convenience of the reader and to establish the notation that is used throughout the paper.	

According to \cite{GGH13,Takahashi03} we can define a transformation $\bX(t):\Omega
\rightarrow \Omega $ as the unique solution of the system
\begin{equation*}
\frac{d}{dt}\bX(t,\by)=\Lambda (t,\bX(t,\by)),
\qquad \bX(0,\by)=\by,
\qquad \forall \ \by\in\Omega .
\end{equation*}
where the velocity of change of coordinates $\Lambda(t,\bx)$ is a vector field that is smooth in the space variables and divergence-free, and satisfies $\Lambda =\mathbf{a}(t)+\bomega(t)\times (\bx-\mathbf{q}(t))$ in a neighborhood of $S(t)$ and $\Lambda =0$ in a neighborhood of $\partial \Omega $.

Note that the function $\Lambda$ is a divergence-free extension of the function $S(t)\ni\bx\mapsto\mathbf{a}(t)+\bomega(t)\times (\bx-\mathbf{q}(t))$ to the set $\Omega$. 
	The construction of the extension $\Lambda$ is given in \cite[Section 3]{GGH13} with little correction to the cut-off function $\chi$, where instead of balls $\overline{S(t)}\subset B_1\subset B_2$, we choose open sets $K_1,K_2$ such that $\overline{S(t)}\subset K_1\subset K_2\subset \Omega$. 
	We denote 
	$$
	\Lambda=:\mathrm{Ext}(\mathbf{a}+\bomega\times (\bx-\mathbf{q})).
	$$

Here, we assume that $\ba, \bomega\in L^{\infty}(0,T)$, which is slightly different from assumptions in \cite{GGH13,Takahashi03}. Therefore, for existence and uniqueness of solution $\bX$ we need Carath\'eodory's theorem (see e.g. \cite{roubivcek2013nonlinear}, Theorem 1.45) instead of the Picard-Lindel\"of theorem.

For all $t\in[0,T]$, the defined transformation $\bX(t)$ is a $C^{\infty}$ diffeomorphism and the derivatives
\begin{equation}	\label{partialDerivatives}
\frac{\partial^{|\alpha|+i}\bX}{\partial t^{i}\partial\by^{\alpha}},\qquad
i=0,1\ ,\quad \alpha\in\N_0^3,
\end{equation}
exist and are bounded.

We denote by $\bY$ the inverse of $\bX$, i.e.
\begin{equation*}
\bY(t,\cdot )=\bX(t,\cdot )^{-1}.
\end{equation*}
It satisfies the system of differential equations
\begin{equation*}
\frac{d}{dt}\bY(t,\bx)=\Lambda^{(\bY)} (t,\bY(t,\bx)),
\qquad \bY(0,\bx)=\bx,
\qquad \forall \ \bx\in\Omega ,
\end{equation*}
where
\begin{equation*}
\Lambda^{(\bY)} (t,\by) = -\nabla{\bX}(t,\by)^{-1}\Lambda(t,\bX(t,\by))
\end{equation*}
Note that $\bY$ possesses the same space and time regularity as $\bX$. Furthermore, $\bX$ and $\bY$ satisfy
\begin{equation*}
\nabla{\bX}(t,\by)\nabla{\bY}(t,\bX(t,\by)) = \mathrm{id}
\end{equation*}
and are volume-preserving, i.e.
\begin{equation}	\label{volumePreserving}
\det \nabla{\bX}(t,\by) = \det \nabla{\bY}(t,\bx) = 1,
\end{equation}
since $\divg\Lambda = 0$.

Then, by Proposition 2.4. in \cite{inoue1977existence}, the transformation of the velocity
\begin{equation*}
\bU(t,\by)
=\nabla{\bY}(t,\bX(t,\by))\bu(t,\bX(t,\by))
\end{equation*}
preserves the divergence, i.e.
\begin{equation*}
\divg_{\by}\bU(t,\by)
=\divg_{\bx}\bu(t,\bX(t,\by)),
\qquad\forall (t,\by)\in\Omega_{F}.
\end{equation*}
Now, by substituting the transformed solution
\begin{equation*}
\left\{
\begin{array}{l}
\bU(t,\by)=\nabla{\bY}(t,{\bX}(t,\by))\bu(t,{\bX}(t,\by)),
\\
P(t,\by)=p(t,{\bX}(t,\by)),
\\
\bA(t)=\Q^{T}(t)\ba(t),
\\
\bOmega(t)=\Q^{T}(t)\bomega(t),
\\
\mathcal{T}({\bU}(t,\by),P(t,\by))=\Q^{T}(t)
\T(\Q(t)\bU(t,\by),P(t,\by))\Q(t)
\end{array}
\right.
\end{equation*}
in the system of equations \eqref{FSINoslip}, we get (see \cite{GGH13} or \cite{chemetov2017weak})

\begin{equation}	\label{NSTransformed}
\left.
\begin{array}{l}
\partial _{t}\bU+(\bU\cdot \nabla )\bU-\triangle \bU+\nabla P = F, \\
\divg\bU = 0
\end{array}
\right\} \;\mathrm{in}\;(0,T)\times\Omega_{F},
\end{equation}
\begin{align}
\bA^{\prime }
&=-\bOmega\times \bA-\int_{\partial S_{0}}\mathcal{T}(\bU,P)\bN\,\d\gamma(\by)\qquad \mathrm{in}\;(0,T),
\label{RigidBodyTransformed1}
\\
(I\bOmega)^{\prime }
&=\bOmega\times (I\bOmega)
-\int_{\partial S_{0}} (
\by-\bq(t))\times {\mathcal{T}}(\bU,P)\bN \,\d\gamma(\by)
\quad \mathrm{in}
\;(0,T),
\label{RigidBodyTransformed2}
\end{align}
\begin{equation}	\label{couplingTransformed}
\bU = \bU_{s}\qquad \mathrm{on}\;(0,T)\times\partial S_0,
\qquad\quad \bU = 0\qquad \mathrm{on}\;(0,T)\times\partial\Omega,
\end{equation}
where
$
\bU_s = \bOmega\times\by + \bA
$
is the transformed rigid velocity $\bu_s$, $\bN=\bN(\by)$ is the unit normal at $\by\in S_0$, directed inside of $S_0$, $I={\Q}^{T}\J{\Q}$ is the transformed inertia tensor which no longer depends on time, and
\begin{equation*}
F = (\mathcal{L}-\triangle )\bU-\mathcal{M}\bU
-\tilde{\mathcal{N}}\mathbf{U}-(\mathcal{G}-\nabla )P.
\end{equation*}
The operator $\mathcal{L}$ is the transformed Laplace operator and it is given by
\begin{align}  \label{OpL}
(\mathcal{L}\mathbf{u})_{i}& =\sum_{j,k=1}^{n}\partial
_{j}(g^{jk}\partial_k \mathbf{u}_{i})+2\sum_{j,k,l=1}^{n}g^{kl}\Gamma
_{jk}^{i}\partial _{l}\mathbf{u}_{j}  \notag \\
& \quad
+\sum_{j,k,l=1}^{n}\big(\partial _{k}(g^{kl}\Gamma
_{jl}^{i})+\sum_{m=1}^{n}g^{kl}\Gamma _{jl}^{m}\Gamma _{km}^{i}\Big)\mathbf{u
}_{j},
\end{align}
the convection term is transformed into
\begin{equation}  \label{OpC}
(\mathcal{N}\mathbf{u})_i = \sum _{j=1}^n \mathbf{u}_j \partial _j \mathbf{u}
_i + \sum_{j,k=1}^n \Gamma ^i_{jk} \mathbf{u}_j\mathbf{u}_k
= (\bu\cdot\nabla\bu)_i + (\tilde{\mathcal{N}}\bu)_i,
\end{equation}
the transformation of time derivative and gradient is given by
\begin{equation}	\label{OpN}
(\mathcal{M} \mathbf{u})_{i} = \sum _{j=1}^n \dot{\mathbf{Y}}_j \partial _j
\mathbf{u}_i + \sum _{j,k=1}^n \Big(\Gamma _{jk}^i \dot{\mathbf{Y}}_k +
(\partial _k \mathbf{Y}_i)(\partial _j \dot{\mathbf{X}}_k)\Big)\mathbf{u}_j,
\end{equation}
and the gradient of pressure is transformed as follows:
\begin{equation}  \label{OpP}
(\mathcal{G}p)_{i}=\sum_{j=1}^{n}g^{ij}\partial _{j}p.
\end{equation}
Here we have denoted the metric covariant tensor
{
	\begin{equation}	\label{covariant}
	g_{ij}=X_{k,i}X_{k,j},\qquad
	X_{k,i}=\frac{\partial \bX_{k}}{\partial \by_{i}},
	\end{equation}
	the metric covariant tensor
	\begin{equation}
	g^{ij}=Y_{i,k}Y_{j,k}\qquad Y_{i,k}=\frac{\partial \bY_{i}}{\partial \bx_{k}},
	\end{equation}
	and the Christoffel symbol (of the second kind)
	\begin{equation}	\label{Christoffel}
	\Gamma _{ij}^{k}=\frac{1}{2}g^{kl}(g_{il,j}+g_{jl,i}-g_{ij,l}),\qquad
	g_{il,j}=\frac{\partial {g_{il}}}{\partial \by_{j}}.
	\end{equation}	
}
It is easy to observe that, in particular, the following holds:
\begin{equation*}
\Gamma _{ij}^{k}=Y_{k,l}X_{l,ij}.\qquad X_{l,ij}=\frac{\partial \bX_{l}}{
	\partial \by_{i}\partial \by_{j}}.
\end{equation*}
With little abuse of notation, we identify the operators $\mathcal{L},\mathcal{M}, \widetilde{\mathcal{N}}$ with
\begin{multline}	\label{transformed_laplace}
	\left\langle \mathcal{L}\bU, \bpsi\right\rangle
	=\int_{\Omega_F(\tau)} \big(\sum_{ijk}
	(g^{jk}\partial_j\bU_{i}\partial_k\bpsi_i
	+ g^{jk}\partial_k\bU_i\partial_i\bpsi_j)
	- \sum_{ijkl} (g^{kl}+\partial_l\bY_k)\Gamma_{li}^j\partial_k\bU_{i}\bpsi_j\\
	+ \sum_{ijkl} (g^{kl}\Gamma_{li}^j\bU_{i}\partial_k\bpsi_j
	+ g^{jl}\Gamma_{li}^k\bU_i\partial_k\bpsi_j)
	- \sum_{ijklm} (g^{kl}+\partial_k\bY_l)\Gamma_{li}^m\Gamma_{km}^j\bU_{i}\bpsi_j
	\big),
\end{multline}
\begin{equation}	\label{transformed_time}
	\left\langle \mathcal{M}\bU, \bpsi\right\rangle
	= \int_{\Omega_F(\tau)}\sum _{i=1}^n\left( 
	\sum _{j=1}^n \dot{\mathbf{Y}}_j \partial _j
	\mathbf{u}_i + \sum _{j,k=1}^n \Big(\Gamma _{jk}^i \dot{\mathbf{Y}}_k +
	(\partial _k \mathbf{Y}_i)(\partial _j \dot{\mathbf{X}}_k)\Big)\mathbf{u}_j\right)
	\bpsi_i,
\end{equation}
\begin{equation}	\label{transformed_convective}
	\left\langle \tilde{\mathcal{N}}\bU, \bpsi\right\rangle
	= \int_{\Omega_F(\tau)}
	\sum_{i,j,k=1}^n \Gamma ^i_{jk} \mathbf{u}_j\mathbf{u}_k
	\bpsi_i,
\end{equation}
for all $\bpsi\in H^{1}(\Omega_{F}(t))$.

\subsection{Reynolds transport theorem - generalization}
\label{Sec:Reynolds_gen}

To prove the weak-strong uniqueness result and the energy equality, we want to cancel the derivation terms from the weak formulations. In the case of smooth functions $\bu$ and $\bv$, by the Reynolds Transport Theorem we have
\begin{multline*}
\int_{0}^{t}\int_{\Omega(\tau)}\big (\bu\cdot\partial_t\bv + \bv\cdot\partial_t\bu\big)
= -\int_{0}^{t}\int_{\Omega(\tau)} \nabla(\bv\cdot\bu)\cdot \partial_t\bX\,
+ \int_{\Omega(t)}\bv(t)\cdot \bu(t)\,
- \int_{\Omega(0)}\bv(0)\cdot \bu(0).
\end{multline*}
However, $\bu$ and $\bv$ are not regular enough, and the expression on the left is not well defined.
But we can use Lemma \ref{Reynolds_gen}, which states that for $\bu, \bv\in {L}^{2}(0,T;{H}^1(\Omega(t)))$
and a coordinate transformation $\bX:\Omega\to\Omega(t)$, we have
\begin{multline}	\label{Reynolds_gen_eq_appendix}
\int_{0}^{t}\int_{\Omega(\tau)}\big (\bu\cdot\partial_t\bv^h + \bv\cdot\partial_t\bu^h\big)\,\d\bx\,\d\tau
\\
\to -\int_{0}^{t}\int_{\Omega(\tau)} \nabla(\bv\cdot\bu)\cdot \partial_t\bX\,\d\bx\,\d\tau
+ \int_{\Omega(t)}\bv(t)\cdot \bu(t)\,\d\bx\,
- \int_{\Omega(0)}\bv(0)\cdot \bu(0)\,\d\bx,
\end{multline}
when $h\to 0$, for almost every $t\in[0,T]$. Here $\bu^h$ denotes the regularization of $\bu$ described by \eqref{RegDef1}-\eqref{RegDef3}.

\begin{remark}
	If the domain is fixed, the coordinate transformation $\bX$ is not necessary ($\bX=id$) and the regularization is standard (convolution in time). Then, by Fubini's theorem and the properties of the mollifier, we get
	\begin{align}
	\int_{0}^{t}\int_{\Omega}&\bu(\tau)\cdot \partial_t\bv^h(\tau)\,\d\bx \,\d\tau
	= \int_{0}^{t}\int_{\Omega}\int_{-\infty}^{+\infty}
	\bu(\tau)\cdot \frac{d}{d\tau} j_h(\tau-s)\bv(s)\,\, \d s \, \d\bx \, \d\tau
	\notag\\
	=& \int_{-\infty}^{+\infty}\int_{\Omega}\int_{0}^{t}
	\frac{d}{d\tau}j_h(\tau-s)\bu(\tau)\cdot \bv(s)\,\, \d \tau \, \d\bx \, \d s
	\notag\\
	=& -\int_{0}^{t}\int_{\Omega}\int_{-\infty}^{+\infty}
	\frac{d}{ds} j_h(s-\tau)\bu(\tau)\cdot \bv(s)\,\, \d\tau \, \d\bx \, \d s
	\label{pi1}
	\\
	&- \int_{-\infty}^{0}\int_{0}^{t}\int_{\Omega}
	\bu(\tau)\cdot \frac{d}{d\tau}j_h(\tau-s)\bv(s)\,\, \d \tau \, \d\bx \, \d s
	\label{pi2}\\
	&- \int_{t}^{+\infty}\int_{0}^{t}\int_{\Omega}
	\bu(\tau)\cdot \frac{d}{d\tau}j_h(\tau-s)\bv(s)\,\, \d \tau \, \d\bx \, \d s
	\label{pi3}\\
	&+ \int_{0}^{t}\int_{-\infty}^{0}\int_{\Omega}
	\bu(\tau)\cdot \frac{d}{d\tau}j_h(\tau-s)\bv(s)\,\, \d \tau \, \d\bx \, \d s
	\label{pi4}\\
	&+ \int_{0}^{t}\int_{t}^{+\infty}\int_{\Omega}
	\bu(\tau)\cdot \frac{d}{d\tau}j_h(\tau-s)\bv(s)\,\, \d \tau \, \d\bx \, \d s
	\label{pi5}
	\end{align}
	We see that
	$$
	\eqref{pi1} = -\int_{0}^{t}d\tau\int_{\Omega}\partial_t\bu^h(\tau)\cdot \bv(\tau)\,\,\d\bx
	$$
	and it is easy to prove, by using the Lebesgue differentiation theorem, that
	$$
	\eqref{pi2}
	+\eqref{pi3}
	+\eqref{pi4}
	+\eqref{pi5}
	\to
	\int_{\Omega}\bu(t)\cdot \bv(t)\,\,\d\bx
	- \int_{\Omega}\bu(0)\cdot \bv(0)\,\,\d\bx,
	\qquad h\to 0,
	$$
	which ends the proof.
	
	In the case of a moving domain the idea of the proof is the same, but the calculation is more complicated because of the changes of variables in the definition of the regularization and before applying Fubini's theorem.
\end{remark}

First we introduce one auxiliary result.
\begin{lemma}	\label{weak_convergence}
	Let $\bu,\bv$ and $\bX$ be as in Lemma \ref{Reynolds_gen}, and let $\bY$ be the inverse transformation $\bY(t,\cdot)=\bX(t,\cdot)^{-1}$.
	Then,
	\begin{equation*}
	\int_{0}^{t}\int_{\Omega(\tau)} \bv\cdot\partial_t \bu^h \,\d\bx\,\d\tau
	-\int_{0}^{t}\int_{\Omega(\tau)} \bv\cdot\partial_t\bU_h \,\d\bx\,\d\tau
	\,\to\, 0,
	\qquad h\to 0,
	\end{equation*}
	where
	\begin{equation*}
	\bU_h(t,\bx) =
	\nabla \bY(t,\bx)^T
	\int_{-\infty}^{+\infty}
	j_h(t-\tau)\nabla \bX(\tau,\bY(t,\bx))^T \nabla \bX(\tau,\bY(t,\bx))\bar{\bu}(\tau,\bY(t,\bx))
	\,\,\d\tau.
	\end{equation*}
\end{lemma}

\begin{proof}
	Since
	\begin{align}
	\int_{0}^{t}&\int_{\Omega(\tau)} \bv\cdot\partial_t\bu^h
	\notag\\
	&= \int_{0}^{t}\int_{\Omega(\tau)} \bv(\tau,\bx)\cdot\frac{d}{d\tau}\Big (
	\int_{-\infty}^{+\infty}
	j_h(\tau-s)
	\nabla\bX(\tau,\bY(\tau,\bx))\,
	\bar{\bu}(s,\bY(\tau,\bx))
	\,\,\d s
	\Big ) \,\d\bx\,\d\tau
	\notag\\
	&= \int_{0}^{t}\int_{\Omega(\tau)} \bv(\tau,\bx)\cdot\frac{d}{d\tau}\Big ( \nabla\bY(\tau,\bx)^T
	\notag\\
	&\qquad\qquad
	\int_{-\infty}^{+\infty}
	j_h(\tau-s)\nabla\bX(\tau,\bY(\tau,\bx))^T\nabla\bX(\tau,\bY(\tau,\bx))\bar{\bu}(s,\bY(\tau,\bx))
	\,\,\d s
	\Big ) \,\d\bx\,\d\tau,
	\notag
	\end{align}
	it follows
	\begin{equation*}
	\int_{0}^{t}\int_{\Omega(\tau)} \bv\cdot\partial_t \bu^h \,\d\bx\,\d\tau
	-\int_{0}^{t}\int_{\Omega(\tau)} \bv\cdot\partial_t\bU_h \,\d\bx\,\d\tau
	= \int_{0}^{t}\int_{\Omega(\tau)} \bv(\tau,\bx)\cdot\frac{d}{d\tau} f_h(\tau,\bx) \,\d\bx\,\d\tau,
	\end{equation*}
	where
	\begin{multline*}
	f_h(\tau,\bx) = \nabla\bY(\tau,\bx)^T
	\int_{-\infty}^{+\infty}
	j_h(\tau-s)
	\\
	\big(\nabla\bX(\tau,\bY(\tau,\bx))^T\nabla\bX(\tau,\bY(\tau,\bx))
	- \nabla\bX(s,\bY(\tau,\bx))^T\nabla\bX(s,\bY(\tau,\bx))\big)
	\\
	\bar{\bu}(s,\bY(\tau,\bx))
	\,\d s.
	\end{multline*}
	Since $f_h\to 0$ strongly in $\mathbf{L}^2 \mathbf{L}^2$ and the derivatives $\frac{d}{d\tau}f_h$ are bounded in $\mathbf{L}^2 \mathbf{L}^2$, it follows that $\frac{d}{d\tau}f_h\to 0$ weakly in $\mathbf{L}^2 \mathbf{L}^2$, so the above expression tends to 0 when $h\to 0$.	
\end{proof}

Now we are able to prove Lemma \ref{Reynolds_gen}.

\begin{proof}[Proof of Lemma \ref{Reynolds_gen}]
	As in the fixed domain case, we start with the first term on the left-hand side of \eqref{Reynolds_gen_eq}:
	\begin{align}
	\int_{0}^{t}\int_{\Omega(\tau)}\bu\cdot \partial_t\bv^h&\,\,\d\bx\,\d\tau
	=
	\int_{0}^{t}\int_{\Omega(\tau)}\bu(\tau,\bx)\cdot \frac{d}{d\tau} \Big(\nabla \bX(\tau,\bY(\tau,\bx))\,\bar{\bv}^h(\tau,\bY(\tau,\bx))\Big)\,\d\bx\,\d\tau
	\notag \\
	& =
	\int_{0}^{t}\int_{\Omega(\tau)}
	\bu(\tau,\bx)\cdot \frac{d}{d\tau} \nabla \bX(\tau,\bY(\tau,\bx))\,\bar{\bv}^h(\tau,\bY(\tau,\bx)
	\,\,\d\bx\,\d\tau
	\label{E1}
	\\
	& +
	\int_{0}^{t}\int_{\Omega(\tau)}
	\bu(\tau,\bx)\cdot \nabla \bX(\tau,\bY(\tau,\bx))\,\partial_t\bar{\bv}^h(\tau,\bY(\tau,\bx))
	\,\,\d\bx\,\d\tau
	\label{E2}
	\\
	& +
	\int_{0}^{t}\int_{\Omega(\tau)}
	\bu(\tau,\bx)\cdot \nabla \bX(\tau,\bY(\tau,\bx))\nabla\bar{\bv}^h(\tau,\bY(\tau,\bx))\,\partial_t \bY(\tau,\bx)
	\,\,\d\bx\,\d\tau
	\label{E3}
	\end{align}
	The integral \eqref{E2} contains the time derivative of the function $\bar{\bv}^h$, so we need to combine it with the second term on the left-hand side of \eqref{Reynolds_gen_eq} before passing to the limit.
	First we change the coordinates. Then we can apply Fubini's theorem, as follows.
	\begin{align}
	\int_{0}^{t}\int_{\Omega(\tau)}&
	\bu(\tau,\bx)\cdot \nabla \bX(\tau,\bY(\tau,\bx))\,\partial_t\bar{\bv}^h(\tau,\bY(\tau,\bx))
	\,\d\bx\,\d\tau
	\notag\\
	&=
	\int_{0}^{t}\int_{\Omega}
	\bu(\tau,\bX(\tau,\by))\cdot \nabla \bX(\tau,\by)\,\partial_t\bar{\bv}^h(\tau,\by)
	\,\d\by\,\d\tau
	\notag\\
	&=
	\int_{0}^{t}\int_{\Omega}\int_{-\infty}^{+\infty}
	\frac{d}{d\tau}j_h(\tau-s)
	\nabla \bX(\tau,\by)\bar{\bu}(\tau,\by)
	\cdot \nabla \bX(\tau,\by)\bar{\bv}(s,\by)
	\,\d s\,\d\by\,\d\tau
	\notag\\
	&=
	-\int_{-\infty}^{+\infty}\int_{\Omega}
	\bar{\bv}(s,\by)\cdot
	\int_{0}^{t}
	\frac{d}{ds}j_h(s-t)
	\nabla \bX(t,\by)^T \nabla \bX(t,\by)\bar{\bu}(t,\by)
	\,\d\tau\,\d\by \, \d s
	\notag\\
	&=
	-\int_{0}^{t}\int_{\Omega}
	\bar{\bv}(s,\by)\cdot
	\int_{-\infty}^{+\infty}
	\frac{d}{ds}j_h(s-t)
	\nabla \bX(t,\by)^T \nabla \bX(t,\by)\bar{\bu}(t,\by)
	\,\d\tau\,\d\by \, \d s
	\label{E4}\\
	&\quad + \int_{\Omega(t)}\bv(t,\bx)\cdot \bu(t,\bx)\,\d\bx
	- \int_{\Omega}\bv(0,\bx)\cdot \bu(0,\bx)\,\d\bx
	+ o(h).
	\notag
	\end{align}
	The last two equalities are simple consequences of the properties of the mollifier and Lebesgue differentiation theorem (as in the fixed domain case). Next, we calculate \eqref{E4}:
	\begin{align}
	\eqref{E4}
	&=
	-\int_{0}^{t}\int_{\Omega(s)}
	\bar{\bv}(s,\bY(s,\bx))\cdot
	\int_{-\infty}^{+\infty}
	\frac{d}{ds}j_h(s-\tau)
	\notag\\
	&\qquad\qquad\qquad\qquad
	\nabla \bX(\tau,\bY(s,\bx))^T \nabla \bX(\tau,\bY(s,\bx))
	\bar{\bu}(\tau,\bY(s,\bx))
	\,\d\tau\,\d\bx \, \d s
	\notag\\
	&=
	-\int_{0}^{t}\int_{\Omega(s)}
	\bar{\bv}(s,\bY(s,\bx))\cdot
	\frac{d}{ds}\Big (\int_{-\infty}^{+\infty}
	j_h(s-\tau)
	\label{E5}\\
	&\qquad\qquad\qquad\qquad
	\nabla\bX(\tau,Y(s,\bx))^T \nabla \bX(\tau,\bY(s,\bx))\bar{\bu}(\tau,\bY(s,\bx))
	\,\,\d\tau\Big )\,\d\bx \, \d s
	\notag\\
	& +\int_{0}^{t}\int_{\Omega(s)}
	\bar{\bv}(s,\bY(s,\bx))\cdot
	\int_{-\infty}^{+\infty}
	j_h(s-\tau)\frac{d}{ds}\big (\nabla \bX(\tau,\bY(s,\bx))^T\big )
	\label{E7}\\
	&\qquad\qquad\qquad\qquad 
	\nabla \bX(\tau,\bY(s,\bx))\bar{\bu}(\tau,\bY(s,\bx))
	\,\,\d\tau\,\d\bx \, \d s
	\notag
	\\
	& +\int_{0}^{t}\int_{\Omega(s)}
	\bar{\bv}(s,\bY(s,\bx))\cdot
	\int_{-\infty}^{+\infty}
	j_h(s-\tau)
	\nabla \bX(\tau,\bY(s,\bx))^T
	\label{E9}\\
	&\qquad\qquad\qquad\qquad
	\frac{d}{ds}\big (\nabla \bX(\tau,\bY(s,\bx)) \bar{\bu}(\tau,\bY(s,\bx))\big )
	\,\,\d\tau\,\d\bx \, \d s,
	\notag
	\end{align}
	and for \eqref{E5} we get
	\begin{align}
	&\eqref{E5}
	\notag\\
	&=
	-\int_{0}^{t}\int_{\Omega(s)}
	\bv(s,\bx)\cdot \nabla \bY(s,\bx)^T
	\notag\\
	&\qquad\qquad
	\frac{d}{ds}\Big (\int_{-\infty}^{+\infty}
	j_h(s-\tau)\nabla \bX(\tau,\bY(s,\bx))^T \nabla \bX(\tau,\bY(s,\bx))\bar{\bu}(\tau,\bY(s,\bx))
	\,\,\d\tau\Big )\,\d\bx \, \d s
	\notag\\
	&=
	-\int_{0}^{t}\int_{\Omega(s)}
	\bv(s,\bx)\cdot
	\frac{d}{ds}
	\Big (\nabla \bY(s,\bx)^T
	\notag\\
	&\qquad\qquad
	\int_{-\infty}^{+\infty}
	j_h(s-\tau)\nabla \bX(\tau,\bY(s,\bx))^T \nabla \bX(\tau,\bY(s,\bx))\bar{\bu}(\tau,\bY(s,\bx))
	\,\,\d\tau\Big )\,\d\bx \, \d s
	\label{E8}\\
	&+ \int_{0}^{t}\int_{\Omega(s)}
	\bv(s,\bx)\cdot
	\partial_s\nabla \bY(s,\bx)^T
	\notag\\
	&\qquad\qquad
	\int_{-\infty}^{+\infty}
	j_h(s-\tau)\nabla \bX(\tau,\bY(s,\bx))^T \nabla \bX(\tau,\bY(s,\bx))\bar{\bu}(\tau,\bY(s,\bx))
	\,\,\d\tau\,\d\bx \, \d s
	\label{E6}
	\end{align}
	Now we can let $h\to 0$. By Lemma \ref{weak_convergence}, we have
	\begin{equation*}
	\eqref{E8} + \int\int \bv\cdot\partial_t\bu^h \to 0.
	\end{equation*}
	The remaining terms do not contain the time derivative of $\bar{\bu}^h$ or $\bar{\bv}^h$, so we can directly pass to the limits.
	Using the property of the transformation $\bX$
	\begin{align*}
	0&=\frac{d}{dt}\big ( \nabla \bX(t,\bY(t,\bx))\nabla \bY(t,\bx) \big )\\
	&= \frac{d}{dt}\big (\nabla \bX(t,\bY(t,\bx)) \big )\nabla \bY(t,\bx) + \nabla \bX(t,\bY(t,\bx))\partial_t\nabla \bY(t,\bx),
	\end{align*}
	we get
	\begin{align}
	\eqref{E1}
	&= \int_{0}^{t} \int_{\Omega(\tau)}
	\bu(\tau,\bx)\cdot \frac{d}{d\tau} \nabla \bX(\tau,\bY(\tau,\bx))\,\bar{\bv}^h(\tau,\bY(\tau,\bx))
	\,\d\bx\,\d\tau
	\notag\\
	&\to \int_{0}^{t}\int_{\Omega(\tau)}
	\bu(\tau,\bx)\cdot \frac{d}{d\tau} \nabla \bX(\tau,\bY(\tau,\bx))\,\bar{\bv}(\tau,\bY(\tau,\bx))
	\,\d\bx\,\d\tau
	\notag\\
	&=-\int_{0}^{t}\int_{\Omega(\tau)}
	\bu(\tau,\bx)\cdot \nabla \bX(\tau,\bY(\tau,\bx))\,\partial_t\nabla \bY(\tau,\bx)\,\bv(\tau,\bY(\tau,\bx))
	\,\d\bx\,\d\tau,
	\notag
	\end{align}
	\begin{align}
	\eqref{E6}
	&= \int_{0}^{t}\int_{\Omega(s)}
	\bv(s,\bx)\cdot
	\partial_s\nabla \bY(s,\bx)^T
	\notag\\
	&\qquad\qquad
	\int_{-\infty}^{+\infty}
	j_h(s-\tau)\nabla \bX(\tau,\bY(s,\bx))^T \nabla \bX(\tau,\bY(s,\bx))\,\bar{\bu}(\tau,\bY(s,\bx))
	\,\d\tau\,\d\bx \, \d s
	\notag\\
	&\to \int_{0}^{t}\int_{\Omega(s)}
	\bv(s,\bx)\cdot
	\partial_s\nabla \bY(s,\bx)^T
	\notag\\
	&\qquad\qquad
	\nabla \bX(s,\bY(s,\bx))^T \nabla \bX(s,\bY(s,\bx))\,
	\bar{\bu}(s,\bY(s,\bx))
	\,\d\bx \, \d s
	\notag\\
	&= \int_{0}^{t}\int_{\Omega(s)}
	\nabla \bX(s,\bY(s,\bx)) \partial_s\nabla \bY(s,\bx)\,\bv(s,\bx)\cdot \bu(s,\bY(s,\bx))
	\,\d\bx \, \d s.
	\notag
	\end{align}
	It follows that
	\begin{displaymath}
	\eqref{E1}+\eqref{E6}\to 0, \qquad h\to 0.
	\end{displaymath}
	Again, using the properties of the transformation of coordinates we get
	\begin{equation}\label{L1}
	\eqref{E3}+\eqref{E7}
	\to -\int_{0}^{t}\int_{\Omega(\tau)} \nabla\bv^T\bu\cdot \partial_t\bX,
	\end{equation}
	\begin{equation}\label{L2}
	\eqref{E9}
	\to -\int_{0}^{t}\int_{\Omega(\tau)} \nabla\bu^T\bv\cdot \partial_t\bX.
	\end{equation}
	Hence,
	\begin{equation}
	\begin{split}
	&\eqref{E1} + \eqref{E3} + \eqref{E7} + \eqref{E9} + \eqref{E6}
	\\
	&\qquad
	\to -\int_{0}^{t}\int_{\Omega_F(\tau)} (\nabla\bv^T\bu+\nabla\bu^T\bv)\cdot \partial_t\bX
	= -\int_{0}^{t}\int_{\Omega_F(\tau)} \nabla(\bu\cdot\bv)\cdot \partial_t\bX.
	\end{split}
	\end{equation}
	Finally, we conclude
	\begin{align*}
	&\int_{0}^{t}\int_{\Omega(\tau)}\bu\cdot \partial_t\bv^h\,\,\d\bx\,\d\tau
	+ \int_{0}^{t}\int_{\Omega(\tau)}\bv\cdot \partial_t\bu^h
	\,\d\bx\,\d\tau
	\\
	&\qquad\qquad
	\to
	-\int_{0}^{t}\int_{\Omega(\tau)} \nabla(\bu\cdot\bv)\cdot \partial_t\bX
	\,\d\bx\,\d\tau
	+ \int_{\Omega(t)}\bv(t)\cdot \bu(t)
	\,\d\bx
	- \int_{\Omega}\bv(0)\cdot \bu(0)\,\d\bx.
	\notag
	\end{align*}
	Let us show \eqref{L1} and \eqref{L2}:
	
	\noindent
	Since
	$$
	\nabla\bar{\bv}^h \to \nabla\bar{\bv}
	\quad\text{when }h\to 0 \quad\text{u } L^2L^2,
	$$
	we have
	\begin{align*}
	\eqref{E3}
	&= \int_0^t\int_{\Omega_F(\tau)} \bu(\tau,\bx)\cdot \nabla\bX(\tau,\bY(\tau,\bx))\nabla\bar{\bv}^h(\tau,\bY(\tau,\bx))\partial_t\bY(\tau,\bx)\,\d\bx\,\d\tau
	\notag\\
	&\to \int_0^t\int_{\Omega_F(\tau)} \bu(\tau,\bx)\cdot \nabla\bX(\tau,\bY(\tau,\bx))\nabla\bar{\bv}(\tau,\bY(\tau,\bx))\partial_t\bY(\tau,\bx)\,\d\bx\,\d\tau
	\notag\\
	&= \int_0^t\int_{\Omega_F(\tau)}
	\sum_{ijk}\bu_i\partial_j\bX_i\partial_k\bar{\bv}\partial_t\bY_k
	= \int_0^t\int_{\Omega_F(\tau)}
	\sum_{ijk}\bu_i\partial_j\bX_i\frac{d}{d\by_k}(\nabla\bY\bv)_j\partial_t\bY_k
	\notag\\
	&= \int_0^t\int_{\Omega_F(\tau)}
	\sum_{ijkl}\bu_i\partial_j\bX_i\frac{d}{d\by_k}(\partial_l\bY_j\bv_l)\partial_t\bY_k
	\notag\\
	&= \int_0^t\int_{\Omega_F(\tau)}
	\sum_{ijklm}\bu_i\partial_j\bX_i(\partial_m\partial_l\bY_j\bv_l + \partial_l\bY_j\partial_m\bv_l)\underbrace{\partial_k\bX_m\partial_t\bY_k}_{(\nabla\bX\partial_t\bY)_{m}=-\partial_t\bX_m}
	\notag\\
	&= \int_0^t\int_{\Omega_F(\tau)}(
	\underbrace{\sum_{ijklm}\bu_i\partial_j\bX_i\partial_m\partial_l\bY_j\partial_k\bX_m\partial_t\bY_k \bv_l}_{I}
	- \underbrace{\sum_{ijlm}\bu_i\partial_j\bX_i\partial_l\bY_j\partial_t\bX_m \partial_m\bv_l}_{II}),
	\end{align*}
	\begin{align*}
	II &= \sum_{ijlm}\bu_i
	\underbrace{\partial_j\bX_i\partial_l\bY_j}_{\delta_{il}}
	\partial_t\bX_m \partial_m\bv_l
	= \sum_{im}\bu_i\partial_t\bX_m \partial_m\bv_i
	= \nabla\bv^T\bu\cdot\partial_t\bX,
	\end{align*}
	\begin{align*}
	I &= \sum_{ijklm}\bu_i\partial_j\bX_i
	\underbrace{\partial_m\partial_l\bY_j\partial_k\bX_m}_{\frac{d}{d\by_k}(\partial_l\bY_j)}
	\partial_t\bY_k \bv_l
	= \sum_{ijkl}\bu_i
	\underbrace{\partial_j\bX_i\frac{d}{d\by_k}(\partial_l\bY_j)}_{= \frac{d}{d\by_k}(\partial_j\bX_i\partial_l\bY_j)-\partial_k\partial_j\bX_i\partial_l\bY_j}
	\partial_t\bY_k \bv_l
	\notag\\
	&= -\sum_{ijkl}\bu_i
	\underline{\partial_k\partial_j\bX_i}\partial_l\bY_j
	\underline{\partial_t\bY_k} \bv_l
	=(*)
	= -\sum_{ijl}\bu_i
	(\frac{d}{dt}(\nabla\bX)_{ij}-\partial_t(\nabla\bX)_{ij})
	\partial_l\bY_j
	\bv_l
	\notag\\
	&= -\frac{d}{dt}\nabla\bX^T\bu\cdot\nabla\bY\bv
	+ \partial_t\nabla\bX^T\bu\cdot\nabla\bY\bv,
	\end{align*}
	\begin{equation}
	\eqref{E3}\to \int_{0}^{t}\int_{\Omega_F(\tau)}(-\frac{d}{dt}\nabla\bX^T\bu\cdot\nabla\bY\bv
	+ \partial_t\nabla\bX^T\bu\cdot\nabla\bY\bv
	- \nabla\bv^T\bu\cdot\partial_t\bX),
	\end{equation}	
	\begin{align*}
	\eqref{E7}
	&= \int_{0}^{t}\int_{\Omega_F(\tau)}
	\underbrace{\bar{\bv}(\tau,\bY(\tau,\bx))}_{=\nabla\bY(\tau,\bx)\bv(\tau,\bx)}
	\cdot
	\int_{-\infty}^{+\infty}
	j_h(\tau-s)\frac{d}{d\tau}\big (\nabla \bX(s,\bY(\tau,\bx))^T\big ) 
	\notag\\
	&\qquad\qquad\qquad\qquad\qquad\qquad\qquad
	\nabla\bX(s,\bY(\tau,\bx))\bar{\bu}(s,\bY(\tau,\bx))
	\,\,\d s\,\d\bx \, \d\tau
	\notag\\
	&= \int_{0}^{t}\int_{\Omega_F(\tau)}\int_{-\infty}^{+\infty}
	\sum_{ijkl}
	\bv_i\partial_i\bY_j j_h(\tau-s)\frac{d}{d\tau}\partial_j\bX_k(s,\bY(\tau,\bx))
	\notag\\
	&\qquad\qquad\qquad\qquad\qquad\qquad\qquad
	\partial_l\bX_k(s,\bY(\tau,\bx))\bar{\bu}_l(s,\bY(\tau,\bx))
	\,\,\d s\,\d\bx \, \d\tau
	\notag\\
	&= \int_{0}^{t}\int_{\Omega_F(\tau)}\int_{-\infty}^{+\infty}
	\sum_{ijklm}
	\bv_i\partial_i\bY_j j_h(\tau-s)\partial_m\partial_j\bX_k(s,\bY(\tau,\bx))\partial_t\bY_m
	\notag\\
	&\qquad\qquad\qquad\qquad\qquad\qquad\qquad \partial_l\bX_k(s,\bY(\tau,\bx))\bar{\bu}_l(s,\bY(\tau,\bx))
	\,\,\d s\,\d\bx \, \d\tau
	\notag\\
	&\to \int_{0}^{t}\int_{\Omega_F(\tau)}
	\sum_{ijklm}
	\bv_i\partial_i\bY_j
	\underbrace{\partial_m\partial_j\bX_k\partial_t\bY_m }_{\frac{d}{dt}(\nabla\bX)_{kj}-\partial_t(\nabla\bX)_{kj}}
	\underbrace{\partial_l\bX_k\bar{\bu}_l}_{\bu_k}
	\,\,\d s\,\d\bx \, \d\tau
	\notag\\
	&=-(*)=-\int\int I
	= \int_{0}^{t}\int_{\Omega_F(\tau)}
	(\nabla\bY\bv\cdot\frac{d}{dt}\nabla\bX^T\bu
	-\nabla\bY\bv\cdot\partial_t\nabla\bX^T\bu)
	\,\d\bx \, \d\tau.
	\end{align*}
	It follows
	\begin{displaymath}
	\eqref{E3}+\eqref{E7}
	\to -\int_{0}^{t}\int_{\Omega_F(\tau)} \nabla\bv^T\bu\cdot \partial_t\bX.
	\end{displaymath}
	It remains to prove \eqref{E9}:
	\begin{align*}
	\eqref{E9}
	&= \int_{0}^{t}\int_{\Omega_F(\tau)}
	\bar{\bv}(\tau,\bY(\tau,\bx))\cdot
	\int_{-\infty}^{+\infty}
	j_h(\tau-s)\nabla \bX(s,\bY(\tau,\bx))^T
	\notag\\
	&\qquad\qquad\qquad\qquad\qquad\qquad
	\frac{d}{d\tau}\big (\nabla \bX(s,\bY(\tau,\bx)) \bar{\bu}(s,\bY(\tau,\bx))\big )
	\,\,\d s\,\d\bx \, \d\tau
	\notag\\
	&= \int_{0}^{t}\int_{\Omega_F(\tau)}
	\sum_{ijk} \bar{\bv}_i \int_{-\infty}^{+\infty}
	j_h(\tau-s)\partial_i\bX_j\frac{d}{d\tau}(\partial_k\bX_j\bar{\bu}_k)
	\notag\\
	&= \int_{0}^{t}\int_{\Omega_F(\tau)}
	\sum_{ijkl} \bar{\bv}_i \int_{-\infty}^{+\infty}
	j_h(\tau-s)\partial_i\bX_j
	(\partial_l\partial_k\bX_j\bar{\bu}_k + \partial_k\bX_j\partial_l\bar{\bu}_k)\partial_t\bY_l
	\notag\\
	&\to \int_{0}^{t}\int_{\Omega_F(\tau)}
	\sum_{ijkl} \bar{\bv}_i \partial_i\bX_j
	\underbrace{(\partial_l\partial_k\bX_j\bar{\bu}_k + \partial_k\bX_j\partial_l\bar{\bu}_k)}_{\frac{d}{d\by_l}(\nabla\bX\bar{\bu})_j=\frac{d}{d\by_l}\bu_j=(\nabla\bu\nabla\bX)_{jl}}
	\partial_t\bY_l
	\notag\\
	&= \int_{0}^{t}\int_{\Omega_F(\tau)}
	\bar{\bv}\cdot\nabla\bX^T
	\nabla\bu\nabla\bX
	\partial_t\bY
	= \int_{0}^{t}\int_{\Omega_F(\tau)}
	\underbrace{\nabla\bX\bar{\bv}}_{=\bv}\cdot
	\nabla\bu
	\underbrace{\nabla\bX\partial_t\bY}_{=-\partial_t\bX}
	\notag\\
	&= -\int_{0}^{t}\int_{\Omega_F(\tau)}
	\bv\cdot \nabla\bu \partial_t\bX
	= -\int_{0}^{t}\int_{\Omega_F(\tau)}
	\nabla\bu^T \bv \cdot \partial_t\bX.
	\end{align*}	
\end{proof}

\subsection{Weak formulation - details}\label{Sec:WFDEtails}
In this subsection we give the remaining technical details of the proof of Proposition \ref{WeakU2}.
For simplicity of notation we denote:
$$
\bX=\widetilde{\bX}_2, \,
\bx=\bx_2, \,
\bY=\widetilde{\bX}_1, \,
\by=\bx_1
$$
and
$$
(\bU,P,\bA,\bOmega) = (\bU_{2},P_{2},\bA_{2},\bOmega_{2}),\,
(\bu,p,\ba,\bomega) = (\bu_{2},p_{2},\ba_{2},\bomega_{2}).
$$

\noindent
\textbf{The fluid time-derivative term}.
\begin{equation*}
\bu\cdot\partial_t\bphi
= \bU\cdot\partial_t\bpsi
-\mathcal{M}\bU\cdot\bpsi
+ \nabla(\bU\cdot\bpsi)\cdot\partial_t\bY.
\end{equation*}
\begin{proof}
	We have
	\begin{align*}
	\bu\cdot\partial_t\bphi
	&= \nabla\bX\bU\cdot\frac{d}{dt}(\nabla\bY^T\bpsi)
	\notag\\
	&= \nabla\bX\bU\cdot(\partial_t\nabla\bY^T\bpsi + \nabla\bY^T\partial_t\bpsi + \nabla\bY^T\nabla\bpsi\partial_t\bY)
	\notag\\
	&= \partial_t\nabla\bY\nabla\bX\bU\cdot\bpsi
	+ \nabla\bY\nabla\bX\bU\cdot\partial_t\bpsi
	+ \nabla\bY\nabla\bX\bU\cdot\nabla\bpsi\partial_t\bY.
	\end{align*}
	The property of the transformation $\bX$
	\begin{equation}	\label{transformationIdentity}
	\nabla\bX(t,\bY(t,\bx)))\nabla\bY(t,\bx)=\mathbb{I}
	\end{equation}
	implies
	\begin{equation*}
	\bu\cdot\partial_t\bphi
	= \bU\cdot\partial_t\bpsi
	+ \underbrace{\partial_t\nabla\bY\nabla\bX\bU\cdot\bpsi }_{I}
	+ \nabla\bpsi^T\bU\cdot\partial_t\bY
	\end{equation*}
	and by definition of $\mathcal{M}$ we have
	\begin{equation*}
	\mathcal{M}\bU\cdot\bpsi
	= \underbrace{\nabla\bU\partial_t\bY\cdot\psi}_{=\nabla\bU^T\bpsi\cdot\partial_t\bY}
	+ \underbrace{\nabla\bY\partial_t\nabla\bX\bU\cdot\bpsi}_{II}
	+ \underbrace{\sum_{ijk}\Gamma_{jk}^i\partial_t\bY_k\bU_j\bpsi_i}_{III}.
	\end{equation*}
	The identity \eqref{transformationIdentity} gives
	\begin{equation*}
	0 = \frac{d}{dt}(\nabla\bY\nabla\bX)_{ij}
	= (\partial_t\nabla\bY\nabla\bX)_{ij}
	+ (\nabla\bY\partial_t\nabla\bX)_{ij}
	+ \sum_{k} \Gamma_{kj}^i\partial_t\bY_k.
	\end{equation*}
	Multiplying this equality by $\bU_j\bpsi_i$ and  summing over $ij$, we get
	\begin{equation*}
	I = -(II+III)
	= -\mathcal{M}\bU\cdot\bpsi
	+ \nabla\bU^T\bpsi\cdot\partial_t\bY.
	\end{equation*}
	Finally, we get
	\begin{align*}
	\bu\cdot\partial_t\bphi
	&= \bU\cdot\partial_t\bpsi
	-\mathcal{M}\bU\cdot\bpsi
	+ \nabla\bU^T\bpsi\cdot\partial_t\bY
	+ \nabla\bpsi^T\bU\cdot\partial_t\bY
	\notag\\
	&= \bU\cdot\partial_t\bpsi
	-\mathcal{M}\bU\cdot\bpsi
	+ \nabla(\bU\cdot\bpsi)\cdot\partial_t\bY.
	\end{align*}
	
\end{proof}

\noindent
\textbf{Convective term}.
\begin{equation*}
\bu\otimes\bu:\nabla\bphi
= \bU\otimes\bU:\nabla\bpsi^T
- \tilde{\mathcal{N}}\bU\cdot\bpsi
\end{equation*}
\begin{proof}
	We derive the convective term by using a known identity
	\begin{equation*}\label{Wu2t3}
	\bu\otimes\bu:\nabla\bphi
	= \divg((\bu\cdot\bphi)\bu)-(\bu\cdot\nabla)\bu\cdot\bphi.
	\end{equation*}
	It is easy to prove (see \cite{inoue1977existence}) that
	$$
	\nabla\bY(\bu\cdot\nabla)\bu
	=\mathcal{N}\bU,
	$$
	which implies
	\begin{equation*}\label{Wu2t4}
	(\bu\cdot\nabla)\bu\cdot\bphi
	=(\bu\cdot\nabla)\bu\cdot\nabla\bY^T\bpsi
	=\nabla\bY(\bu\cdot\nabla)\bu\cdot\bpsi
	=\mathcal{N}\bU\cdot\bpsi.
	\end{equation*}
	On the other side we conclude
	\begin{align*}\label{Wu2t5}
	\divg((\bu\cdot\bphi)\bu)
	&= \nabla(\bu\cdot\bphi)\cdot\bu
	= \nabla_x(\nabla\bX\,\bU\cdot\nabla\bY^T\bpsi)\cdot\nabla\bX\,\bU
	= \nabla_x(\bU\cdot\bpsi)\cdot\nabla\bX\,\bU
	\notag\\
	&= (\nabla\bY^T\nabla\bU^T\bpsi + \nabla\bY^T\nabla\bpsi^T\bU)\cdot\nabla\bX\,\bU
	= (\nabla\bU^T\bpsi + \nabla\bpsi^T\bU)\cdot\bU
	\notag\\
	&=\nabla(\bU\cdot\bpsi)\cdot\bU
	\notag\\
	&= \bU\cdot\nabla\bU\cdot\bpsi + \bU\otimes\bU:\nabla\psi^T.
	\end{align*}
	Therefore,
	\begin{equation*}
	\bu\otimes\bu:\nabla\bphi
	= \bU\cdot\nabla\bU\cdot\bpsi
	+ \bU\otimes\bU:\nabla\bpsi^T
	- \mathcal{N}\bU\cdot\bpsi
	= \bU\otimes\bU:\nabla\bpsi^T
	- \tilde{\mathcal{N}}\bU\cdot\bpsi.
	\end{equation*}
\end{proof}

\noindent
\textbf{Diffusive term}.
\begin{equation*}
\int_{\Omega_F(\tau)} 2\,\D\bu:\D\bphi
= \left\langle \mathcal{L}\bU, \bpsi\right\rangle,
\end{equation*}
where
\begin{multline*}
\left\langle \mathcal{L}\bU, \bpsi\right\rangle
= \int_{\Omega_F(\tau)} \big(\sum_{ijk}
(g^{jk}\partial_j\bU_{i}\partial_k\bpsi_i
+ g^{jk}\partial_k\bU_i\partial_i\bpsi_j)
- \sum_{ijkl} (g^{kl}+\partial_l\bY_k)\Gamma_{li}^j\partial_k\bU_{i}\bpsi_j\\
+ \sum_{ijkl} (g^{kl}\Gamma_{li}^j\bU_{i}\partial_k\bpsi_j
+ g^{jl}\Gamma_{li}^k\bU_i\partial_k\bpsi_j)
- \sum_{ijklm} (g^{kl}+\partial_k\bY_l)\Gamma_{li}^m\Gamma_{km}^j\bU_{i}\bpsi_j
\big).
\end{multline*}
\begin{proof}
	We have
	$$
	2\D\bu:\D\bphi
	= 2\D\bu:\nabla\bphi
	= \nabla\bu:\nabla\bphi
	+ \nabla\bu^T:\nabla\bphi
	$$
	For the first term we get
	\begin{align}
	\nabla\bu:\nabla\bphi
	&= \sum_{ij} \partial_j\bu_{i}\partial_j\bpsi_i
	= \sum_{ij} \frac{d}{d\bx_j}(\nabla\bX\bU)_i\frac{d}{d\bx_j}(\nabla\bY^T\bpsi)_i
	\notag\\
	&= \sum_{ijkl} \frac{d}{d\bx_j}(\partial_k\bX_i\bU_k)\frac{d}{d\bx_j}(\partial_i\bY_l\bpsi_l)
	\notag\\
	&= \sum_{ijklmp} (\partial_m\partial_k\bX_i\bU_k + \partial_k\bX_i\partial_m\bU_k)\underline{\partial_j\bY_m
		\partial_j\bY_p}(\frac{d}{d\by_p}\partial_i\bY_l\bpsi_l + \partial_i\bY_l\partial_p\bpsi_l)
	\notag\\
	&= \underbrace{\sum_{iklmp} g^{mp}\partial_m\partial_k\bX_i\partial_{\by_p}\partial_i\bY_l\bU_k\bpsi_l}_{I}
	+ \underbrace{\sum_{iklmp} g^{mp}\partial_m\partial_k\bX_i\partial_i\bY_l\bU_k\partial_p\bpsi_l}_{II}
	\notag\\
	&\quad +\underbrace{\sum_{iklmp} g^{mp}\partial_k\bX_i\partial_{\by_p}\partial_i\bY_l\partial_m\bU_k\bpsi_l}_{III}
	+ \underbrace{\sum_{iklmp} g^{mp}\partial_k\bX_i\partial_i\bY_l\partial_m\bU_k\partial_p\bpsi_l}_{IV},
	\notag
	\end{align}
	\begin{align*}
	II &= \sum_{klmp} g^{mp}(\sum_{i}\partial_m\partial_k\bX_i\partial_i\bY_l)\bU_k\partial_p\bpsi_l
	= \sum_{klmp} g^{mp}\Gamma_{mk}^l\bU_k\partial_p\bpsi_l,
	\end{align*}
	\begin{align*}
	IV &= \sum_{klmp} g^{mp}(\sum_{i}\partial_k\bX_i\partial_i\bY_l)\partial_m\bU_k\partial_p\bpsi_l
	= \sum_{kmp} g^{mp}\partial_m\bU_k\partial_p\bpsi_k,
	\end{align*}
	\begin{align*}
	III &= \sum_{iklmp} g^{mp}(\partial_p(\partial_k\bX_i\partial_i\bY_l) - \partial_p\partial_k\bX_i\partial_i\bY_l)\partial_m\bU_k\bpsi_l
	= -\sum_{klmp} g^{mp}\Gamma_{pk}^l\partial_m\bU_k\bpsi_l,
	\end{align*}
	\begin{align*}
	I &= \sum_{iklmp} g^{mp}(\partial_p(\partial_m\partial_k\bX_i\partial_i\bY_l) - \partial_p\partial_m\partial_k\bX_i\partial_i\bY_l)\bU_k\bpsi_l
	\notag\\
	&= \underbrace{\sum_{klmp} g^{mp}\partial_p(\Gamma_{mk}^l)\bU_k\bpsi_l}_{V}
	- \underbrace{\sum_{iklmp} g^{mp} \partial_p\partial_m\partial_k\bX_i\partial_i\bY_l\bU_k\bpsi_l}_{VI},
	\notag
	\end{align*}
	\begin{align*}
	VI &= (\cite{inoue1977existence})
	= \sum_{iklmp} g^{mp} \partial_p(\sum_{q}\Gamma_{mk}^q\partial_q\bX_i)\partial_i\bY_l\bU_k\bpsi_l
	\notag\\
	&= \sum_{iklmpq} g^{mp} \partial_p\Gamma_{mk}^q\underline{\partial_q\bX_i\partial_i\bY_l}\bU_k\bpsi_l
	+ \sum_{iklmpq} g^{mp} \Gamma_{mk}^q\underline{\partial_p\partial_q\bX_i\partial_i\bY_l}\bU_k\bpsi_l
	\notag\\
	&= \underbrace{\sum_{klmp} g^{mp} \partial_p\Gamma_{mk}^l\bU_k\bpsi_l}_{= V}
	+ \sum_{klmpq} g^{mp} \Gamma_{mk}^q\Gamma_{pq}^l\bU_k\bpsi_l.
	\notag
	\end{align*}
	It follows
	\begin{align*}
	I &= - \sum_{klmpq} g^{mp} \Gamma_{mk}^q\Gamma_{pq}^l\bU_k\bpsi_l.
	\end{align*}
	Hence,
	\begin{multline}
	\nabla\bu:\nabla\bphi
	= \sum_{ijk} g^{jk}\partial_j\bU_{i}\partial_k\bpsi_i
	- \sum_{ijkl} g^{kl}\Gamma_{li}^j\partial_k\bU_{i}\bpsi_j
	+ \sum_{ijkl} g^{kl}\Gamma_{li}^j\bU_{i}\partial_k\bpsi_j
	- \sum_{ijklm} g^{kl}\Gamma_{li}^m\Gamma_{km}^j\bU_{i}\bpsi_j.
	\notag
	\end{multline}
	
	Similarly, we get
	\begin{align}
	\nabla\bu^T:\nabla\bphi
	&= \sum_{ij} \partial_i\bu_{j}\partial_j\bpsi_i
	= \sum_{ij} \frac{d}{d\bx_i}(\nabla\bX\bU)_j\frac{d}{d\bx_j}(\nabla\bY^T\bpsi)_i
	\notag\\
	&= \sum_{ijkl} \frac{d}{d\bx_i}(\partial_k\bX_j\bU_k)\frac{d}{d\bx_j}(\partial_i\bY_l\bpsi_l)
	\notag\\
	&= \sum_{ijklmp} (
	\partial_m\partial_k\bX_j\bU_k 
	+ \partial_k\bX_j\partial_m\bU_k
	)\partial_i\bY_m
	(
	\partial_{j}\partial_i\bY_l\bpsi_l 
	+ \partial_i\bY_l\partial_p\bpsi_l\partial_{j}\bY_p)
	\notag\\
	&= \underbrace{
	\sum_{iklmp} 
	\partial_i\bY_m\partial_m\partial_k\bX_j
	\partial_{j}\partial_i\bY_l\bU_k\bpsi_l 
	}_{I'}
	+ \underbrace{
	\sum_{iklmp} \partial_i\bY_m\partial_i\bY_l
	\partial_m\partial_k\bX_j\partial_{j}\bY_p\bU_k 
	\partial_p\bpsi_l
	}_{II'}
	\notag\\
	&\quad +\underbrace{
	\sum_{iklmp} 
	\partial_i\bY_m\partial_k\bX_j
	\partial_{j}\partial_i\bY_l
	\partial_m\bU_k\bpsi_l
	}_{III'}
	+ \underbrace{
	\sum_{iklmp} 
	\partial_i\bY_m
	\partial_i\bY_l\partial_k\bX_j\partial_{j}\bY_p
	\partial_m\bU_k\partial_p\bpsi_l
	}_{IV'},
	\notag
	\end{align}
	\begin{align*}
	II' &= \sum_{klmp} g^{ml}\Gamma_{mk}^p\bU_k\partial_p\bpsi_l,
	\end{align*}
	\begin{align*}
	IV' &= \sum_{klmp} g^{ml}\delta_{pk}
	\partial_m\bU_k\partial_p\bpsi_l
	= \sum_{klm} g^{ml}\partial_m\bU_k\partial_k\bpsi_l,
	\end{align*}
	\begin{align*}
	III' &= \sum_{iklmp}
	\partial_i\bY_m
	\left(\partial_i
	(\partial_k\bX_j
	\partial_{j}\bY_l)
	- \partial_i\partial_k\bX_j
	\partial_{j}\bY_l
	\right)
	\partial_m\bU_k\bpsi_l
	= -\sum_{iklm} \partial_i\bY_m
	\Gamma_{ik}^l
	\partial_m\bU_k\bpsi_l,
	\end{align*}
	\begin{align*}
	I' &= \sum_{ijklmp} 
	\partial_i\bY_m
	\left(
	\partial_i(\partial_m\partial_k\bX_j
	\partial_{j}\bY_l)
	- \partial_i\partial_m\partial_k\bX_j
	\partial_{j}\bY_l
	\right)
	\bU_k\bpsi_l\\
	&= \underbrace{\sum_{iklmp} \partial_i\bY_m
	\partial_i\Gamma_{mk}^l\bU_k\bpsi_l}_{V'}
	- \underbrace{\sum_{ijklmp} \partial_i\bY_m \partial_i\partial_m\partial_k\bX_j\partial_j\bY_l\bU_k\bpsi_l}_{VI'}, 
	\end{align*}
	\begin{align*}
	VI' &= (\cite{inoue1977existence})
	=\sum_{ijklmp} \partial_i\bY_m
	\partial_i(\sum_{q}\Gamma_{mk}^q\partial_q\bX_j)
	\partial_j\bY_l\bU_k\bpsi_l
	\\
	&= \sum_{ijklmpq} \partial_i\bY_m \partial_i\Gamma_{mk}^q
	\underline{\partial_q\bX_j\partial_j\bY_l}
	\bU_k\bpsi_l
	+ \sum_{ijklmpq} \partial_i\bY_m \Gamma_{mk}^q
	\underline{\partial_i\partial_q\bX_j\partial_j\bY_l}
	\bU_k\bpsi_l
	\\
	&= \underbrace{\sum_{iklmp} \partial_i\bY_m \partial_i\Gamma_{mk}^l
	\bU_k\bpsi_l}_{= V'}
	+ \sum_{iklmpq} \partial_i\bY_m \Gamma_{mk}^q\Gamma_{iq}^l
	\bU_k\bpsi_l
	\end{align*}
	It follows
	\begin{align*}
	I' &= - \sum_{iklmpq} \partial_i\bY_m \Gamma_{mk}^q\Gamma_{iq}^l
	\bU_k\bpsi_l.
	\end{align*}
	Therefore,
	\begin{multline*}
	\nabla\bu^T:\nabla\bphi
	= \sum_{ijk} g^{jk}\partial_k\bU_i\partial_i\bpsi_j
	-\sum_{iklm} \partial_l\bY_k
	\Gamma_{il}^j \partial_k\bU_i\bpsi_j
	+ \sum_{ijkl} g^{jl}\Gamma_{li}^k\bU_i\partial_k\bpsi_j
	- \sum_{ijklm} \partial_k\bY_l \Gamma_{li}^m\Gamma_{km}^j
	\bU_i\bpsi_j
	\end{multline*}
	
	Finally, we get
	\begin{multline}
	\left\langle \mathcal{L}\bU, \bpsi\right\rangle
	:=\int_{\Omega_F(\tau)} 2\,\D\bu:\D\bphi
	= \int_{\Omega_F(\tau)}
	\left(
	\nabla\bu:\nabla\bphi
	+ \nabla\bu^T:\nabla\bphi
	\right) \\
	= 
	\frac{1}{2}\int_{\Omega_F(\tau)} \big(\sum_{ijk}
	(g^{jk}\partial_j\bU_{i}\partial_k\bpsi_i
	+ g^{jk}\partial_k\bU_i\partial_i\bpsi_j)
	- \sum_{ijkl} (g^{kl}+\partial_l\bY_k)\Gamma_{li}^j\partial_k\bU_{i}\bpsi_j\\
	+ \sum_{ijkl} (g^{kl}\Gamma_{li}^j\bU_{i}\partial_k\bpsi_j
	+ g^{jl}\Gamma_{li}^k\bU_i\partial_k\bpsi_j)
	- \sum_{ijklm} (g^{kl}+\partial_k\bY_l)\Gamma_{li}^m\Gamma_{km}^j\bU_{i}\bpsi_j
	\big).
	\notag
	\end{multline}
	
\end{proof}

\section{Conflict of Interest Statement}

On behalf of all authors, the corresponding author states that there is no conflict of interest.

\section{Acknowledgements}
The authors would like to thank the anonymous referee for her/his comments that helped us to
improve the manuscript.

\bibliographystyle{siamplain}

\end{document}